\documentclass[a4paper,11pt]{amsart}

\usepackage[left=2.7cm,right=2.7cm,top=3.5cm,bottom=3cm]{geometry}
\usepackage{amsthm,amssymb,amsmath,amsfonts,mathrsfs,amscd,dsfont,verbatim}
\usepackage[utf8]{inputenc}
\usepackage[T1]{fontenc}
\usepackage[all,cmtip]{xy}
\usepackage{latexsym}
\usepackage{longtable}
\usepackage{mathtools}
\usepackage[pagebackref]{hyperref}
\usepackage{graphicx}

\numberwithin{equation}{section}

\newfont{\cyr}{wncyr10 scaled 1100}
\newfont{\cyrr}{wncyr9 scaled 1000}

\theoremstyle{plain}
\newtheorem{theorem}{Theorem}[section]
\newtheorem{proposition}[theorem]{Proposition}
\newtheorem{lemma}[theorem]{Lemma}
\newtheorem{corollary}[theorem]{Corollary}

\theoremstyle{definition}
\newtheorem{definition}[theorem]{Definition}
\newtheorem{assumption}[theorem]{Assumption}

\theoremstyle{remark}
\newtheorem{remark}[theorem]{Remark}

\newcommand{\Q}{\mathbb Q}

\newcommand{\Z}{\mathbb Z}
\newcommand{\R}{\mathbb R}
\newcommand{\C}{\mathbb C}

\newcommand{\F}{\mathbb F}

\newcommand{\G}{\mathbb G}
\newcommand{\A}{\mathbb A}
\newcommand{\I}{\mathbb{I}}

\newcommand{\J}{\mathbb{J}}

\newcommand{\defeq}{\vcentcolon=}

\newcommand{\pwseries}[1]{[\![ #1]\!]}
\newcommand{\p}{\mathfrak{p}}

\DeclareMathOperator{\Spec}{Spec}
\DeclareMathOperator{\Pic}{Pic}
\DeclareMathOperator{\End}{End}

\DeclareMathOperator{\Frob}{Frob}

\DeclareMathOperator{\Hom}{Hom}

\DeclareMathOperator{\Gal}{Gal}
\DeclareMathOperator{\GL}{GL}

\DeclareMathOperator{\M}{M}

\DeclareMathOperator{\Fil}{Fil}

\DeclareMathOperator{\rank}{rank}

\DeclareMathOperator{\Ta}{Ta}

\DeclareMathOperator{\et}{\text{\'et}}

\newcommand{\res}{\mathrm{res}}
\newcommand{\cores}{\mathrm{cores}}
\newcommand{\cor}{\mathrm{cor}}
\newcommand{\cyc}{{\mathrm{cyc}}}

\newcommand{\ord}{\mathrm{ord}}

\newcommand{\rig}{\mathrm{rig}}
\newcommand{\rec}{\mathrm{rec}}
\newcommand{\dR}{\mathrm{dR}}
\newcommand{\unr}{\mathrm{unr}}
\newcommand{\ES}{\mathrm{ES}}

\newcommand{\ST}{\mathrm{ST}}

\usepackage[usenames]{color}
\definecolor{Indigo}{rgb}{0.2,0.1,0.7}
\definecolor{Violet}{rgb}{0.5,0.1,0.7}
\definecolor{White}{rgb}{1,1,1}
\definecolor{Green}{rgb}{0.1,0.9,0.2}

\usepackage{tikz-cd}

\newcommand{\longmono}{\mbox{\;$\lhook\joinrel\longrightarrow$\;}}

\newcommand{\mat}[4]{\left(\begin{array}{cc}#1&#2\\#3&#4\end{array}\right)}
\newcommand{\smallmat}[4]{\bigl(\begin{smallmatrix}#1&#2\\#3&#4\end{smallmatrix}\bigr)}

\newcommand{\invlim}{\mathop{\varprojlim}\limits}

\begin{document}

\include{thebibliography}

\title[Quaternionic families of Heegner points and $p$-adic $L$-functions]{Quaternionic families of Heegner points and $p$-adic $L$-functions}
\today
\date{}
\author{M. Longo, P. Magrone, E. R. Walchek}

\thanks{The authors thank Stefano Vigni, Andrea Mori and Francesc Castella  for several discussions 
about the topics of this paper. They also thank David Loeffler for his interest in this work and his encouragement. M.L. is partially supported by PRIN 2022 ``The arithmetic of motives and $L$-functions'' and by the GNSAGA group of INdAM}

\begin{abstract}
Following up a previous article of the authors which studies the interpolation of certain anticyclotomic $p$-adic $L$-functions associated to quaternionic modular forms in a Hida family, we extend the work of F. Castella on the interpolation and specialization of big Heegner points to the quaternionic setting. We prove an explicit reciprocity law relating the big $p$-adic $L$-function to the big Heegner points in this quaternionic setting. We deduce arithmetic results for 
the Selmer group of Hida's Big Galois representation and its specialization over a quadratic imaginary field satisfying a relaxed Heegner hypothesis. 
\end{abstract}

\address{Dipartimento di Matematica, Universit\`a di Padova, Via Trieste 63, 35121 Padova, Italy}
\email{mlongo@math.unipd.it}

\subjclass[2010]{}


\maketitle


\section{Introduction}\label{Intro}

This note arises with the aim of generalizing to the setting of indefinite quaternion algebras studied in \cite{LV-MM} some of the results contained in the papers \cite{Castella-MathAnn, Castella} by Castella; we also deduce from this extension some arithmetic applications to Selmer groups of Big Galois representations attached to a Hida family and to Bloch-Kato Selmer groups of its specializations 
over a quadratic imaginary field which satisfies a relaxed Heegner hypothesis. 

The main result of these two papers \cite{Castella-MathAnn, Castella} by Castella is the description of the specialization of big Heegner points introduced by Howard \cite{Howard-Inv} 
at certain arithmetic morphisms of a given Hida family in terms of generalized Heegner cycles. The approach of Castella, following previous works by Darmon--Rotger \cite{DR}, is to use density of weight $2$ primes in Hida family, and the explicit description of big Heegner points as limits of Heegner points (which therefore, by construction, are directly related to the specializations at weight $2$ of big Heegner points). The relation between higher weight specializations of big Heegner points and generalized Heegner cycles is obtained using families of $p$-adic $L$-functions as a bridge between the two, and the comparison is made possible by the density of weight $2$ specializations. 

In the quaternionic setting, Howard's big Heegner points have been introduced by Fouquet \cite{Fouquet} (even over totally real number fields) and one of the authors of this paper in collaboration with Vigni \cite{LV-MM}. The natural question is to what extent the techniques and approach in \cite{Castella-MathAnn, Castella, DR} can be adapted to the quaternionic setting. It should be noticed that the $p$-adic $L$-function 
which appears as a bridge in the approach of \cite{Castella} is a $p$-adic variation of the $p$-adic $L$-function constructed in \cite{CH} following the approach of Brako\v{c}evi\'{c} \cite{Brako} (and Bertolini--Darmon--Prasanna \cite{BDP}). This approach makes use of Serre--Tate expansions of modular forms, and a part of it can be adapted to 
the quaternionic setting: this portion of the work has been done in \cite{LMW}. The first goal of this work is to accomplish the comparison between higher weight specialization of big Heegner points in the quaternionic setting and generalized Heegner cycles, using again the $p$-adic family of $p$-adic $L$-functions and weight $2$ specializations as a bridge. 

Our main result proves an equality (up to units, \emph{cf.} Theorem \ref{thmL=L}) 
\begin{equation}\label{main}
\mathscr{L}_{\I,\boldsymbol{\xi}}^\mathrm{an}=\mathscr{L}_{\I,\boldsymbol{\xi}}^\mathrm{alg}\end{equation} of two $p$-adic $L$-functions attached to a quaternionic Hida family $\I$ and a $p$-adic family of Hecke characters $\boldsymbol{\xi}$ of a quadratic imaginary extension $K/\Q$.
We explain in the follwing lines (a part of) the terminology used and the 
nature of the objects in \eqref{main}; the reader is also referred to Section \ref{assumptions} for a more detailed discussion on the hypothesis of this paper and the terminology.
We start by saying that $\mathbb{I}$ is a primitive branch of a Hida family passing through a fixed $p$-stabilized newform $f\in S_{k_0}(\Gamma_0(Np))$ of trivial character and even weight $k_0\equiv 2\mod{2(p-1)}$, and we assume throughout that the 
restriction of the residual $p$-adic representation attached to $f$ to a decomposition group at $p$ is  $p$-distinguished and irreducible. The discriminant $-D_K$ of $K$ is assumed to be coprime with $Np$, $p$ is assumed to be split in $K$, and $N$ is assumed to factor as $N=N^+N^-$ with $(N^+,N^-)=1$, $\ell\mid N^+$ (respectively, $\ell\mid N^-$) if and only if $\ell$ is split (respectively, inert) in $K$ with $N^+\geq 4$ and $N^-$ a product of an even number of distinct primes. Since the case $N^-=1$ is well-known (\cite{Castella, Castella-MathAnn}) 
we assume that $N^->1$ and we say that $K$ satisfies a \emph{generalized Heegner hypothesis} in this case. 
Equality \eqref{main} holds in 
$\widetilde{\I}[[\widetilde{\Gamma}_\infty]]$ where $\widetilde{\I}=\I\otimes\Z_p^\unr$ and 
$\widetilde{\Gamma}_\infty$ is the Galois group of the union of the ring class fields of $K$ of conductors $cp^n$ for all integers $n\geq 1$, where $c\geq 1$ is a fixed integer coprime with $D_KNp$, 
which depends on $\boldsymbol{\xi}$. 
Here $\mathscr{L}_{\I,\boldsymbol{\xi}}^\mathrm{an}$ is the family of $p$-adic $L$-functions constructed in \cite{LMW} and alluded to before, while $\mathscr{L}_{\I,\boldsymbol{\xi}}^\mathrm{alg}$ is the $p$-adic $L$-function which arises as the evaluation of a big Perrin-Riou logarithmic map at the quaternionic big Heegner point of tame conductor $c$ relative to the imaginary quadratic field $K$.
As mentioned before, Equality \eqref{main} is obtained by an explicit comparison of the weight $2$ specializations of the two sides. 

We list some arithmetic application of \eqref{main} (see \S\ref{secArithApp} for a more detailed discussion). We suppose that the generic root number of $\I$ is $+1$ 
(see \S\ref{secArithApp} and the references therein).  
Let $\mathfrak{Z}_c$ be the quaternionic Big Heegner point of conductor $c$ introduced in \cite{LV-MM}, and whose construction is recalled in \S\ref{sec8.2}. As a consequence 
of \eqref{main}, we obtain the following result (see Corollary \ref{coro})
\begin{equation}\label{THMINTRO1}
\text{$\mathfrak{Z}_c$ is not $\I$-torsion}.\end{equation}
This proves \cite[Conjecture 10.3]{LV-MM}. As a consequence of \eqref{THMINTRO1}, we obtain 
some results on Selmer groups of Hida Big Galois representation attached to $\I$ and its specializations. 
More precisely, let $\mathbf{T}^\dagger$ be the self-dual twist of Hida Big Galois representation 
attached to $\I$, which is a free $\I$-module of rank $2$ equipped with a $\Gal(\overline{\Q}/\Q)$-action such that for each arithmetic morphism $\nu:\I\rightarrow F_\nu$ the specialization $\mathbf{T}^\dagger_\nu=\mathbf{T}^\dagger\otimes_{\I}F_\nu$ is isomorphic to the self-dual twist of the representation attached to the specialization $f_\nu$ of the Hida family $\I$ at $\nu$
(the reader is referred to \S\ref{BigGaloisRep} for details; here $F_\nu$ is a finite extension of $\Q_p$ containing all the Fourier coefficients of the normalized eigenform $f_\nu$). 
Let $\widetilde{H}^1_f(K,\mathbf{T}^\dagger)$ be Nekov\'{a}\v{r}'s extended Selmer group attached to $\mathbf{T}^\dagger$. Then as a consequence of the fact that $\mathfrak{Z}_1$ is not $\I$-torsion we see that 
(see Corollary \ref{coro1}) 
\begin{equation}\label{THMINTRO2}
\rank_\I\Bigl(\widetilde{H}^1_f(K,\mathbf{T}^\dagger)\bigr)=1.\end{equation}
Finally, for an arithmetic morphis $\nu$ as before, let $H^1_f(K,\mathbf{T}_\nu^\dagger)$ denote 
the Bloch-Kato Selmer group of the self-dual representation $\mathbf{T}_\nu^\dagger$. Then  
for all arithmetic $\nu$ except possible a finite number of them
(see again Corollary \ref{coro1})
\begin{equation}\label{THMINTRO3}
\dim_{F_\nu}\Bigl(H^1_f(K,\mathbf{T}^\dagger_\nu)\bigr)=1.\end{equation}
The results \eqref{THMINTRO1}, \eqref{THMINTRO2} and \eqref{THMINTRO3} are 
the main arithmetic contributions of this paper; we also remark that their novelty with respect to 
the existing literature (\emph{e.g.} \cite{Castella}, \cite{JLZ}, \cite{BL-Coleman}) is that here $K$ satisfied a \emph{relaxed} Heegner hypothesis
in which we allow an \emph{even} (positive) number of primes dividing $N$ to be inert in $K$. 

We finally make a comment on the relation between generalized Heegner cycles and specialization of Big Heegner points, in the current quaternionic setting (see \S\ref{secspecialization} for 
a more detailed discussion). Since $\mathscr{L}_{\I,\boldsymbol{\xi}}^\mathrm{an}$ specializes at higher weights to generalized Heegner cycles by a result of one of the authors of this paper in \cite{Magrone}, this provides a relation between the specialization at higher weights of the Perrin-Riou big logarithm evaluated at the relevant big Heegner point (\emph{cf.} Theorem \ref{teospec}).  
However, at the moment our result is not completely satisfactory, for the following reason. The construction of  
$\mathscr{L}_{\I,\boldsymbol{\xi}}^\mathrm{an}$ performed in \cite{LMW} is based on a generalization of Hida--Ohta theory \cite{Ohta-ES, OhtaC, OhtaMA} to the quaternionic setting, which provides a canonical pairing in the Hodge--Tate filtration of the inverse limit of \'etale cohomology group of the modular curves of $p$-power level.  In the $\GL_2$-case, a result of Kings--Loeffler--Zerbes \cite[Theorem 10.1]{KLZ.ERL} allows to relate the higher weight specialization of Ohta's pairing to a completely different pairing, arising from Kuga--Sato varieties over modular curves: to the best of our knowledge, this result is not available in the quaternionic setting (the missing ingredient is a suitable generalization of Beilinson--Kato elements used in \cite{KLZ.ERL} to obtain the aforementioned result). 

\section{Notation and assumptions}\label{assumptions}

We fix throughout the text an embedding $\overline{\Q}\hookrightarrow\C$ and embeddings $\overline{\Q}\hookrightarrow\bar\Q_\ell$ for each prime number $\ell$.  

Fix a positive integer $N$ and a prime number $p\nmid N$. 
Let $K/\Q$ be a quadratic imaginary field of discriminant $-D_K$ prime to $Np$, and factor 
$N=N^+N^-$, where $N^+$ is divisible only by primes which are split in $K$, and $N^-$ is a square-free integer, divisible only by primes which are inert in $K$. Assume that $N^+\geq 4$ and $p$ is split in $K$; write $p=\mathfrak{p}\bar{\mathfrak{p}}$, where $\mathfrak{p}$ is the prime ideal corresponding to the embedding $\overline{\Q}\hookrightarrow\overline{\Q}_p$. 

\subsection{Quaternion algebras}\label{subsec.QuatAlg}
Let $B$ be the quaternion algebra of discriminant $N^-$, defined over $\Q$, let $\mathcal{O}_B$ be the maximal order of $B$. 
For primes $\ell\nmid N^-$, we fix isomorphisms $i_\ell\colon B_\ell\defeq B\otimes_\Q\Q_\ell\simeq\M_2(\Q_\ell)$. 
Take the sequence of Eichler orders $\mathcal{O}_B\supseteq R_0\supseteq R_1\supseteq\dots$ such that each $R_m$ has level $N^+p^m$ and the image of $R_m$ is equal to the order of upper triangular matrices modulo $\ell^{\ord_\ell(N^+p^m)}$. Fix also $i_\infty\colon B_\infty\defeq B\otimes_\Q\R\simeq\M_2(\R)$ coming from the splitting at $\infty$. Finally, since $K$ splits $B$, we may fix an embedding of $\Q$-algebras 
$\iota_K\colon K\hookrightarrow B$; we will sometimes write $x$ for $\iota_K(x)$ for $x\in K$ when the context is clear. 

From Section \ref{secBHP} on, to obtain a clear description of CM points, it will be convenient to choose the isomorphisms $i_\ell$ and $i_\infty$ (and, consequently, the Eichler orders $R_m$) as follows. Take the $\Q$-basis $\{1,\theta\}$ of $K$, where $\theta=\frac{D'+\sqrt{-D_K}}{2}$, being $D'=D_K$ if $2\nmid D_K$ and $D'=D_K/2$ if $2\mid D_K$. For each place $v\mid N^+p\infty$ of $\Q$, we may assume the isomorphism $i_v\colon B_v\cong \M_2(\Q_v)$
to satisfy
\[i_v(\theta)=\mat {\mathrm{T}_{K/\Q}(\theta)}{-\mathrm{N}_{K/\Q}(\theta)}{1}{0}.\]

Define the orthogonal idempotents $e$ and $\bar{e}$ in $K\otimes_\Q K$ to be: 
\[e=\frac{1\otimes{\theta}-\theta\otimes 1}{(\theta-\bar{\theta})\otimes 1}\hspace{15pt}\text{and}\hspace{15pt}\bar{e}=\frac{\theta\otimes 1-1\otimes\bar\theta}{(\theta-\bar{\theta})\otimes1}.\]

A simple computation shows that $e+\bar{e}=1$. 
Let $\ell\mid N^+p$ be a prime number. Then $\ell$ splits in $K$ as $\ell=\mathfrak{l}\bar{\mathfrak{l}}$, where 
$\mathfrak{l}$ is the prime ideal corresponding to the chosen embedding $\overline{\Q}\hookrightarrow\overline{\Q}_\ell$, so $K_\ell=K\otimes_\Q\Q_\ell$ splits as the direct sum $\Q_\ell e\oplus \Q_\ell \bar{e}$ of two copies of $\Q_\ell$. We have a canonical map 
\[j_\ell\colon K\otimes_\Q K\longmono K\otimes_\Q\Q_\ell\longmono B_\ell\overset{i_\ell}\longrightarrow\M_2(\Q_\ell)\] and one may verify that $j_\ell(e)=\smallmat 1000$ and $j_\ell(\bar{e})=\smallmat 0001$.  

Denote by $i\colon K\hookrightarrow \M_2(\Q)$ the $\Q$-linear map which takes $\theta$ to $\smallmat{\mathrm{T}_{K/\Q}(\theta)}{-\mathrm{N}_{K/\Q}(\theta)}{1}{0}$. Then we have $i=i_\infty\circ\iota_K$ and we obtain a map \[j\colon K\otimes_\Q K\longmono\M_2(K)\] defined by $j(x\otimes y)=i(x)y$, and one verifies again that $j(e)=\smallmat 1000$ and $j(\bar{e})=\smallmat 0001$. 

\subsection{Shimura curves} 
Let $X_m$ denote the Shimura curve of $V_1(N^+p^m)$-level structure attached to the indefinite quaternion algebra $B$. By the action of $B^\times$ via fractional linear transformations through the embedding $i_\infty\colon B^\times\hookrightarrow\GL_2(\R)$ on $\mathcal{H}^\pm=\C\setminus\R$, often identified with $\Hom_\R(\C,B_\infty)$, we have, for any integer $m\geq 0$,
\begin{equation}\label{Xm}
X_m(\C)= B^\times\backslash(\mathcal{H}^\pm\times\widehat{B}^\times)/U_m,\end{equation} 
where $U_m$ is the subgroup of $\widehat{R}_m^\times$ consisting of elements whose $\ell$-component is upper triangular modulo $\ell^{\ord_\ell(N^+p^m)}$ for all primes $\ell\nmid N^-$. We will write $[(x,g)]$ for a point in $X_m(\C)$.
Define $ J_m=\mathrm{Jac}( X_m)$ and 
$\Ta_p=\invlim 
\Ta_p( J_m)$ (the inverse limit is computed with respect to the canonical projection maps $J_{m+1}\rightarrow J_{m}$ for $m\geq 1$).
The $\Z_p$-module $\Ta_p( J_m)$ is equipped with a continuous action of the absolute Galois group of $\Q$ and an action of Hecke operators $T_\ell$ for primes 
$\ell\nmid Np$ and $U_\ell$ for primes $\ell\mid N^+p$ attached to the indefinite quaternion algebra 
$B$ (\cite[\S6.2]{LV-MM}); denote $\mathfrak{h}_m\subseteq\End_{\Z_p}(\Ta_p( J_m))$ the Hecke algebra generated by these operators. Taking the projective limit of these Hecke algebras one defines a big Hecke algebra $\mathfrak{h}_\infty$ acting on $\Ta_p$; we may define a Hida ordinary idempotent $e^\ord$ attached to $U_p$
and set $\mathfrak{h}_\infty^\ord=e^\ord\mathfrak{h}_\infty$. 

\subsection{Moduli spaces} The Shimura curve $X_m$ has a model $\mathcal{X}_m$ over $\Z_p$ which is constructed by means of na\"ive level $V_1(N^+)$ structures and Drinfeld level structures at $p$. For a subring $\mathcal{O}\subseteq\bar\Q_p$, and assuming that $N^+\geq 4$, a $\mathcal{O}$-rational point of $\mathcal{X}_m$ is a quadruplet $(A,\iota,\alpha,\beta)$ where $(A,\iota)$ is a quaternionic multiplication abelian surface, \emph{i.e.} an abelian surface $A\rightarrow\Spec(\mathcal{O})$ equipped with an homomorphism $\iota\colon \mathcal{O}_B\hookrightarrow\End(A)$, $\alpha$ is a level $V_1(N^+)$ structure and $\beta$ is a Drinfeld level structure, \emph{i.e.} a finite flat subgroup scheme of $eA[p^m]$ which is locally free of rank $p^{2m}$ equipped with a choice of generator in the sense of Katz--Mazur. 

We call \emph{test objects over $\mathcal{O}$} sets of the form $T=(A,\iota,\alpha,\beta)$ where $(A,\iota)$ is a QM abelian surface, $\alpha$ is a level $V_1(N^+)$ structure on $A$ and $\beta$ is a Drinfeld level $\Gamma_1(p^m)$ structure on $A$.
A \emph{modular form} $\mathcal{F} \in S_k(\Gamma_m, \mathcal{O})$ is then a rule that assigns to each such test object $T$ over an $\mathcal{O}$-algebra $R$ a differential $\mathcal{F}(T) \in \underline{\omega}_A^{\otimes k}$,
satisfying a base-change compatibility condition. Equivalently, a modular form $\widetilde{\mathcal{F}}$ can be viewed as a rule that assigns to each test object $T = (A, \iota, \alpha, \beta)$ over an $\mathcal{O}$-algebra $R$ and each section $\omega$ of $\underline{\omega}_A^{\otimes k}$, a value $\widetilde{\mathcal{F}}(T, \omega) \in R$ satisfying a base-change compatibility condition and is homogeneous of weight $k$ in $\omega$. These two descriptions are equivalent via the relation $\mathcal{F}(T) = \widetilde{\mathcal{F}}(T, \omega) \cdot \omega$ for any choice of section $\omega$ of $\underline{\omega}_A^{\otimes k}$.

If $k=2$ and $\mathcal{O}=L$ is a field, one can check that this notion coincides with the notion of modular forms as global sections of $H^0(X_{m/L},\Omega^1)$. More generally, for general $k$ we see that $S_k(\Gamma_m,L)$ is the $L$-vector space of global sections of $H^0(X_{m/L},\underline{\omega}_m^{\otimes{k}})$ where $\underline{\omega}_m=e\pi_*\Omega^1_{\mathcal{A}_m}$, $e$ is the idempotent defined in \S\ref{subsec.QuatAlg} and $\pi\colon\mathcal{A}_m\rightarrow\mathcal{X}_m$ is the universal object. 

\subsection{Hida families}
Let $f\in S_{k_0}(\Gamma_0(Np))$ be an elliptic newform \textit{or} the $p$-stabilization of a newform $f^\sharp\in S_{k_0}(\Gamma_0(N))$. Let $\mathcal{O}$ be the valuation ring of a finite unramified extension of $\Q_p$ and suppose that the normalized Fourier expansion $\sum_{n\geq 1}a_nq^n$ of $f$ is contained in $\mathcal{O}[[q]]$. 

\begin{assumption}
We suppose that
\begin{itemize} 
\item $f$ is $p$-ordinary, \emph{i.e.} $a_p\in\mathcal{O}^\times$; 
\item The restriction to a decomposition group at $p$ of the 
residual Galois representation $\bar\rho$ attached to $f$ is irreducible and $p$-distinguished (which means that 
the semisimplification of $\bar\rho$ is the direct sum of two distinct characters); 
\item $k_0\equiv 2\pmod{2(p-1)}$. 
\end{itemize}
\end{assumption}
The modular form $f$ gives rise to a homomorphism $\mathfrak{h}_\infty^\ord\rightarrow \mathcal{O}$ which factors 
through the quotient $\mathfrak{h}_\infty^\ord/\mathfrak{a}$ for a unique minimal ideal $\mathfrak{a}$. Denote $\I$ the integral closure of $\mathfrak{h}_\infty^\ord/\mathfrak{a}$. 
Then $\I$ is a finite flat extension of $\Lambda_\mathcal{O}=\mathcal{O}[[\Gamma]]$ where $\Gamma=1+p\Z_p$. 

An $\mathcal{O}$-linear homomorphism $\nu\colon\I\rightarrow\overline{\Q}_p$ is \emph{arithmetic} if its restriction to $\Lambda=\Z_p[[\Gamma]]$ is of the form $\nu (\gamma)=\psi_\nu (\gamma)\gamma^{k_\nu -2}$ for a finite order character $\psi_\nu \colon\Gamma\rightarrow\overline\Q_p^\times$ (the \textit{wild character} of $\nu$) and an integer $k_\nu \geq 2$ (the \textit{weight} of $\nu$); the pair $(k_\nu ,\psi_\nu )$ is the \emph{signature} of $\nu $. Let $F_\nu$ be the finite extension of $\Q_p$ whose valuation ring is $\mathcal{O}_\nu\defeq\I/\ker(\nu)\I$, which contains $\mathcal{O}$. Then there exists a power series $\mathbf{f}=\sum_{n\ge 1}\mathbf{a}_nq^n\in\I[[q]]$ 
such that for each arithmetic morphism $\nu$ the power series 
$f_\nu\defeq \nu(\mathbf{f})=\sum_{n\geq1}\nu(\mathbf{a}_{n}) q^n\in \mathcal{O}_\nu[[q]]$ is the 
$q$-expansion of a normalized $\GL_2$ modular form of weight $k_\nu$, character $\psi_\nu$ and level $\Gamma_1(Np^{m_\nu})$ where $m_\nu$ is the maximum between $1$ and the conductor of $\psi_\nu$. 
We call $f_\nu$ the \emph{specialization} of $\mathbf{f}$ at $\nu$. 
One may also write $\nu(\mathbf{f})=\mathbf{f}_\nu$ and $\nu(\mathbf{a}_n)=\mathbf{a}_{n,\nu}$ for the specializations at $\nu$. 
We assume throughout that 

\begin{assumption}
    There is $\nu_0$ of weight $k_0$ and trivial wild character such that $f_{\nu_0}=f.$
\end{assumption} 

Note that specializations $f_\nu$ of $\I$ at arithmetic morphisms $\nu$ of weight $k\equiv k_0\pmod{p-1}$ and trivial characters are eigenforms in $S_k(\Gamma_0(N))$, and for $k\neq 2$ are newforms.

\section{Galois representations} 
The goal of this section is to collect the results on families of Galois representations that will be used in this paper. 

\subsection{Modular forms}
We let $V_{{f}_{\nu}}$ denote the $p$-adic Galois $\Gal(\overline{\Q}/\Q)$-representation attached to ${f}_{\nu}$ and determined by the property that the characteristic polynomial of the \emph{arithmetic} Frobenius element $\Frob_\ell$ at a prime ideal $\ell\nmid Np$ is equal to 
the \emph{Hecke polynomial} at $\ell$: \[P_{\ell,\nu}(X)=X^2-\mathbf{a}_{\ell,\nu} X+\psi_\nu(\ell)\ell^{k_\nu-1}.\]  
The contragredient representation $V_{{f}_{\nu}}^*$ of $V_{{f}_{\nu}}$ is then determined by the property that the characteristic polynomial of the  \emph{geometric} Frobenius element $\Frob_\ell^{-1}$ at a prime ideal $\ell\nmid Np$ is equal to $P_{\ell,\nu}(X)$; then we have 
$V_{{f}_{\nu}}^*(k-1)\simeq V_{{f}_{\nu}}$. 

\subsection{Critical characters}
We now introduce critical characters. Let $G_\Q=\Gal(\overline\Q/\Q)$ and $\chi_\cyc\colon G_\Q\rightarrow\Z_p^\times$ be the cyclotomic character. We denote by $\Q_p^\cyc=\Q(\zeta_{p^\infty})=\bigcup_{n\geq 1}\Q(\zeta_{p^n})$ 
the $p$-cyclotomic extension of $\Q$, 
where, for all integers $n\geq 1$, $\zeta_{p^n}$ is a primitive $p^n$-root of unity. Set $G_\infty^\cyc=\Gal(\Q(\zeta_{p^\infty})/\Q)$. The cyclotomic character then induces an isomorphism 
$\chi_\cyc\colon G_\infty^\cyc\overset\sim\rightarrow\Z_p^\times$. 
Factor $\chi_\cyc$ as $\chi_\cyc(x)=\omega(x)\cdot\langle x\rangle$, where 
$\omega\colon G_\Q\rightarrow\mu_{p-1}$ takes values in the group $\mu_{p-1}$ of $(p-1)$-th roots of unity in $\Z_p^\times$ and $\langle\cdot\rangle\colon G_\Q\rightarrow\Gamma$ takes values in the group of principal units.
Let $z\mapsto [z]$ denote the inclusions of group-like elements
$\Z_p^\times\hookrightarrow\Z_p[[\Z_p^\times]]^\times$ and 
$\Gamma\hookrightarrow\Z_p[[\Gamma]]^\times$. 
The critical character $\Theta\colon G_\Q\rightarrow\Lambda^\times$ defined in \cite[Definition 2.1.3]{Howard-Inv} by (recall that $k_0\equiv2\mod{2(p-1)}$) 
\[\Theta(\sigma)=[\langle\sigma\rangle^{1/2}],\] where 
$x\mapsto x^{1/2}$ is the unique square root of $x\in\Gamma$. We still write $\Theta\colon G_\Q\rightarrow\I^\times$ for the composition of $\Theta$ with the canonical inclusion $\Lambda\hookrightarrow\I$.
We also need some variants of this character, that we now introduce. Write $\boldsymbol{\theta}\colon\Q^\times\backslash\A_\Q^\times\rightarrow\I^\times$ for the composition of $\Theta$ with the geometrically normalized reciprocity map $\mathrm{rec}_\Q.$
Since $\Theta$ factors through $G_\infty^\cyc$, 
precomposing it with the inverse of the 
cyclotomic character, we obtain a character of $\Z_p^\times$ which we denote with $\boldsymbol{\vartheta}\colon\Z_p^\times\rightarrow\I^\times$. 
If $\nu \colon\I\rightarrow\overline\Q_p$ is an arithmetic morphism  of signature $(k_\nu ,\psi_\nu )$ we put ${\theta}_\nu =\nu \circ\boldsymbol{\theta}$ 
and ${\vartheta}_\nu =\nu \circ\boldsymbol{\vartheta}$. For any $x\in\Z_p^\times$,
if $k_\nu \equiv k_0\equiv 2\mod{2(p-1)}$, then we have
\[{\vartheta}_\nu (x)
=\psi_\nu ^{1/2}(\langle x\rangle)\cdot  x^{\frac{k_\nu -2}{2}}.
\] Finally, we introduce another variant of critical characters. Denote by $\mathbf{N}_{K/\Q}\colon\A_K^\times\rightarrow\A_\Q^\times$ the adelic norm map, by $\mathbf{N}_\Q\colon\A^\times_\Q\rightarrow\Q^\times$ the adelic absolute value and let 
$\mathbf{N}_K\colon\A^\times_K\rightarrow\Q^\times$ denote the composition $\mathbf{N}_K=\mathbf{N}_\Q\circ\mathbf{N}_{K/\Q}$.
Define 
the 
character ${\boldsymbol{\chi}}\colon K^\times\backslash\widehat{K}^\times\rightarrow\I^\times$  
by ${\boldsymbol{\chi}}=\boldsymbol{\theta}\circ \mathbf{N}^{-1}_{K/\Q}.$
For an arithmetic morphism  $\nu $, define 
$\hat{\chi}_\nu =\nu \circ\boldsymbol{\chi}$.
Since $\chi_\cyc \circ \mathrm{rec}_\Q$ is the $p$-adic avatar of the adelic absolute value $\boldsymbol{\mathrm{N}}_\Q\colon\A_\Q^\times\rightarrow\Q^\times$, we obtain, for $x\in\widehat{K}^\times$ and $k_\nu \equiv k_0\equiv 2\mod{2(p-1)}$, 
\begin{equation}\label{chi_k}
\hat\chi_\nu (x)= 
\psi_\nu ^{-1/2}(\langle \mathbf{N}_K(x)x_\mathfrak{p}  x_{\bar{\mathfrak{p}}}\rangle)\cdot
(\mathbf{N}_K(x)x_\mathfrak{p} x_{\bar{\mathfrak{p}}})^{-\frac{k_\nu -2}{2}}.
\end{equation}

\subsection{Big Galois representations}\label{BigGaloisRep}
Using the critical character introduced before, we now describe the big Galois representation associated with the primitive branch $\I$. 
Consider the ordinary submodule $\Ta_p^\ord=e^{\ord}\Ta_p$ of $\Ta_p$. Since $\I$ is a primitive branch  of the Hida ordinary Hecke algebra $\mathfrak{h}_\infty^\ord$, as a consequence of the Jacquet--Langlands correspondence for $p$-adic families of modular forms (\cite[Proposition 6.4]{LV-MM}, see also  \cite{Chenevier}), one has that
\[\mathbf{T}=\Ta_p^\ord\otimes_{\mathfrak{h}_\infty^\ord}\I\] 
is a free $\I$-module of rank $2$ equipped with a $G_\Q=\Gal(\overline\Q/\Q)$-action, having the following property: $\mathbf{T}$ is  
unramified outside $Np$ and the characteristic polynomial of the 
\emph{arithmetic} Frobenius element $\Frob_\ell$ at a prime ideal $\ell\nmid Np$ is equal to 
\[P_\ell(X)=X^2-T_\ell X+(\chi_\cyc\Theta^2)(\ell).\]  
Thus for each arithmetic character $\nu\colon\I\rightarrow F_\nu$, $\mathbf{T}_\nu=\mathbf{T}\otimes_{\I,\nu}F_\nu$ is isomorphic to $V_{{f}_\nu}$, where the tensor product is taken with respect to $\nu$, composed with the inclusion $\mathcal{O}_\nu\subseteq F_\nu$ as indicated. 
We also set $\mathbf{T}_\nu^*:=V_{{f}_\nu}^*$.

We now describe the ordinary filtration of $\mathbf{T}$. Let $v$ be the place of $\overline\Q$ over $p$ corresponding to the fixed embedding $\overline\Q\hookrightarrow\overline\Q_p$, 
and let $D_v\cong  G_{\Q_p}=\Gal(\overline\Q_p/\Q_p)$ denote
the decomposition group 
of $G_\Q$ at $v$ and $I_v\subseteq D_v$ the inertia subgroup, isomorphic to the inertia subgroup $I_{\Q_p}$ 
of $G_{\Q_p}$ via the isomorphism $D_v\cong  G_{\Q_p}$. 
Let $\eta_v\colon D_v/I_v\rightarrow \I$ be the unramified 
character defined by $\eta_v(\Frob_v)=U_p$, where $\Frob_v$ is an arithmetic Frobenius element of $D_v/I_v$; we identify $\eta_v$ with a character of $G_{\Q_p}/I_{\Q_p}$.   
There is a short exact sequence of $G_{\Q_p}$-modules (depending on the choice of $v$, and thus on $\overline\Q\hookrightarrow\overline\Q_p$)
\begin{equation}\label{filtration}
0\longrightarrow \mathbf{T}^+\longrightarrow  \mathbf{T}\longrightarrow\mathbf{T}^-\longrightarrow 0\end{equation}
such that both $\mathbf{T}^+$ and $\mathbf{T}^-$ are free $\I$-modules of rank $1$, and $G_{\Q_p}$ acts on $ \mathbf{T}^+$ via 
{$\eta_v^{-1}\chi_\cyc\Theta^2$} and acts on the unramified quotient 
$\mathbf{T}^-$ via $\eta_v$; see \cite[\S5.5, Corollary 6.5]{LV-MM} for details. 
As $G_{\Q_p}$-representations we then have an isomorphism 
\[\mathbf{T}\cong \mat{\eta_v^{-1}\chi_\cyc\Theta^2}{*}{0}{\eta_v}.\] 

We now use the critical character to twist $\mathbf{T}$ and get a family of self-dual Galois representations. 
Define the \emph{critical twist} of $\mathbf{T}$ to be the twist 
\[\mathbf{T}^\dagger\defeq\mathbf{T}\otimes\Theta^{-1}\]of $\mathbf{T}$ by the Galois action of $\Theta$ (\cite[\S6.4]{LV-MM}).  
Then $\mathbf{T}_\nu^\dagger=\mathbf{T}^\dagger\otimes_{\I,\nu}F_\nu$ 
is isomorphic to the self-dual twist $V_{{f}_\nu}^\dagger\defeq V_{{f}_\nu}^*(k/2)$ of the $p$-adic Galois representation $V_{{f}_\nu}$.

\subsection{Families of Hecke characters}\label{sec.BigCharacters}
We now construct a certain family of Hecke characters $\boldsymbol{\xi}$ after a fixed Hecke character
\[\lambda\colon K^\times\backslash\A_K^\times\longrightarrow\C^\times\] 
of infinity type $(1,0)$, unramified at $p$, and for which we suppose that the $p$-adic avatar $\hat\lambda\colon K^\times\backslash\widehat{K}^\times\rightarrow{\overline\Q_p^\times}$ of $\lambda$ takes values in $\mathcal{O}^\times$. If $\lambda$ has conductor $\mathfrak{c}$ prime to $p$, then $\hat{\lambda}$ factors through $\Gal(K({\p^\infty\mathfrak{c}})/K)$, where $K({\p^\infty\mathfrak{c}})=\bigcup_{n\geq 1}K({\p^n\mathfrak{c}})$ and $K({\p^n\mathfrak{c}})$ is the ray class field of $K$ of conductor $\mathfrak{p}^n\mathfrak{c}$. 
Denote by $\bar\lambda$ the complex conjugate character 
of $\lambda$ defined by $x\mapsto\lambda(\bar{x})$, where $x\mapsto \bar x$ is the complex conjugation on $K$. 
Then $\bar\lambda$ has infinity type $(0,1)$ and the $p$-adic avatar 
of $\lambda\bar\lambda$ is equal to the product $\chi_\lambda\cdot\chi_{\cyc,K}$ where  
$\chi_{\cyc,K}=\chi_\cyc\circ \rec_\Q \circ \mathbf{N}_{K/\Q}$ and $\chi_\lambda$ is a finite order character unramified at $p$. Let $\mathcal{O}^\times_\mathrm{free}$ denote the maximal $\Z_p$-free quotient of $\mathcal{O}$, and let $W\subseteq\mathcal{O}^\times_\mathrm{free}$ the subset topologically generated by the values of $\hat\lambda$. Write $p^s=[W:\Gamma]$ and let $\mathcal{O}[[S]]$ be the extension of $\Lambda$ defined by the relation $(1+S)^{p^s}=1+p$.
Enlarging $\I$ if necessary, we may assume that $\mathcal{O}[[S]]\subseteq\mathbb{I}$. 
Let $w$ be a topological generator of $W$ and $x\mapsto\langle \lambda(x)\rangle$ the composition of $\lambda$ with the projection to $\mathcal{O}^\times_\mathrm{free}$. Define the family of Hecke characters 
\[{\boldsymbol{\lambda}}\colon K^\times\backslash\widehat{K}^\times\longrightarrow\I^\times\]
by the formula 
$
{\boldsymbol{\lambda}}(x)=\hat{\lambda}(x)(1+S)^{l(x)}$, where $l(x)$ is defined by the equation $\langle\lambda(x)\rangle=w^{l(x)}$. Denote by $x \mapsto \hat{\lambda}(\bar{x})^{-1}=\lambda(\bar{x})^{-1}x_{\bar{\mathfrak{p}}}^{-1}$ (for $x \in \widehat{K}^{\times}$) the $p$-adic avatar of the Hecke character given by $x \mapsto \lambda(\bar{x})^{-1}$ (for $x \in \A^\times_K$) of infinity type $(0,-1)$; define
\[
{\boldsymbol{\lambda}}^{-1}(\bar{x})=\hat{\lambda}(\bar{x})^{-1}[\langle\hat{\lambda}(\bar{x})^{-1}\rangle^{1/2}]\] which 
we see as taking values in $\I^\times$.
Finally, define the character \begin{equation}\label{defXI}{\boldsymbol{\xi}}\colon K^\times\backslash\widehat{K}^\times\longrightarrow \I^\times\end{equation}   by
$
{\boldsymbol{\xi}}(x)={\boldsymbol{\lambda}}(x)\cdot{\boldsymbol{\lambda}}^{-1}(\bar x).$ Note that ${\boldsymbol{\xi}}_{|\widehat{\Q}^\times}$ is trivial and its prime-to-$p$ conductor is $c=\mathfrak{c}\bar{\mathfrak{c}}$. 

We now study the specializations of these characters at arithmetic weights. 
Let $\nu\colon\I\rightarrow\overline\Q_p$ be an arithmetic morphism of signature $(k_\nu ,\psi_\nu )$ 
and write $\hat{\lambda}_\nu =\nu \circ{\boldsymbol{\lambda}}$. 
Then, for $x \in \widehat{K}^\times$ and $k_\nu \equiv k_0\equiv 2\mod{2(p-1)}$, we have 
\[\hat{\lambda}_\nu (x)=\psi_\nu ^{1/2}(\langle \hat{\lambda}(x)\rangle)\cdot {\lambda}(x) ^{k_\nu /2} x_{\mathfrak{p}}^{k_\nu /2}.\]
Hence $\hat{\lambda}_\nu $ is the $p$-adic avatar 
of an algebraic Hecke character $\lambda_\nu $ of infinity type $(k_\nu /2,0)$. 
Also, set as above $\hat{\xi}_\nu =\nu \circ {\boldsymbol{\xi}}$. For any $x \in \widehat{K}^\times$ and $k_\nu \equiv k_0\equiv 2\mod{2(p-1)}$, we have 
\begin{equation}\label{hatxi}
\hat{\xi}_\nu (x)= \psi_\nu ^{1/2}(\langle{\lambda}(x
\bar{x}^{-1})x_\mathfrak{p} x^{-1}_{\bar{\mathfrak{p}}}\rangle)\cdot {\lambda}(x\bar{x}^{-1}) ^{k_\nu /2} \cdot x_{\mathfrak{p}}^{k_\nu /2} x_{\bar{\mathfrak{p}}}^{-k_\nu /2}.\end{equation}
Therefore, 
$\hat{\xi}_\nu $ is the $p$-adic avatar of an anticyclotomic Hecke character $\xi_\nu $ of infinity type $(k_\nu /2, -k_\nu /2)$. 

\subsection{Twist of big Galois representations}\label{subsec.TwistGaloisHecke}
We now consider the representation obtained by twisting $\mathbf{T}^\dagger$ by $\boldsymbol{\xi}$. Fix a continuous character 
$\boldsymbol{\xi}\colon K^\times\backslash\widehat{K}^\times\rightarrow\I^\times$ as in \eqref{defXI}, 
and denote by the same symbol the associated Galois character $\boldsymbol{\xi}\colon G_K\rightarrow\I^\times$.
Let $\mathbf{T}^\dagger_{|G_K}$ denote the restriction of $\mathbf{T}^\dagger$ to the subgroup $G_K\subseteq G_\Q$. 
Define the $G_K$-representation 
\[\mathbf{T}^\dagger_{\boldsymbol{\xi}}=
\mathbf{T}^\dagger_{|G_K}\otimes\boldsymbol{\xi}^{-1}.\] From \eqref{filtration} we obtain a filtration 
of $D_v\cong  G_{\Q_p}$-modules (recall that $p$ is split in $K$) 
\begin{equation*}\label{filtration1}
0\longrightarrow \mathbf{T}^{\dagger,+}_{\boldsymbol{\xi}}
\longrightarrow  \mathbf{T}^\dagger_{\boldsymbol{\xi}}\longrightarrow\mathbf{T}^{\dagger,-}_{\boldsymbol{\xi}}\longrightarrow 0 
\end{equation*} and as $G_{\Q_p}$-representations we have an isomorphism 
\[\mathbf{T}^\dagger_{\boldsymbol{\xi}}\cong \mat{\eta_v^{-1}\chi_\cyc\Theta\boldsymbol{\xi}^{-1}}{*}{0}{\eta_v\Theta^{-1}\boldsymbol{\xi}^{-1}}.\]
Define the Galois character $\Psi\colon G_K\rightarrow\I^\times$ by 
$\Psi=\eta_v^{-1}\chi_\cyc\Theta\boldsymbol{\xi}^{-1}$. 
\begin{lemma}\label{lemmaunram} $\Psi\colon G_{K_\mathfrak{p}}\rightarrow\I^\times$ is unramified. 
\end{lemma}
\begin{proof} Since $\lambda$ has infinity type $(1,0)$ and it is unramified at $p$, we have  $\hat{\lambda}\hat{\bar{\lambda}}=\chi_{\cyc} \beta$, with $\hat{\bar{\lambda}}\colon  x \mapsto \hat{\lambda}(\bar{x})$ and $\beta$ a character of finite order and unramified at $\mathfrak{p}$. 
Since $k \equiv 2 \mod{2(p-1)}$, a simple computation shows that 
$\eta_v^{-1} \chi_\cyc \Theta \boldsymbol{\xi}^{-1} = \eta_v^{-1} \beta^{-1}\hat{\bar{\lambda}}^2 [\langle \beta\rangle^{-1/2}] [\langle \hat{\bar{\lambda}}\rangle]$,
and the result follows because 
$\eta_v^{-1}$ is unramified as a $G_{K_\mathfrak{p}}$-character and $\beta^{-1}\hat{\bar{\lambda}}^2 [\langle \beta\rangle^{-1/2}] [\langle \hat{\bar{\lambda}}\rangle]$ is unramified at $\mathfrak{p}$ seen as the $p$-adic avatar  of a Hecke character.
\end{proof}


For each arithmetic morphism $\nu\colon \I\rightarrow\mathcal{O}_\nu$, define $\mathbf{T}^\dagger_{\xi_\nu}\defeq\mathbf{T}^\dagger_{\boldsymbol{\xi}}\otimes_{\I,\nu}F_{\nu}$. We have then an exact sequence of $G_{\Q_p}$-modules 
\[0\longrightarrow \mathbf{T}^{\dagger,+}_{\xi_\nu}\longrightarrow \mathbf{T}^\dagger_{\xi_\nu}\longrightarrow \mathbf{T}^{\dagger,-}_{\xi_\nu}\longrightarrow 0\] where $\mathbf{T}^{\dagger,\pm}_{\xi_\nu}=\mathbf{T}^{\dagger,\pm}_{\boldsymbol{\xi}}\otimes_{\I,\nu}F_\nu$  
are $F_\nu$-vector spaces of dimension one. 

We finally introduce further twists by Hodge--Tate characters. 
Fix a prime $\mathfrak{P}$ of $\overline\Q$ over $\mathfrak{p}$. Denote by
$F=H_{c,\mathfrak{P}}$ the completion of ${H}_{c}$ at $\mathfrak{P}$, $F_\infty=\Q_p^\unr$ the maximal unramified extension of $\Q_p$ (which contains $F$ and is also the maximal unramified extension of $F$ because $p\nmid c$)  
and $L_\infty={H}_{cp^\infty,\mathfrak{P}}$, the completion of ${H}_{cp^\infty}$ at $\mathfrak{P}$. Recall that $L_\infty=F(\mathfrak{F})$ is obtained by adjoining the torsion points of the relative Lubin--Tate formal group 
$\mathfrak{F}$ of parameter $\pi/\bar\pi$, where if $s$ is the order of $\p$ in $\Pic(\mathcal{O}_c)$, then $\p^s=(\pi)$ with $\pi\in\mathcal{O}_c$ (see \cite[Proposition 8.3]{Shnidman} for the proof; see also \cite[page 604]{CH}). 
Let $K_\infty=L_\infty(\mu_{p^\infty})$ and define 
$\mathcal{G}=\Gal(K_\infty/F),$
$\Gamma_\infty=\Gal(L_\infty/F),$
$\Gamma_\cyc=\Gal(F(\mu_{p^\infty})/F).$ 
We also note that if we let $\widetilde{H}_{cp^n}=H_{cp^n}(\mu_{p^n})$ and $\widetilde{H}_{cp^\infty}=\bigcup_{n\geq 1}\widetilde{H}_{cp^n}$, then $K_\infty=\widetilde{H}_{cp^\infty,\mathfrak{P}}$ is the completion of 
$\widetilde{H}_{cp^\infty}$ at $\mathfrak{P}$. 
We thus have the following diagram of local fields: 
\begin{equation}\label{fieldext}
\xymatrix{
&K_\infty\ar@{-}[d]\ar@/^2.0pc/@{-}[dd]^-{\mathcal{G}}\\ 
F(\mu_{p^\infty})\ar@{-}[dr]_-{\Gamma_\cyc}\ar@{-}[ur]& L_\infty\ar@{-}[d]_-{\Gamma_\infty} \\ 
& F
}
\end{equation}For any finite extension $L$ of $F$ in $L_\infty$, and any $\mathcal{G}$-stable subquotient $\mathbf{M}$ of 
$\mathbf{T}^\dagger_{\boldsymbol{\xi}}$, define 
\[H^1_\mathrm{Iw}(L_\infty/L,\mathbf{M})=
H^1_\mathrm{Iw}\left(\Gal(L_\infty/L),\mathbf{M}\right)=\invlim_{L'} H^1(L',\mathbf{M})\] 
where $L'$ runs over the finite extensions of $L$ contained in $L_\infty$.

Let $(\nu,\phi)$ be a pair consisting of an arithmetic morphism
$\nu\colon \I\rightarrow\mathcal{O}_\nu$ 
and a Hodge--Tate character $\phi\colon \mathcal{G}\rightarrow\overline\Q_p^\times$ of Hodge--Tate weight $m\in\Z$; 
we adopt the convention that the Hodge--Tate weight of the cyclotomic character $\chi_\cyc\colon G_{\overline\Q_p}\rightarrow\Z_p^\times$ 
is $+1$, so $\phi=\chi_\cyc^m\varphi$ for some unramified character $\varphi\colon \mathcal{G}\rightarrow\overline\Q_p^\times$.
For $\bullet$ being $+$, $-$ or no symbol, let 
$\mathbf{T}^{\dagger,\bullet}_{{\xi}_\nu}(\phi)$ denote the twist of the representation 
$\mathbf{T}^{\dagger,\bullet}_{{\xi}_\nu}$ by  
 $\phi$. We then have 
specialization maps 
\[\mathrm{sp}_{\nu,\phi}\colon H^1_\mathrm{Iw}(\Gamma_\infty,\mathbf{T}^{\dagger,\bullet}_{\boldsymbol{\xi}})
\overset\sim\longrightarrow H^1(F,\mathbf{T}^{\dagger,\bullet}_{\boldsymbol{\xi}}
\widehat{\otimes}_{\I}\I[[\mathcal{G}]])
\longrightarrow H^1(F,\mathbf{T}_{\xi_\nu}^{\dagger,\bullet}(\phi)),\]
where the first map is induced by Shapiro's isomorphism.

\subsection{Eichler--Shimura cohomology}
Define the \emph{$p$-adic Eichler Shimura group} to be the $G_\Q=\Gal(\overline{\Q}/\Q)$-representation:  
\[\ES_{\Z_p}=\invlim_m H^1_\text{\'et}(\overline{X}_m,\Z_p)\]
where $\overline{X}_m=X_m\otimes_\Q\overline{\Q}$ and the inverse limit is taken with respect to the canonical projection maps $X_{m+1}\rightarrow X_{m}$ for $m\geq 1$. 
If $\mathcal{O}$ is the valuation ring of a
complete subfield $K\subseteq\C_p$, denote $\mathrm{ES}_\mathcal{O}=\mathrm{ES}\widehat{\otimes}_{\Z_p}\mathcal{O}.$
Let $T_\ell^*$ for $\ell\nmid Np$ and $U_\ell^*$ for $\ell\mid N^+p$ be the standard Hecke operators 
acting by correspondences on $H^1_{\et}(\overline{X}_m,\Z_p)$; we also denote $\langle a\rangle^*$ the diamond operator; recall that 
the relation with operators acting on modular forms is $T=\tau_m T^*\tau_m^{-1}$, where $\tau_m$ is the Atkin--Lehner involution. 
These actions are compatible with respect to the projection maps, and therefore we obtain actions on $\ES_\mathcal{O}$. Let $\mathfrak{h}_m^*\subseteq \End(H^1_{\et}(\overline{X}_m,\Z_p))$ be the subalgebra generated by these operators, and put $\mathfrak{h}_\infty^*=\invlim\mathfrak{h}_m^*$. Then $\ES_\mathcal{O}$ has a natural structure of $\mathfrak{h}_\infty^*$-module. 
Let $\widetilde\Lambda_\mathcal{O}=\mathcal{O}[[\Z_p^\times]]$. Observe that diamond operators equip $\mathfrak{h}_\infty^*$ with a canonical structure of $\widetilde\Lambda_\mathcal{O}$-algebra; 
therefore, $\ES_\mathcal{O}$ is also a $\widetilde{\Lambda}_\mathcal{O}$-module. 
Let $e^{\ord,*}$ denote the Hida ordinary idempotent associated to $U_p^*$ and define $\mathfrak{h}_\infty^{\ord,*}=e^{\ord,*}\mathfrak{h}_\infty^*$. We also set 
$\ES_\mathcal{O}^\ord=e^{\ord,*}\ES_\mathcal{O}.$
Then $\ES_\mathcal{O}^\ord$ is a $\mathfrak{h}_\infty^{\ord,*}$ and also a $\widetilde{\Lambda}_\mathcal{O}$-module (and the two structures are compatible with the 
$\widetilde{\Lambda}_\mathcal{O}$-module structure of $\mathfrak{h}_\infty^{\ord,*}\otimes_{\Z_p}\mathcal{O}$). 
We can write $\widetilde{\Lambda}_\mathcal{O}\simeq \Lambda_\mathcal{O}\times \mathcal{O}[(\Z/p\Z)^\times]$ with $\Lambda_\mathcal{O}\simeq \mathcal{O}[[T]]$. 
If $\omega$ is the Teichmüller character, we can consider for each $0\le i\le p-1$ the $\omega^i$-eigenspace $\ES^\ord_\mathcal{O}(i)$ of $\ES_\mathcal{O}^\ord$. 

Recall the primitive branch $\I/\Lambda_\mathcal{O}$ fixed before. Let $\mathfrak{O}$ the valuation ring of a complete subfield $F$ of $\overline{\Q}_p$ which contains all the $p$-power roots of unity and the ring $\Z_p^\unr=W(\overline{\F}_p)$ of Witt vectors of the algebraic closure of the field with $p$-elements $\F_p$. Then $\mathcal{O}\subseteq\mathfrak{O}$ and we put 
$\mathbb{J}=\mathbb{I}\otimes_{\mathcal{O}}\mathfrak{O}$; 
using the isomorphism $\mathfrak{h}_\infty^\ord\simeq\mathfrak{h}_\infty^{\ord,*}$ we may define the $\Lambda_\mathcal{O}$-algebras $\I^*\simeq\I$ and $\J^*\simeq\J$.  
Define \[\ES_\J=\ES^\ord_{\mathcal{O}}(0)_{\Z_p}\otimes_{\mathfrak{h}_\infty^{\ord,*}}\mathbb{J}^*.\] Then the $G_\Q$-representation $\ES_\J$ is a free $\J$-module of rank $2$, equipped with a split filtration 
\[0\longrightarrow \mathfrak{A}_\J^*\longrightarrow\ES_\J\longrightarrow \mathfrak{B}_\J^*\longrightarrow 0\]
with free $\J$-modules $\mathfrak{A}_\J^*$ and $\mathfrak{B}_\J^*$ of rank $1$, satisfying the following properties. For $\bullet$ being $+$, $-$, no symbol; or any of these three paired with $\dagger$,
define $\widetilde{\mathbf{T}}^\bullet=\mathbf{T}^\bullet\otimes_\I\J$. We have 
\[\ES_\J(\Theta^2\chi_\cyc)\simeq\widetilde{\mathbf{T}}\] as $\J[G_\Q]$-modules,
and an operator $T$ on $\widetilde{\mathbf{T}}$ correspond to the operator $T^*$ on $\ES_\J$; note that this isomorphism is \emph{canonical}, as it 
comes from the canonical isomorphism between \'etale cohomology groups and the dual of Tate modules. 
Then, we have 
$\mathfrak{A}_\J^*=(\ES_\J)^{I_{\Q_p}},$ and  
\[\mathfrak{A}_\J^*(\Theta^2\chi_\cyc)\simeq \widetilde{\mathbf{T}}^-\] 
as $\J[G_\Q]$-modules, also canonically. 
Furthermore, $\sigma\in I_{\Q_p}$ acts on $\mathfrak{B}_{\J}^{*}$ as $\Theta^{-2}\chi_\cyc^{-1}$ and  
\[\mathfrak{B}_\J^*(\Theta^2\chi_\cyc)\simeq  \widetilde{\mathbf{T}}^+\] as $\J[G_\Q]$-modules, again canonically. Finally, we have a canonical perfect and Galois-equivariant pairing (coming from the Poincaré duality) 
\[\mathfrak{A}^*_\J\times \mathfrak{B}_\J^*\longrightarrow \J\]
of free $\J$-modules of rank $1$. Define \[\mathbf{D}(\mathbf{T}^+)=(\mathbf{T}^+\widehat{\otimes}_{\Z_p}\Z_p^\unr)^{G_{\Q_p}},\] which 
by \cite[Lemma 3.3]{Ochiai-Coleman}, is a free $\I$-module of rank $1$ (recall that $\mathbf{T}^+$ is unramified). Furthermore, since $\widetilde{\mathbf{T}}^+\simeq (\mathbf{T}^+\otimes_{\Z_p}\Z_p^\unr)\otimes_{\Z_p^\unr}\mathfrak{O}$, we have $\mathbf{D}(\mathbf{T}^+)=(\widetilde{\mathbf{T}}^+)^{G_{\Q_p}}$. Fix an $\I$-basis $\omega_\I$ for $\mathbf{D}(\mathbf{T}^+)$ and denote by the same symbol the correspondent $\I$-basis of $(\widetilde{\mathbf{T}}^+)^{G_{\Q_p}}$ under the isomorphism; under inclusion $(\widetilde{\mathbf{T}}^+)^{G_{\Q_p}}\subseteq \widetilde{\mathbf{T}}^+$, combined with the isomorphism $\widetilde{\mathbf{T}}^+\simeq\mathfrak{B}_\J^*(\chi_\cyc\Theta^2)$, $\omega_\I$ gives 
a $\J$-generator of $\mathfrak{B}_\J^*$ as $\J$-module, denoted by the same symbol.

We now review some of the results in \cite{LMW}.
Let $\Gamma_m$ denote the subgroup of elements of $R_m^\times$ of norm $1$ which are congruent to $\smallmat 1*01$ modulo $\ell^{\ord_\ell(N^+p^m)}$ for all primes $\ell\nmid N^-$. Then as Riemann surfaces we have $X_m(\C)\simeq\Gamma_m\backslash\mathcal{H}$, where $\Gamma_m$ acts on the complex upper half plane $\mathcal{H}$ via fractional linear transformations through the fixed isomorphism $i_\infty\colon B_\infty\simeq \mathrm{M}_2(\R)$. For any field $C$ containing $\Q$, define 
\[S_2(\Gamma_m,C)=H^0(X_{m/C},\Omega^1)\]
where $X_{m/C}=X_m\otimes_\Q C$. 
The module $\mathfrak{B}_\J^*$ injects into the inverse limit $\invlim S_2(\Gamma_m,K)$ of quaternionic modular forms of level $\Gamma_m$ with coefficients in $K$, where the inverse limit is computed with respect to the trace maps $\mathrm{Tr}_m\colon S_k(\Gamma_m,F)\rightarrow S_k(\Gamma_{m-1},F)$ for $m\geq 2$. 
The image is contained in the $\mathcal{O}$-submodule consisting of those $(\mathcal{F}_m)_{m\geq1}$ such that 
$\mathcal{F}_m|\tau_m$ and $\mathcal{F}_m|\tau_m|U_p^m$ have Serre--Tate expansions in $\mathfrak{O}[[T_x]]$ for any point $x$ in the Igusa tower over $\mathbb{X}_0^\ord$, the ordinary locus of the special fiber $\mathbb{X}_0$ of the model $\mathcal{X}_0$ of the Shimura curve $X_0$ over $\Z_p$; here $\tau_m$ is the Atkin--Lehner involution. To each such sequence of modular forms $(\mathcal{F}_m)_{m\geq 1}$ we can attach a power series in $\J[[T_x]]$.
Then $\omega_\I$ induces a family of modular forms $\boldsymbol{\mathcal{F}}$ and therefore a power series $\boldsymbol{\mathcal{F}}(T_x)=\sum_{n\geq 1}a_nT_x^n\in\widetilde{\I}[[T_x]]$, with $\widetilde{\I}\defeq \I\otimes_{\Z}\Z^{\unr}$ (see \cite[Lemma 5.9]{LMW}). For each arithmetic morphism $\nu\colon\mathbb{I}\rightarrow \mathcal{O}_\nu$, 
by specialization we obtain a power series $\mathcal{F}_\nu(T_x)=\sum_{n\geq 1}a_{n,\nu}T_x^n$ in $\mathcal{O}_\nu^\unr[[T_x]]$, $a_{n,\nu}=\nu(a_n)$ and $\mathcal{O}_\nu^\unr$ is the maximal unramified extension of $\mathcal{O}_\nu$. 
Then $\mathcal{F}_\nu(T_x)$ is the $T_x$-expansion of a unique modular form $\mathcal{F}_\nu\in S_k(\Gamma_m,F_\nu)$ 
(so, in fact $\mathcal{F}_\nu\in S^\ST_k(\Gamma_m,\mathcal{O}^\unr_\nu)$).  
This fixes a choice of Jacquet--Langlands lift of the elliptic modular form $f_\nu$ corresponding to $\nu$; this choice depends on the choice of $\omega_\I$ only. Applying this construction to an arithmetic morphism $\nu$ of signature $(2,\psi)$ such that the conductor of $\psi$ is $p^m$, we thus obtain from $\omega_\I$ a modular form $\mathcal{F}_\nu\in S_2(\Gamma_m,F_\nu)$ whose Serre--Tate expansion $\mathcal{F}_\nu(T_x)$ is just the specialization of $\boldsymbol{\mathcal{F}}(T_x)$ at $\nu$, and an element 
\begin{equation}\label{defomeganu}
\omega_{\mathcal{F}_\nu}\in \mathrm{Fil}^1(\mathbf{D}_{\mathrm{dR}}(\widetilde{\mathbf{T}}_\nu^*)).\end{equation}

We then have a commutative diagram
\begin{equation}\label{diagramdeRham}
\begin{tikzcd}
\mathbf{D}(\mathbf{T}^+)\arrow{r}{\omega_\I}\arrow{d}{\mathrm{sp}_\nu}&{\I}\arrow{d}{\nu} \\ \mathbf{D}_{\mathrm{dR}}(\mathbf{T}^+_\nu)\arrow{r}{\omega_{\mathcal{F}_\nu}}& F_\nu
\end{tikzcd}
\end{equation}
where the bottom horizontal arrow is given by pairing with the differential 
$
\omega_{\mathcal{F}_\nu}\in \mathrm{Fil}^1(\mathbf{D}_\mathrm{dR}({\mathbf{T}}_\nu^*))$ 
under the isomorphisms
\[\mathbf{D}_{\mathrm{dR}}({\mathbf{T}}^+_{\nu})\simeq \frac{\mathbf{D}_\mathrm{dR}({\mathbf{T}}_\nu)}{\mathrm{Fil}^0(\mathbf{D}_\mathrm{dR}({\mathbf{T}}_\nu))}\simeq
\mathrm{Fil}^1(\mathbf{D}_{\mathrm{dR}}({\mathbf{T}}_\nu^*))^\vee\]
where the first isomorphism is \cite[Lemma 3.2]{Ochiai-Coleman}, and the second is 
given by the de Rham pairing \[\langle\cdot,\cdot\rangle_\mathrm{dR}:\mathbf{D}_\mathrm{dR}({\mathbf{T}}_\nu)\times
\mathbf{D}_\mathrm{dR}({\mathbf{T}}_\nu^*)\longrightarrow K
\] so that 
$\omega_{\mathcal{F}_\nu}(\mathrm{sp}_\nu(x))=\langle \mathrm{sp}_\nu(x),\omega_{\mathcal{F}_\nu}\rangle_\mathrm{dR}$. 

Let $\nu$ be an arithmetic morphism of signature $(2,\psi)$. Under the comparison isomorphisms 
$H^1(X_{m/K},K)\simeq H^1_\mathrm{dR}(X/K)$ and $H^1_\mathrm{dR}(X/K)[\mathcal{F}_\nu]\simeq\mathbf{D}_\mathrm{dR}(\mathbf{T}_\nu)$ 
(where $H^1_\mathrm{dR}(X/K)[\mathcal{F}_\nu]$ is the subspace of $H^1_\mathrm{dR}(X/K)$ where Hecke operators act 
with the same eigenvalues as $\mathcal{F}_\nu$) we see that the differential for $\mathcal{F}_\nu\in H^0(X_{m/K},\Omega^1)$ corresponds to the differential $\omega_{\mathcal{F}_\nu}$ in \eqref{defomeganu}.

\begin{remark} For each $\nu$, \eqref{defomeganu} specifies a choice of generators of 
$\mathbf{D}_\mathrm{dR}(\mathbf{T}_\nu^+)_K$; this is a \emph{non-canonical} choice, because 
it depends (up to multiplication by elements in $\I^\times$) on the choice of $\omega_\I$. 
The relation between $\mathcal{F}_\nu$ and $\omega_{\mathcal{F}_\nu}$ is $\omega_{\mathcal{F}_\nu}(T_x)=\mathcal{F}_\nu(T_x)dT_x$, where $\mathcal{F}_\nu(T_x)$ 
is the Serre--Tate expansion of $\mathcal{F}_\nu$ at $x$.  
\end{remark}

\begin{remark}
At the moment, there is no analogue of the relation between $\mathcal{F}_\nu$ and $\omega_{\mathcal{F}_\nu}$ for primes $\nu$ of signature $(k,\psi)$ with $k\neq 2$. The reason is that the Galois representation $\mathbf{T}_\nu$ does not arise as projection from $\mathbf{T}$ to any of the cohomology groups $H^1(\Gamma_m,K)$, but instead is constructed using Kuga--Sato varieties. In the $\GL_2$-case, nevertheless, a beautiful argument using Beilinson--Kato elements shows that such a relation holds even when the weight of $\nu$ is different from $2$: see \cite[Theorem 10.1]{KLZ.ERL}.
\end{remark}

The next goal is to twist \eqref{diagramdeRham} by $\boldsymbol{\xi}$ and 
Hodge--Tate characters $\phi$. We first recall some notation and definitions from \cite[Definition 3.12]{Ochiai-Coleman} and \cite[\S3.2]{Castella}. Fix a compatible sequence $(\zeta_{p^n})_{n\geq1}$ of $p$-power roots of unity; so for each integer $n\geq 1$, $\zeta_{p^n}$ is a primitive $p^n$-th root of unity such that $\zeta_{p^{n+1}}^p=\zeta_{p^{n}}$. This choice defines a generator of $\Q_p(j)$, denoted $e_j$. Let $t$ denote Fontaine's $p$-adic analogue of $2\pi i$, defined, \emph{e.g.} in \cite[Ch. II, \S1.1.15]{Kato1}. Then $\delta_r=t^{-r}\otimes e_r$ is a generator of the $1$-dimensional $\Q_p$-vector space 
$\mathbf{D}_{\dR}\left(\Q_p(\chi_\cyc^{r})\right)$; where for a character $\varphi$ we let $\Q_p(\varphi)$ denote the one-dimensional representation affording $\varphi$. 
Next, recall the Galois group $\mathcal{G}=\Gal(K_\infty/F)$ introduced in\S\ref{subsec.TwistGaloisHecke} and let $\phi\colon \mathcal{G}\rightarrow\overline\Q_p^\times$ be the $p$-adic avatar of a Hecke character of infinity type $(r,-r)$ for an integer $r\in\Z$. 
Then $\phi\chi_\cyc^{-r}$ is unramified at $\p$. Fix a basis $\omega_{\phi\chi_\cyc^{-r}}$ of the $1$-dimensional $\Q_p(\phi\chi_\cyc^{-r})$-vector space 
$\mathbf{D}_{\dR,F}(\Q_p(\phi\chi_\cyc^{-r}))$. 
{One defines a map 
$\varphi_\phi\colon \Z_p[[\mathcal{G}]]\rightarrow  \mathbf{D}_{\dR,F}\left(\Q_p(\phi\chi_\cyc^{-r}))\right)$  
setting $\varphi_\phi(\sigma)=(\phi\chi_\cyc^{-r})(\sigma)\omega_{\phi\chi_\cyc^{-r}}$ on group-like elements.
Define as before the $\I$-module 
\[\mathbf{D}(\mathbf{T}_{\boldsymbol{\xi}}^{\dagger,+})=(\mathbf{T}_{\boldsymbol{\xi}}^{\dagger,+}\widehat\otimes_{\Z_p}\widehat{\Z}_p^\unr)^{G_{\Q_p}}.\] 
By Lemma \ref{lemmaunram}, $\mathbf{T}_{\boldsymbol{\xi}}^{\dagger,+}$ is unramified, 
therefore $\mathbf{D}(\mathbf{T}_{\boldsymbol{\xi}}^{\dagger,+})$ is a free $\I$-module of rank $1$.  
We construct a map 
\begin{equation}\label{spec}
\mathrm{sp}_{\nu,\phi}\colon \mathbf{D} (\mathbf{T}^{\dagger,+}_{\boldsymbol{\xi}} )\widehat\otimes_{\Z_p}\mathcal{O}_F[[{\mathcal{G}}]]\longrightarrow\mathbf{D}_{\dR,F} ( \mathbf{T}^{\dagger,+}_{\xi_\nu}(\phi) )\end{equation}
setting $\mathrm{sp}_{\nu,\phi}=\mathrm{sp}_{\nu}\otimes \varphi_\phi\otimes \delta_r$ and using the canonical map 
\[  
\mathbf{D}_{\dR,F}(\mathbf{T}^{\dagger,+}_{\xi_\nu})\otimes_{F_\nu}\mathbf{D}_{\dR}(F_{\nu,\phi}(\phi\chi_\cyc^{-r})))\otimes_{F_{\nu,\phi}}\mathbf{D}_{\dR,F}(F_{\nu,\phi}(\chi_\cyc^{r})))
\longrightarrow \mathbf{D}_{\dR,F} ( \mathbf{T}^{\dagger,+}_{\xi_\nu}(\phi) ).
\]
From 
\eqref{diagramdeRham} we obtain a commutative diagram:
\begin{equation}\label{lemmaFaltings}
\xymatrix{\mathbf{D}({\mathbf{T}}^{\dagger,+}_{\boldsymbol{\xi}})\widehat\otimes_{\Z_p}\mathcal{O}_F[[{\mathcal{G}}]]
\ar[r]^-{\omega_\I\otimes1} \ar[d]^-{\mathrm{sp}_{\nu,\phi}}& \I\widehat\otimes_{\Z_p}\mathcal{O}_F[[{\mathcal{G}}]]\ar[d]^-{\mathrm{sp}_{\nu,\phi}}\\
\mathbf{D}_{\dR,F}({\mathbf{T}}_{\xi_\nu}^{\dagger,+}(\phi))\ar[r]^-{{\omega}_{\mathcal{F}_\nu}\otimes\phi^{-1}} &F_\nu(\phi)
}
\end{equation}
where $F_\nu(\phi)$ is the field generated over $F_\nu$ by the values of $\phi$ and, as before, the bottom horizontal map is given by the de Rham pairing \[\langle\cdot,\cdot\rangle_\mathrm{dR}:\mathbf{D}_{\mathrm{dR},F}({\mathbf{T}}_{\xi_\nu}(\phi))\times
\mathbf{D}_{\mathrm{dR},F}({\mathbf{T}}_{\xi_\nu}^*(\phi^{-1}))\longrightarrow F_\nu(\phi)
\] so that 
$(\omega_{\mathcal{F}_\nu}\otimes\phi^{-1})(\mathrm{sp}_\nu(x))=\langle \mathrm{sp}_\nu(x),\omega_{\mathcal{F}_\nu}\otimes\phi^{-1}\rangle_\mathrm{dR}$. }

\subsection{The big Perrin-Riou map}
Recall from \S\ref{subsec.TwistGaloisHecke} the Galois character $\Psi\colon G_{K_\p}\rightarrow\I^\times$ given by 
$\Psi=\eta_v^{-1}\chi_\cyc\Theta\boldsymbol{\xi}^{-1}$. For $\p$ the prime of $K$ above $p$ corresponding to the fixed embedding $\overline{\Q}\hookrightarrow\overline{\Q}_p$, denote by $i_\mathfrak{p}\colon \Z_p^\times\rightarrow\widehat{K}^\times$ the map which takes $z\in \Z_p^\times \cong \mathcal{O}_{K,\mathfrak{p}}^\times$ to the element $i_\mathfrak{p}(z)$ with  $\mathfrak{p}$-component equal to $z$ and trivial components at all the other places. Let $\Frob_\p\in G_K$ be an arithmetic Frobenius at $\p$; for 
each arithmetic character $\nu\colon \I\rightarrow\overline\Q_p$, let $\Psi_\nu=\nu\circ\Psi$. Since 
$k\equiv 2\mod{2(p-1)}$ 
and, identifying Galois and adelic character when convenient,  
$\chi_\cyc(i_\p(p))=\mathbf{N}_K(i_\mathfrak{p}(p)^{-1})(i_\mathfrak{p}(p)^{-1})$, a simple computation shows that 
$\Psi(\Frob_\p)=\mathbf{a}_p^{-1}\boldsymbol{\xi}(i_\p(p))$, 
and therefore
$\Psi_\nu(\Frob_\p)=\mathbf{a}_{p,\nu}^{-1}{\xi}_{\nu, \mathfrak{p}}(p)p$.

Set $\mathfrak{j}=\Psi(\Frob_\p)-1\in\I$. 
Define $\mathcal{J}=(\mathfrak{j},\gamma_\cyc-1)$ to be the ideal of $\I$ generated by $\mathfrak{j}$ and $\gamma_\cyc-1$, where $\gamma_\cyc$ is a fixed topological generator of $\Gamma_\cyc$. An arithmetic morphism $\nu\colon \I\rightarrow\overline\Q_p$ is \emph{exceptional} if its signature is 
$(2,\mathbf{1})$, where $\mathbf{1}$ is the trivial character, and $\mathfrak{j}=0$.  

Let $\phi\colon \mathcal{G}\rightarrow\overline\Q_p^\times$ be a character
of Hodge--Tate type with Hodge--Tate weight $w$ and conductor $p^n$ for some integer $n\geq 0$. Write $\phi=\chi_\cyc^w\phi'$ for some unramified character $\phi'$. 
For each $\nu$ 
we may consider the $1$-dimensional (over $F_\nu$) representation $V(\Psi_\nu)=F_\nu(\Psi_\nu)$ and its crystalline 
Dieudonn\'e module $\mathbf{D}_\mathrm{cris}(V(\Psi_\nu))$. Then the crystalline Frobenius acts on  
$\mathbf{D}_{\mathrm{cris}}(V(\Psi_\nu))$ by $\Phi_\nu=\Psi_\nu^{-1}(\Frob_\p)$ (\cite[Lemma 8.3.3]{BrCo}). 
Define
$\mathcal{E}_p(\phi,\nu)$ by
\[\mathcal{E}_p(\phi,\nu)=\begin{cases}
\displaystyle{\frac{1-p^w\phi'(\Frob_\p)\Phi_\nu}{1-(p^{w+1}\phi'(\Frob_\p)
\Phi_\nu)^{-1}}},&\text{if }n=0,\\
\epsilon(\phi^{-1})\cdot \Phi_\nu^n,& \text{if }n\geq 1,
\end{cases}\]
where, for any character $\varphi\colon \Gal(\Q_p^\mathrm{ab}/\Q_p)\rightarrow\overline\Q_p^\times$ of Hodge--Tate type, 
$\epsilon(\varphi)$ is the $\epsilon$-factor of the Weil--Deligne representation 
$\mathbf{D}_\mathrm{pst}(\varphi)$; we adopt the convention in \cite[\S2.8]{LZ-IwasawaZ2}
for $\epsilon$-factors, and we refer to \emph{loc. cit.} for a careful discussion. 

If $w\leq -1$, then
the finite Bloch--Kato subspace $H^1_f(F,\mathbf{T}_{\xi_\nu}^{\dagger,+}(\phi^{-1}))$ of  
$H^1(F,\mathbf{T}_{\xi_\nu}^{\dagger,+}(\phi^{-1}))$ coincides with $H^1(F,\mathbf{T}_{\xi_\nu}^{\dagger,+}(\phi^{-1}))$; 
the Bloch--Kato logarithm  for $V_{\nu,\phi}^\dagger$ gives rise to a map
\[\log\colon H^1 (F,\mathbf{T}_{\xi_\nu}^{\dagger,+}(\phi^{-1}) )\longrightarrow 
\mathbf{D}_{\mathrm{dR}} (\mathbf{T}_{\xi_\nu}^{\dagger,+}(\phi^{-1}) ).\]
By \cite[Theorem 3.7]{Castella}, there exists an injective $\I[[\mathcal{G}]]$-linear map 
\[\mathrm{Log}\colon H^1_\mathrm{Iw}
 (\Gamma_\infty,\mathbf{T}^{\dagger,+}_{\boldsymbol{\xi}} )\longrightarrow\mathfrak{j}^{-1}\cdot\mathcal{J}\cdot (\mathbf{D}(\mathbf{T}_{\boldsymbol{\xi}}^{\dagger,+})\widehat\otimes_{\Z_p}\widehat{\mathcal{O}}_{F_\infty}[[\mathcal{G}]] )\] where $\widehat{\mathcal{O}}_{F_\infty}$ is the completion of the valuation ring $\mathcal{O}_{F_\infty}$ of $F_\infty$, such that for each non-exceptional $\nu\colon \I\rightarrow\mathcal{O}_\nu$ and each non-trivial character $\phi\colon \mathcal{G}\rightarrow L^\times$ of Hodge--Tate type of conductor 
$p^n$ and Hodge--Tate weight $w\leq -1$ as above, the following diagram commutes: 
\begin{equation}\label{diaPR1}
\xymatrix{
H^1_\mathrm{Iw} (\Gamma_\infty,\mathbf{T}^{\dagger,+}_
{\boldsymbol{\xi}} )\ar[rrr]^-{\mathrm{Log}}\ar[d]^{\mathrm{sp}_{\nu,\phi^{-1}}}&&&\mathfrak{j}^{-1}\cdot\mathcal{J}\cdot (\mathbf{D}(\mathbf{T}_{\boldsymbol{\xi}}^{\dagger,+})\widehat\otimes_{\Z_p}\widehat{\mathcal{O}}_{F_\infty}[[\mathcal{G}]] )\ar[d]^{\mathrm{sp}_{\nu,\phi^{-1}}}\\
H^1 (F,\mathbf{T}_{\xi_\nu}^{\dagger,+}{(\phi^{-1})}) \ar[rrr]^{{\frac{(-1)^{-w-1}}{(-w-1)!}
\log\mathcal{E}_p(\phi,\nu)}}&&&\mathbf{D}_{\mathrm{dR}} (\mathbf{T}_{\xi_\nu}^{\dagger,+}{(\phi^{-1})}).
}
\end{equation}
Let $\widetilde{\I}[\mathfrak{j}^{-1}]= \I[\mathfrak{j}^{-1}]\widehat\otimes_{\Z_p}\Z_p^\unr$. Combining \eqref{diaPR1}  and \eqref{lemmaFaltings}, the argument in \cite[Proposition 5.2]{Castella} shows that 
there exists an injective $\I$-linear map 
\begin{equation}\label{PR}
\mathcal{L}_{\omega_{\I}}^{\Gamma_\infty}\colon 
H^1_\mathrm{Iw}(\Gamma_\infty,{\mathbf{T}}^{\dagger,+}_{\boldsymbol{\xi}})\longrightarrow \widetilde{\I}[\mathfrak{j}^{-1}][[\Gamma_\infty]] \end{equation}
with pseudo-null kernel and cokernel, such that for all characters $\phi\colon \Gamma_\infty\rightarrow\overline\Q_p^\times$ of Hodge--Tate type, with Hodge--Tate weight $w\leq -1$ and conductor $p^n$, all 
$\mathfrak{Y}\in H^1_\mathrm{Iw}(\Gamma_\infty,{\mathbf{T}}^{\dagger,+}_{\boldsymbol{\xi}})$ and all non-exceptional $\nu$ of weight $2$ we have 
\begin{equation}\label{prop12.2}
\begin{split}
\mathrm{sp}_{\nu,\phi^{-1}}\left(\mathcal{L}_{\omega_{\I}}^{\Gamma_\infty}(\mathfrak{Y})\right)&=\frac{(-1)^{-w-1}}{(-w-1)!}\cdot \mathcal{E}(\phi,\nu)\cdot
({\omega}_{\mathcal{F}_\nu}\otimes\phi)\left(\log(\mathrm{sp}_{\nu,\phi^{-1}}(\mathfrak{Y}))\right)
.\end{split}\end{equation}

\section{The algebraic $p$-adic $L$-function} \label{secBHP}


\subsection{CM points}\label{sec.CMpoints} 
Let $c$ be an integer coprime with $pND_K$ and for each integer $n\geq 0$ let 
$\mathcal{O}_{cp^n}=\Z+cp^n\mathcal{O}_K$ be the order of $K$ of conductor $cp^n$. Class field theory gives an isomorphism 
$\Pic(\mathcal{O}_{cp^n})\cong \Gal(H_{cp^n}/K)$ for an abelian extension $H_{cp^n}$ of $K$, called the \emph{ring class field of $K$ of conductor $cp^n$}. 
Define the union of these fields 
$H_{cp^\infty}=\bigcup_{n\geq 1} H_{cp^n}.$
Since $c$ is prime to $p$, $H_{c}\cap H_{p^\infty}=H$, where $H=H_1$ is the Hilbert class field of $K$, 
so we have an isomorphism of groups 
\[\Gal(H_{cp^\infty}/K)\cong  \Gal(H_c/K)\times\Gal(H_{p^\infty}/H).\]
Since $p$ is split in $K$, we have  $\Gal(H_{p^\infty}/H)\cong \Z_p^\times$ and we decompose $\Z_p^\times\cong \Delta\times\Gamma$, with $\Gamma=(1+p\Z_p)$ and $\Delta=(\Z/p\Z)^\times$.

A \emph{Heegner point $x\in X_m(\C)$ of conductor $cp^n$} is represented by a pair $(f,g)$ 
satisfying the condition $f(\mathcal{O}_{cp^n})=f(K)\cap gU_mg^{-1}$. Shimura's reciprocity law asserts that for $a\in\widehat{K}^\times$, 
we have 
$x^\sigma=[(f,\hat{f}(a^{-1})g)]$
where $\hat{f}\colon\widehat{K}\rightarrow\widehat{B}$ is the adelization of $f$,
$\mathrm{rec}_K(a)=\sigma$, and $\mathrm{rec}_K$ is the geometrically normalized reciprocity map.

Let $c=c^+c^-$ with $c^+$ divisible by primes which are split in $K$ and $c^-$ divisible by primes which are inert in $K$. Choose decompositions  $c^+=\mathfrak{c}^+\bar{\mathfrak{c}}^+$ and 
$N^+=\mathfrak{N}^+\bar{\mathfrak{N}}^+$ coming from splitting each prime factor. 
For each prime number $\ell$ and each integer $n\geq 0$, define 
\begin{itemize}
      \item $\xi_\ell=1$ if $\ell\nmid N^+cp$;
           \item $\xi_{p}^{(n)}=\frac{1}{\theta-\bar{\theta}}\smallmat{\theta}{\bar\theta}{1}{1}\smallmat{p^n}{1}{0}{1}
           \in \GL_2(K_\mathfrak{p})=\GL_2(\Q_p)$; 
        \item $\xi_{\ell}=\frac{1}{\theta-\bar{\theta}}\smallmat{\theta}{\bar\theta}{1}{1}\smallmat{\ell^s}{1}{0}{1}\in \GL_2(K_\mathfrak{l})=\GL_2(\Q_\ell)$ if $\ell\mid c^+$ and $\ell^s$ is the exact power of $\ell$ dividing $c^+$, 
        where $(\ell)=\mathfrak{l}\bar{\mathfrak{l}}$ is a factorization into prime ideals in $\mathcal{O}_K$ and $\mathfrak{l}\mid \mathfrak{c}^+$; 
       \item $\xi_{\ell}=\smallmat{0}{-1}{1}{0}\smallmat{\ell^s}{0}{0}{1}\in \GL_2(\Q_\ell)$ if $\ell\mid c^-$ and $\ell^s$ is the exact power of $\ell$ dividing $c^-$;
         \item $\xi_\ell=\frac{1}{\theta-\bar{\theta}}\smallmat{\theta}{\bar\theta}{1}{1}\in \GL_2(K_\mathfrak{l})=\GL_2(\Q_\ell)$ if $\ell\mid N^+$, where $(\ell)=\mathfrak{l}\bar{\mathfrak{l}}$ is a factorization into prime ideals in $\mathcal{O}_K$ and $\mathfrak{l}\mid \mathfrak{N}^+$.
  \end{itemize}
We understand these elements $\xi^\star_\bullet$ as elements in $\widehat{B}^\times$ by implicitly using the isomorphisms $i_\ell$ defined before. 
With this convention, define $\xi^{(n)}=(\xi_\ell,\xi_{p}^{(n)})_{\ell\neq p}\in \widehat{B}^\times$.  
Define a map $x_{cp^n,m}\colon\Pic(\mathcal{O}_{cp^n})\rightarrow {X}_m(\C)$ by $[\mathfrak{a}]\mapsto[(\iota_K,a\xi^{(n)})]$,  
where
if $\mathfrak{a}$ represents the ideal class $[\mathfrak{a}]$, then $a\in \widehat{K}^\times$ satisfies $\mathfrak{a}=a\widehat{\mathcal{O}}_{cp^n}\cap K$; here $a \in \widehat{K}^\times$ acts on $\xi^{(n)}\in \widehat{B}^\times$ via left multiplication by $\hat{\iota}_K(a)$. We often write $x_{cp^n,m}(a)$ or $x_{cp^n,m}(\mathfrak{a})$ 
for $x_{cp^n,m}([\mathfrak{a}])$.
One easily verifies that $x_{cp^n,m}(a)$ are Heegner points of conductor $cp^n$ in $X_m(H_{cp^n})$, for all 
$a\in\Pic(\mathcal{O}_{cp^n})$, and all integers $n\geq 0$ and $m\geq0$.

\subsection{Families of Heegner points}\label{CM section}
Recall $\chi_\cyc\colon \Gal(\overline\Q/\Q)\rightarrow\Z_p^\times$ the cyclotomic character and let $\vartheta\colon \Gal(\overline\Q/\Q(\sqrt{p^*}))\rightarrow\Z_p^\times/\{\pm1\}$
be the unique character which satisfies $\vartheta^2=\chi_\cyc$, where $p^*=(-1)^{\frac{p-1}{2}}p$ (see \cite[\S4.4]{LV-MM} for details).   
For integers $n\geq 0$ and $m\geq 1$, define $L_{cp^n,m}=H_{cp^{n}}(\mu_{p^m})$.  
and ${P}_{cp^n,m}={x}_{cp^n,m}(1)$. 
These points are known to satisfy the following properties: 
\begin{enumerate}
    \item ${P}_{cp^n,m}\in {X}_m(L_{cp^n,m})$; 
    \item ${P}_{cp^n,m}^\sigma=\langle\vartheta(\sigma)\rangle\cdot{P}_{cp^n,m}$ for all $\sigma\in\Gal(L_{cp^n,m}/H_{cp^{n+m}})$;
    \item \emph{Vertical compatibility}: if $m>1$, then 
    $\sum_{\sigma}\varsigma_m({P}_{cp^n,m}^\sigma)=U_p\cdot {P}_{cp^n,m-1},$
    where the sum is over all $\sigma\in\Gal(L_{cp^n,m}/L_{cp^{n-1},m})$ and 
    $\varsigma_m\colon {X}_m\rightarrow {X}_{m-1}$ is the canonical projection map; 
    \item \emph{Horizontal compatibility}: if $n>0$, then 
    $\sum_{\sigma}{P}_{cp^n,m}^\sigma=U_p\cdot {P}_{cp^{n-1},m},$
    where the sum is over all $\sigma\in\Gal(L_{cp^n,m}/L_{cp^{n-1},m})$. 
\end{enumerate}

\begin{remark} See \cite[Theorem 1.2]{CL} for a proof of the above properties; in \emph{loc.cit} only the case of definite quaternion algebras and $c=1$ is treated, but it is easy 
to see that the proof, which combines results in \cite{LV-MM} and the description of optimal embeddings in 
\cite{ChHs2}, works in this generality as well. \end{remark}

\subsection{Big Heegner points} \label{sec8.2}
Recall the fixed modular form $f$ of weight $k_0\equiv2\pmod{2(p-1)}$, let $J_m$ be the Jacobian of the Shimura curve $X_m$ and 
let $e_{k_0-2}$ denote the projector 
\[e_{k_0-2}=\frac{1}{p-1}\sum_{\delta\in\Delta}[\delta]\in \Z_p[\Z_p^\times].\] 
By \cite[(42)]{LV-MM}, 
$\Theta(\sigma)=\langle\vartheta(\sigma)\rangle$ for all 
$\sigma\in\Gal(L_{cp^{n+m},m}/H_{cp^{n+m}})$, as endomorphisms of $(e_{k_0-2}\cdot e^\ord)\cdot{J}_m(L_{cp^{n+m},m})$, and therefore, using that $U_p$ has degree $p$ (\emph{cf.} \cite[\S6.2]{LV-MM}), projecting to the ordinary submodule gives 
points 
\[\mathbf{P}_{cp^{n+m},m}=(e_{k_0-2}\cdot e^\ord)\cdot  P_{cp^{n+m,m}}
\in H^0(H_{cp^{n+m}},{J}_m^{\dagger,\ord}(L_{cp^{n+m},m})),\] 
where ${J}_m^{\dagger,\ord}(L)=e^\ord\cdot{J}_m^\dagger(L)$ for any extension $L/\Q$, and for any $\Gal(\overline\Q/\Q)$-module $M$, 
we denote $M^\dagger$ the Galois module $M\otimes\Theta^{-1}$ as before. Corestricting from $H_{cp^{n+m}}$ to $H_{cp^n}$, we obtain classes
\[\mathcal{P}_{cp^n,m}\in H^0(H_{cp^{n}},{J}_m^{\dagger,\ord}(L_{cp^{n+m},m})).\] 
Composing the (twisted) Kummer map we obtain classes 
$\mathfrak{X}_{cp^n,m}$ in $H^1(H_{cp^n},\Ta_p^\ord({J}_m)^\dagger)$ (where $\Ta_p^\ord({J}_m)=e^\ord\Ta_p({J}_m)$) and then, 
using the trace-compatibility properties enjoyed by the collection of points ${P}_{cp^{n+m},m}$ recalled in \S\ref{CM section}, we may define a class \[\mathfrak{X}_{cp^n}=\invlim_m U_p^{-m}\mathfrak{X}_{cp^n,m}\in H^1(H_{cp^n},\mathbf{T}^\dagger).\] 
Under the assumption that $p$ does not divide the class number of $K$, 
using the properties of the points ${P}_{cp^{n+m},m}$ once again, 
we may also define Iwasawa classes 
\[\mathfrak{X}_{cp^\infty}=\invlim_n U_p^{-n}\mathfrak{X}_{cp^n}\in H^1_\mathrm{Iw}(H_{cp^\infty}/H_c,\mathbf{T}^\dagger)\defeq\invlim_{n\geq 0}H^1(H_{cp^n},\mathbf{T}^\dagger),\]
where the inverse limit is taken with respect to the corestriction maps. Since $\mathfrak{P}$ in totally ramified in the extension $H_{cp^\infty}/H_c$, we have $\Gal(H_{cp^\infty}/H_c)\cong\Gamma_\infty$, so we can write 
\[\mathfrak{X}_{cp^\infty}\in H^1_\mathrm{Iw}(\Gamma_\infty,\mathbf{T}^\dagger).\] 
We may thus consider the class 
\[\mathfrak{X}_{\boldsymbol{\xi}}\defeq\mathfrak{X}_{cp^\infty}\otimes \boldsymbol{\xi}^{-1}\in 
H^1_{\mathrm{Iw}}(\Gamma_\infty,\mathbf{T}^{\dagger}_{\boldsymbol{\xi}}).\] Let $\widetilde{\Gamma}_\infty=\Gal(H_{cp^\infty}/K)$. Taking corestriction we get a class
\begin{equation}\label{def=Xi}
 \mathfrak{Z}_{\boldsymbol{\xi}}=\cor_{H_c/K}\left(\mathfrak{X}_{\boldsymbol{\xi}}\right)\in 
H^1_{\mathrm{Iw}}(\widetilde{\Gamma}_\infty,\mathbf{T}^{\dagger}_{\boldsymbol{\xi}})\defeq \invlim_{n\geq -1}H^1(H_{cp^n},\mathbf{T}_{\boldsymbol{\xi}}^\dagger)\end{equation} where $H_{cp^{-1}}\defeq K$. Under the condition that 
the residual Galois representation $\bar\rho$ attached to the Hida family $\I$ is ramified at all primes dividing $N^-$, one can prove that $ \mathfrak{X}_{cp^n}$ belongs to the Greenberg Selmer group (see \cite[Proposition 4.5]{CW}). 

\subsection{The algebraic $p$-adic $L$-function}\label{sec8.3}
Recall the big Perrin-Riou map 
$\mathcal{L}^{{\Gamma}_\infty}_{\omega_{\I}}$ in \eqref{PR} 
and define 
$\mathcal{L}^{{\widetilde\Gamma}_\infty}_{\omega_{\I}}=\varrho\circ\mathcal{L}^{{\Gamma}_\infty}_{\omega_{\I}}$, where $\varrho\colon \widetilde{\I}[\mathfrak{j}^{-1}][[\Gamma_\infty]]\rightarrow \widetilde{\I}[\mathfrak{j}^{-1}][[\widetilde\Gamma_\infty]]$ is the map arising from the canonical map 
$\Gamma_\infty\hookrightarrow\widetilde{\Gamma}_\infty$. 
Since $p$ is split in $K$,
$\res_{\mathfrak{P}}(\mathfrak{X}_{\boldsymbol{\xi}})$ belongs to 
$H^1_\mathrm{Iw}(\Gamma_\infty,\mathbf{T}^{\dagger,+}_{\boldsymbol{\xi}})$ by \cite[Proposition 2.4.5]{Howard-Inv}, so the following definition make sense. 

\begin{definition}\label{geodef} 
${\mathscr{L}}_{\I,\boldsymbol{\xi}}^{\mathrm{alg}}=
\mathcal{L}^{{\widetilde\Gamma}_\infty}_{\omega_{\I}}\left(\res_\mathfrak{P}(\mathfrak{Z}_{\boldsymbol{\xi}})\right)$
is the \emph{algebraic} anticyclotomic $p$-adic $L$-function attached to the family $\I$.
\end{definition}

 \section{Coleman integration on Shimura curves} 

 The results of this section are generalizations to the Shimura curves setting of results available for modular curves. The proofs are the same, and we reproduce them for lacking of precise references. The only new input is the use of Serre--Tate coordinates to normalize che choice of Coleman primitives. 

 
 \subsection{Rigid analytic Shimura curves} Let ${\mathcal{X}}_m$ be the $\Z_p$-model of $X_m$, for integers $m\geq 0$, and denote by $ X_m^\rig$ the rigid analytic space over $\Q_p$ associated with ${\mathcal{X}}_m$.
 Let $\mathbf{Ha}$ be the Hasse invariant of the special fiber $\mathbb{X}_0$ of $\mathcal{X}_0$, and let $\widetilde{\mathbf{Ha}}$ be a lift of $\mathbf{Ha}$ to $\mathcal{X}_0$ (\cite[\S7]{Kassaei}). Then 
$\mathcal{X}_0^\ord=\mathcal{X}_0[1/\widetilde{\mathbf{Ha}}]$ 
is an affine open $\Z_{p}$-subscheme of $\mathcal{X}_0$ representing  the moduli problem which associates to any $\Z_{p}$-scheme $S$ 
the isomorphism classes of triplets $(A,\iota,\alpha)$ where $(A,\iota)$ is an ordinary quaternionic multiplication abelian surface over $S$ 
and $\alpha$ a na\"{\i}ve level $V_0(N^+)$ structure. 
Denote ${X}_0^{\ord,\rig}$ the rigid analytic space associated with $\mathcal{X}_0^\ord$, which is the complement in $X_0^\rig$ of residue disks $D_x$ corresponding to supersingular points $x$ in the special fiber of $\mathcal{X}_0^\ord$ (we refer \emph{e.g.} to \cite[\S3]{Buz} for the notion of \emph{supersingular} abelian surface with quaternionic multiplication). 

For any real number $0\leq \varepsilon<1$, denote $\mathcal{X}_0^-(\varepsilon)$ the open rigid analytic subspace of ${X}_0^\rig$ defined by the condition $|\widetilde{\mathbf{Ha}}|>|p|^{\varepsilon}$; we view $\mathcal{X}_0^-(\varepsilon)$ as defined over any field extension $L/\Q_p$ in which there exists an element $x\in L$ with 
$|x|=|p|^\varepsilon$. 
For any integer $m\geq 1$, let $\varepsilon_m= \frac{1}{p^{m-2}(p-1)}$; then 
$\mathcal{X}_0^-(\varepsilon_m)$ is defined over $\Q_p(\zeta_{p^m})$, and later we will adopt the same symbol for their base change to finite field extensions $L$ of the cyclotomic field $\Q_p(\zeta_{p^m})$. 
By \cite[Proposition 6.30]{Brasca-Eigen2}, any point $x=(A,\iota,\alpha)$ in $\mathcal{X}_0^-(\varepsilon_m)$ 
admits a \emph{canonical subgroup} $C_{p^m}\subseteq A[p^m]$ of order $p^{2m}$ (see \cite[\S3]{Brasca-Eigen2} for the notion of canonical subgroup in this setting; see also \cite[\S10]{Kassaei} and \cite[\S3.2]{Scholze3} for related results).

Let $\digamma_m\colon {\mathcal{X}}_m\rightarrow\mathcal{X}_0$ denote the forgetful map. Define 
$\mathcal{W}_1(p^m)$ (respectively, $\mathcal{W}_2(p^m)$) to be the open rigid analytic subspace of  
${{X}}_m^\rig$ whose closed points corresponds to QM abelian surfaces with level structure 
$x=(A,\iota,\alpha,\beta) $ where:
\begin{itemize}
\item $(A,\iota)$ is a QM abelian surface equipped with a $V_0(N^+)$-structure $\alpha$; 
\item $\beta\colon\boldsymbol{\mu}_{p^m}\rightarrow eC_{p^m}$ is an isomorphism, where $\boldsymbol{\mu}_{p^m}$ is teh group of $p^m$-th roots of unity and, as before, we indicate $C_{p^m}\subseteq A[p^m]$ the 
canonical subgroup of $A$ of order $p^{2m}$; thus, $\beta(\zeta_{p^m})$ is a generator of $eC_{p^m}$; 
\item $\digamma_m(x)$ belongs to $\mathcal{X}_0^-(\varepsilon_m)$  
(respectively, $\mathcal{X}_0^-(\varepsilon_{m+1})$).
\end{itemize} 
We thus have a chain of inclusions of rigid analytic spaces 
$\mathcal{W}_1(p^m)\subseteq\mathcal{W}_2(p^m)\subseteq X_m^\rig.$ 

The \emph{Deligne--Tate map} $\phi_{\mathrm{DT}}\colon\mathcal{X}_0^-(\varepsilon_{m+1}) \rightarrow \mathcal{X}^-_0(\varepsilon_{m})$    is defined by taking quotients by the canonical subgroup $C_p\subseteq A[p]$ of order $p^2$, \emph{i.e.} we put $\phi_{\mathrm{DT}}(A,\iota,\alpha)=(A_0,\iota_0,\alpha_0)$ where $A_0=A/C_p$, and if 
$\varphi\colon A\rightarrow A/C_p$ is the canonical isogeny, 
$\iota_0$ is the polarization induced by $\iota$ and $\varphi$, 
and $\alpha_0$ 
is the level $V_1(N^+)$ structure induced by $\alpha$ and $\varphi$. The map $\bar\phi_{\mathrm{DT}}$ induced by 
$\phi_{\mathrm{DT}}$ on the special fibers of 
$\mathcal{X}_0^-(\varepsilon_{m+1})$ and $\mathcal{X}_0^-(\varepsilon_{m})$ coincides with the Frobenius map $\Frob_p$, and so $\phi_{\mathrm{DT}}\colon\mathcal{X}_0^-(\varepsilon_{m+1}) \rightarrow \mathcal{X}^-_0(\varepsilon_{m})$ is also known as \emph{Frobenius map}. The map $\phi_{\mathrm{DT}}$ thus obtained can be lifted to a map (denoted with the same symbol and also called \emph{Frobenius map})  
\begin{equation}\label{DTFrob}
\phi_{\mathrm{DT}}\colon\mathcal{W}_2(p^m)\longrightarrow \mathcal{W}_1(p^m)\end{equation}
setting 
$\phi_{\mathrm{DT}}(A,\iota,\alpha,\beta)=(A_0,\iota_0,\alpha_0,\beta_0)$ where $\beta_0\colon\mu_{p^m}\rightarrow A/C_p$ sends $\zeta_{p^m}$ to $\varphi(P_{m+1})$ where $P_{m+1}\in C_{p^{m+1}}$ satisfies $pP_{m+1}=P_m=\beta(\zeta_{p^m})$. 

 \subsection{Rigid de Rham cohomology}\label{sec:9.1.2}
We denote by $\widetilde{\mathcal{X}}_m$ the proper, flat, regular balanced model of ${\mathcal{X}}_m$ over $\Z[\zeta_{p^m}]$. The special fiber of $\widetilde{\mathcal{X}}_m$ is the union of a finite number of reduced Igusa curves over $\F_p$, meeting at their supersingular points, and two of these components, denoted $\mathrm{Ig}_\infty$ and $\mathrm{Ig}_0$, are isomorphic to the Igusa curve $\mathrm{Ig}_{m,1}$ of level $m$ over $\F_p$; we let $\mathrm{Ig}_\infty$ denote the connected component corresponding to the canonical inclusion of $\mathrm{Ig}_{m,1}$ into $\widetilde{\mathcal{X}}_m\otimes_{\Z[\mu_{p^m}]}\F_p$. We have an involution $w_{\zeta_{p^m}}$ attached to the chosen $p^n$-root of unity $\zeta_{p^m}$ which interchanges the two components $\mathrm{Ig}_\infty$ and $\mathrm{Ig}_0$ 
(see \cite{Morita} and its generalization to totally real fields in \cite{Carayol-Shimura}). 

\begin{remark} 
The results of Carayol \cite{Carayol-Shimura} 
formally exclude the case under consideration when the fixed totally real number field $F$ is equal to $\Q$, but refers to the paper of Morita \cite{Morita} for this case. A proof of these facts can also be obtained by a direct generalization of the arguments in \cite[Theorem 4.10]{Buz} which considers the case $m=1$. 
\end{remark}

Let $L$ be a finite extension of $\Q_p(\zeta_{p^m})$ where ${\mathcal{X}}_m$ acquires semistable reduction. 
Let $\mathcal{O}_L$ be the valuation ring of $L$ and $k_L$ its residue field. We denote $\varpi\colon\mathscr{X}_m\rightarrow {\mathcal{X}}_m\otimes_{\Z_p[\zeta_{p^m}]}\mathcal{O}_L$ a semistable model of 
$\widetilde{\mathcal{X}}_m$ over $\mathcal{O}_L$. 
Let $\mathcal{G}_m$ denote the dual graph of the special fiber $\mathbb{X}_m$ of $\mathscr{X}_m$; the set $\mathcal{V}(\mathcal{G}_m)$ of vertices of $\mathcal{G}_m$ 
is in bijection with the irreducible components of the special fiber $\mathbb{X}_m$ of $\mathscr{X}_m$, 
and the set $\mathcal{E}(\mathcal{G}_m)$ of oriented edges of $\mathcal{G}$ is in bijection with the singular points of
$\mathbb{X}_m$, together with an ordering of the two components which intersect
at that point. Given $v\in \mathcal{V}(\mathcal{G}_m)$, let $\mathbb{X}_v$ 
denote the associated component in $\mathbb{X}_m$, and let
$\mathbb{X}_v^\mathrm{sm}$ denote the smooth locus of $\mathbb{X}_v$. Let 
$\mathrm{red}\colon\mathscr{X}_m(\C_p)\rightarrow \mathbb{X}_m(\overline{\F}_p)$ be the canonical reduction map. For any $v\in\mathcal{V}(\mathcal{G}_m)$, let $\mathcal{W}_v=\mathrm{red}^{-1}(\mathbb{X}_v(\overline{\F}_p))$ denote the wide open space associated with the connected component $\mathbb{X}_v$, and let $\mathcal{A}_v=\mathrm{red}^{-1}(\mathbb{X}_v^\mathrm{sm}(\overline{\F}_p))$ denote the underlying affinoid $\mathcal{A}_v\subseteq\mathcal{W}_v$. If $e=(s(e),t(e))\in\mathcal{E}(\mathcal{G}_m)$ is a edge, then $\mathcal{W}_e=\mathcal{W}_{s(e)}\cap \mathcal{W}_{t(e)}$ is equal to $\mathrm{red}^{-1}(\{x_e\})$, where $\{x_e\}=\mathbb{X}_{s(e)}\cap\mathbb{X}_{t(e)}$. The set $\{\mathcal{W}_v:v\in\mathcal{V}(\mathcal{G}_m)\}$ form an \emph{admissible cover} of the rigid analytic space $\mathscr{X}_m(\C_p)=\widetilde{X}_m(\C_p)$ by wide open subsets. Let $d\colon\mathcal{O}(\mathcal{U})\rightarrow \Omega^1_{\rig}(\mathcal{U})$ be the differential map for any wide open $\mathcal{U}$, where $\mathcal{O}=\mathcal{O}^\rig_{\mathscr{X}_m}$ is the sheaf of rigid analytic functions on $\mathscr{X}_m$ and $\Omega^1_\rig$ the sheaf of rigid 1-forms; the de Rham cohomology group can be described as the set of hyper-cocycles \[\omega=(\{\omega_v\}_{v\in\mathcal{V}(\mathcal{G}_m)},\{f_e\}_{e\in\mathcal{E}(\mathcal{G}_m)})\in\prod_{v\in\mathcal{V}(\mathcal{G}_m)}\Omega^1(\mathcal{W}_v)\times \prod_{ e\in\mathcal{E}(\mathcal{G}_m)}\mathcal{O}_{\mathcal{W}_e}\] such that $df_e=\omega_{t(e)}-\omega_{s(e)}$ and $f_{\bar{e}}=-f_e$ for each $e=(s(e),t(e))\in\mathcal{E}(\mathcal{G}_m)$ (where $\bar{e}$ denotes the edge with same vertices of $e$ in the reverse orientation) modulo hyper-coboundaries, which are elements of the form $(df_v,f_{t(e)}-f_{s(e)})$ for a set $\{f_v\}_{v\in\mathcal{V}(\mathcal{G}_m)}$ of functions $f_v\in\mathcal{O}_{\mathcal{W}_v}$. For each edge $e=(s(e),t(e))$, we have an \emph{annular residue map} $\mathrm{res}_{\mathcal{W}_e}\colon\Omega^1_{\mathscr{X}_m}(\mathcal{W}_e)\rightarrow\C_p$ defined by expanding a differential form $\omega\in \Omega^1_{\mathscr{X}_m}(\mathcal{W}_e)$ as $\omega=\sum_{n\in\Z}a_nt^n dt$ for a fixed uniformizing parameter $t$ on $\mathcal{W}_e$ and setting $\mathrm{res}_{\mathcal{W}_e}(\omega)=a_{-1}$. 
We say that a class $\omega\in H^1_\mathrm{dR}(\mathscr{X}_m)$ is \emph{pure} if for every edge $e\in\mathcal{E}(\mathcal{G}_m)$, $\res_{\mathcal{W}_e}(\omega_{s(e)})$ vanishes. For pure classes $\omega=(\omega_v,f_e)$, $\eta=(\eta_v,g_e)$ the de Rham pairing $\langle \omega,\eta\rangle_\dR$ is computed by the formula 
\begin{equation}\label{dRpairing}\langle \omega,\eta\rangle_\dR=\sum_{e=(s(e),t(e))\in\mathcal{E}(\mathcal{G}_m)}\mathrm{res}_{\mathcal{W}_{e}}(F_e\eta_{s(e)})\end{equation} where $F_e$ is an analytic primitive of the restriction to $\mathcal{W}_e$ of $\omega_{s(e)}$, which exists because $\omega_v$ has vanishing annular residues at $\mathcal{W}_e$, and is well defined up to a constant (and since $\eta_v$ has also vanishing annular residue at $v$, the value of the pairing is independent of this choice). See \cite[\S3.5]{CI} (or \cite[\S3.1]{DR}) for more details. 

The birational map $\varpi:\mathscr{X}_m\rightarrow \widetilde{\mathcal{X}}_m\otimes_{\Z_p[\zeta_{p^m}]}\mathcal{O}_L$ induces an isomorphism between the generic fibers; it also induces an isomorphism between two of the components of the special fiber $\mathbb{X}_m$ of $\mathscr{X}_m$ 
with $\mathrm{Ig}_\infty\otimes_{\F_p}k_L$ and $\mathrm{Ig}_0\otimes_{\F_p}k_L$: we denote by $\mathtt{Ig}_\infty$ and $\mathtt{Ig}_0$ these two components of $\mathbb{X}_m$.  
Let $\mathcal{W}_\infty(p^m)=\mathrm{red}^{-1}(\mathtt{Ig}_\infty)$ and $\mathcal{W}_0(p^m)=\mathrm{red}^{-1}(\mathtt{Ig}_0)$ be the corresponding wide open subsets with associated underlying affinoids $\mathcal{A}_\infty(p^m)$ and $\mathcal{A}_0(p^m)$, respectively. The $L$-valued points of the rigid anaytic space $\mathcal{A}_\infty(p^m)$ are in bijection with quadruplets $(A,\iota,\alpha,\beta) $ where $(A,\iota)$ is a QM abelian surface, $\alpha$ is a level $V_0(N^+)$ structure and $\beta\colon\boldsymbol{\mu}_{p^m}\rightarrow eC_{p^m}$ is an isomorphism (as before, $C_{m}\subseteq A[p^m]$ indicates the 
canonical subgroup of $A$ of order $p^{2m}$).  
The $L$-vector spaces 
\[H^1_\mathrm{rig}(\mathcal{W}_\infty(p^m))=\frac{\Omega^1_\mathrm{rig}(\mathcal{W}_\infty(p^m))}{d\mathcal{O}_{\mathcal{W}_\infty(p^m)}}\quad\text{and}\quad H^1_\mathrm{rig}(\mathcal{W}_0(p^m))=\frac{\Omega^1_\mathrm{rig}(\mathcal{W}_0(p^m))}{d\mathcal{O}_{\mathcal{W}_0(p^m)}}\] are equipped with a canonical action of Hecke operators $T_\ell$ for primes $\ell\nmid Np$, and with canonical $L$-linear Frobenius endomorphisms defined by choosing characteristic zero lifts $\Phi_\infty$ and $\Phi_0$ of the Frobenius endomorphism in characteristic $p$ to a system of wide open neighborhoods of the affinoids $\mathcal{A}_\infty(p^m)$ in $\mathcal{W}_\infty(p^m)$ and  $\mathcal{A}_0(p^m)$ in $\mathcal{W}_0(p^m)$, respectively.
In the case of Shimura curves, we take $\Phi_\infty=\phi_{\mathrm{DT}}$ and $\Phi_0=\tilde{\phi}_{\mathrm{DT}}\defeq w_{\zeta_{p^m}}\circ\phi_{\mathrm{DT}}\circ w_{\zeta_{p^m}}$, where $w_{\zeta_{p^m}}$ is the Atkin--Lehner involution associated with the choice of $\zeta_{p^m}$ which interchanges the two wide opens $\mathcal{W}_\infty(p^m)$ and $\mathcal{W}_0(p^m)$. Let 
\[\mathrm{res}_\mathcal{W}\colon H^1_\mathrm{dR}(\mathscr{X}_m)\longrightarrow H^1_\mathrm{rig}(\mathcal{W})=\frac{\Omega^1_\mathrm{rig}(\mathcal{W})}{d\mathcal{O}_{\mathcal{W}}}\] 
be the restriction map, where $\mathcal{W}$ is an admissible wide open space obtained as inverse image via the reduction map of an irreducible component of the special fiber of $\widetilde{\mathcal{X}}_m$; in particular we have the two maps $\mathrm{res}_\infty=\mathrm{res}_{\mathcal{W}_\infty(p^m)}$ 
and $\mathrm{res}_0=\mathrm{res}_{\mathcal{W}_0(p^m)}$. Let $H^1_\mathrm{dR}(\mathscr{X}_m)^\mathrm{prim}$ be the subspace of the de Rham cohomology of $\mathscr{X}_m$ associated with the primitive subspace of the $L$-vector space of modular forms of weight $2$ and level $N^+p^m$, and $H^1_\mathrm{rig}(\mathcal{W})^\mathrm{pure}$ is the subspace generated by pure classes of rigid differentials (\emph{i.e.} those classes with vanishing annular residues, as before), for $\mathcal{W}=\mathcal{W}_\infty(p^m)$ and $\mathcal{W}=\mathcal{W}_0(p^m)$.  

\begin{proposition}\label{EB}
The restriction maps $\mathrm{res}_\infty$ and $\mathrm{res}_0$ induce an isomorphism 
of $L$-vector spaces 
\[\res=\res_\infty\oplus\res_0:H^1_\mathrm{dR}(\mathscr{X}_m)^\mathrm{prim}\simeq H^1_\mathrm{rig}(\mathcal{W}_\infty(p^m))^\mathrm{pure}\oplus H^1_\mathrm{rig}(\mathcal{W}_0(p^m))^\mathrm{pure}\]
which is equivariant with respect to the action of Hecke operators $T_\ell$ for $\ell\nmid Np$ on both sides, the crystalline Frobenius endomorphism, denoted by $\Phi$, acting on the LHS and the Frobenius endomorphism $(\phi_{\mathrm{DT}},\widetilde{\phi}_\mathrm{DT})$ acting on the RHS. 
\end{proposition}

\begin{proof}
The proof of these results can be obtained as in \cite[\S4.4]{EmBr} using a generalization of  \cite[Theorem 2.1]{Colemanhigher} to the case of Shimura curves. This generalization does not present technical difficulties and is left to the interested reader. 
\end{proof}

Fix a finite set of points $S$ of $\mathscr{X}_m(\C_p)$ which reduce to smooth points in $\mathbb{X}_m(\overline{\F}_p)$. The residue disk $D_Q$ of each $Q\in S$ (defined as the set of points of $\mathscr{X}_m(\C_p)$ whose reduction is equal to the reduction of $Q$) is conformal to the open unit disk $D\subseteq\C_p$ because $\mathrm{red}(Q)$ is smooth, and we may fix an isomorphism $\varphi_Q\colon D_Q\overset\sim\rightarrow D$ of rigid analytic space which takes $Q$ to $0$. For each $Q\in S$, fix a real number $r_Q<1$ which belongs to the set $\{|p|^m:m\in\Q\}$. 
Let $\mathcal{V}_{Q}\subseteq D_Q$ be the annulus consisting of points $x\in D_Q$ such that $r_Q<|\varphi_Q(x)|_p<1$; define the \emph{orientation} of $\mathcal{V}_Q$ by choosing the subset $\{x\in D_Q: |\varphi_Q(x)|_p\leq r_Q\}$ of the set $D_Q-\mathcal{V}_Q$, which consists in two connected components. We may then consider the
affinoid 
\[\mathcal{A}_S=\mathscr{X}_m(\C_p)-\bigcup_{Q\in S}D_Q\] and the 
 wide open neighborhood
\[\mathcal{W}_S=\mathcal{A}_S\cup \bigcup_{Q\in S}\mathcal{V}_Q\]
of $\mathcal{A}_S$, so that $\mathcal{A}_S$ is the underlying affinoid of $\mathcal{W}_S$. We also put 
\[\widetilde{\mathcal{W}}_\infty= \mathcal{W}_\infty(p^m)-\bigcup_{Q\in S}(D_Q-\mathcal{V}_Q)\quad\text{and}\quad\widetilde{\mathcal{W}}_0= \mathcal{W}_0(p^m)-\bigcup_{Q\in S}(D_Q-\mathcal{V}_Q).\]
For a Hecke module $M$, denote by $M[\mathcal{F}]$ the eigencomponent corresponding to an eigenform $\mathcal{F}$. 
Let $Y_S=\mathscr{X}_m-S$ and let $\mathcal{F}\in S_2(\Gamma_m,L)$ be a weight $2$ newform on ${X}_m$. 
An excision argument from Proposition \ref{EB}
shows that the canonical restriction map $\res=(\res_0,\res_\infty)$ induces an isomorphism 
\begin{equation}\label{eq:9.3} 
\res\colon H^1_\dR(Y_S/L)[\mathcal{F}]\overset\simeq\longrightarrow H^1_\rig(\widetilde{\mathcal{W}}_\infty)[\mathcal{F}]\oplus H^1_\rig(\widetilde{\mathcal{W}}_0)[\mathcal{F}].\end{equation} Moreover, again from  Proposition \ref{EB}, a class in $H^1_\dR(Y_S/L)[\mathcal{F}]$ is the restriction of a class of $H^1_\dR(\mathscr{X}_m/L)$ if and only if it can be represented by a pair of differentials $\tilde\omega_\infty\in \Omega^1_{\rig}(\widetilde{\mathcal{W}}_\infty)$ and $\tilde\omega_0\in \Omega^1_{\rig}(\widetilde{\mathcal{W}}_0)$
with vanishing annular residues. 
If $\omega$ and $\eta$ are classes in 
$H^1_\mathrm{dR}(\mathscr{X}_m)^\mathrm{prim}$, set $\omega_\infty=\mathrm{res}_\infty(\omega)$, $\omega_0=\mathrm{res}_0(\omega)$, $\eta_\infty=\mathrm{res}_\infty(\eta)$, $\eta_0=\mathrm{res}_0(\eta)$. 
Let $F_{\infty|\mathcal{V}_Q}$ be any solution of the differential equation $dF=\omega_{\infty}$ on $\mathcal{V}_Q$, and let $F_{0|\mathcal{V}_Q}$ be any solution of the differential equation $dF=\omega_{0}$ on $\mathcal{V}_Q$.
It follows from \eqref{dRpairing} 
that for each $\omega,\eta\in H^1_\mathrm{dR}(\mathscr{X}_m)[\mathcal{F}]$ we have 
\begin{equation}\label{pairing-residues} 
\langle\eta,\omega\rangle_\dR= \sum_{\mathcal{V}\subseteq\widetilde{\mathcal{W}}_\infty}\mathrm{res}_\mathcal{V}(F_{\infty|\mathcal{V}}\cdot \eta_{\infty|\mathcal{V}})+ 
\sum_{\mathcal{V}\subseteq\widetilde{\mathcal{W}}_0}\mathrm{res}_\mathcal{V}(F_{0|\mathcal{V}}\cdot \eta_{0|\mathcal{V}})\end{equation}
where the sum is over all annuli $\mathcal{V}$. 

\subsection{Coleman primitives}
Let $x=({A},\iota,\alpha,\beta)$ be a point of ${\mathcal{X}}_m(\mathcal{O}_{\C_p})$ which reduces to a smooth point $\bar{x}=(\bar{A},\bar\iota,\bar\alpha)$ in the special fiber of $\mathcal{X}_0$.
We assume that $A$ is ordinary, and $\beta\colon\boldsymbol{\mu}_{p^m}\rightarrow eA[p^m]^0$ is a \emph{trivialization}.
Let $\mathcal{R}_{\bar{x}}$ be the universal quaternionic deformation ring of $\bar{x}$ and let $\mathcal{A}_{\bar{x}}\rightarrow \Spec(\mathcal{R}_{\bar{x}})$ be the universal quaternionic multiplication abelian surface.  
Fix a $\Z_p$-basis $\{x_{\bar{A}},x_{\bar{A}}'\}$ of $\Ta_p(\bar{A})$ such that $x_{\bar{A}}$ is a $\Z_p$-basis of $e\Ta_p({\bar{A}})$ and $ex_{\bar{A}}'=0$ and let $T_{\bar{x}}$ be the associate Serre--Tate coordinate. 
We consider the formal differential form 
$\omega_{\bar{x}}$ obtained by pulling-back $dT/T$ along the map $\widehat{\mathcal{A}}_{\bar x}\rightarrow \G_m$, where $\widehat{\mathcal{A}}_{\bar x}$ is the formal group  of 
${\mathcal{A}}_{\bar x}$. 
Let $D_x$ be the residue disk of $\bar x$ in ${\mathcal{X}}_m$, defined to be the set of points of the associated rigid analytic space whose reduction is equal to $\bar{x}$.  
Using the Serre--Tate coordinates around $A$ associated with the choice of the basis $\{x_A,x_A'\}$, for $\mathcal{F}\in S_2(\Gamma_m,\mathcal{O}_{\C_p})$ we may write on $D_x$ the differential form associated with $\mathcal{F}$ as 
\begin{equation}\label{diff}
\omega_\mathcal{F}=\mathcal{F}(T_x)\omega_x.\end{equation}
Let   $D_x^\phi=\phi_{\mathrm{DT}}(D_x)$ be the residue disk in ${\mathcal{X}}_m$ of $\bar\phi_{\mathrm{DT}}(x)$.  

Let $(A,\iota,\alpha,\beta)$ be a point in $D_x$. 
The operator $V$ is described by the formula
\[V\mathcal{F}(A,\iota,\alpha,\beta)=\mathcal{F}(A_0,\iota_0,\alpha_0,\beta_0)\] where:
\begin{itemize}
\item  
$A_0=A/C_p$ is the quotient by the canonical subgroup, and $\iota_0$ is induced by the projection map $\pi\colon A\rightarrow A_0$ from $\iota$;  
\item $\pi$ and the dual isogeny $\pi^\vee\colon A_0^\vee\rightarrow A^\vee$ induce isomorphisms between $A[N^+]$ and $A_0[N^+]$, and 
we may define $\alpha_0=(\pi^\vee)^{-1}\circ\alpha$ (here, we view a $V_1(N^+)$-level structure 
as represented by a homomorphism of group schemes $\alpha\colon(\Z/N^+\Z)^2\rightarrow eA[N^+]$); 
\item The dual isogeny $\pi^\vee$ is \'etale, so it induces an isomorphism on formal completions; composing 
with the principal polarizations of $A$ and $A_0$, we obtain an isomorphism, still denoted $\pi^\vee\colon A_0[p^m]^0\rightarrow A[p^m]^0$, 
and define a trivialization $\beta_0\colon\boldsymbol{\mu}_{p^m}\rightarrow eA_0[p^m]^0$ by the equation $\beta_0^{-1}=\beta^{-1}\circ\pi^\vee$. 
\end{itemize}

\begin{lemma}\label{lemma9.3} 
$\phi_{\mathrm{DT}}^*(\omega_{\mathcal{F}})=p\omega_{V\mathcal{F}}$. 
\end{lemma}

\begin{proof}
From the definition of $V$ and \eqref{diff} 
we have 
\[\phi_{\mathrm{DT}}^*(\omega_\mathcal{F}) 
=(V\mathcal{F})(T_x)\phi_{\mathrm{DT}}^*(\omega_{\bar\phi_{\mathrm{DT}} (x)}).\] On the other hand, $\phi_{\mathrm{DT}}^*\big(\omega_{\bar\phi_{\mathrm{DT}} (x)}\big)=p\omega_x$ by \cite[Lemma 3.5.1]{Kat} (see also \cite[Lemmas 4.4, 4.11]{Brooks}), concluding the proof.  
\end{proof}

Let $a_p$ denote the $U_p$-eigenvalue of $\mathcal{F}$ and  
define the polynomial $\Pi(X)=1-\frac{a_p}{p}X$.

\begin{proposition}\label{Col1}
\begin{enumerate} 
\item There exists a locally analytic function $F_{\infty}$ on $\mathcal{W}_\infty(p^m)$, unique up to a constant, such that $dF_{\infty}=\omega_\mathcal{F}$ on $\mathcal{W}_\infty(p^m)$ and  
$\Pi(\phi_{\mathrm{DT}}^*)F_{\infty}$ is a rigid analytic function on a wide-open neighborhood $\mathcal{W}_\infty$ of 
$\mathcal{A}_\infty(p^m)$ contained in $\mathcal{W}_\infty(p^m)$.  
\item There exists a locally analytic function $F_0$ on 
$\mathcal{W}_0(p^m)$, unique up to a constant, such that $dF_0=\omega_\mathcal{F}$ on $\mathcal{W}_0(p^m)$ and $\Pi(\tilde\phi_{\mathrm{DT}}^*)F_0$ is a rigid analytic function on a wide-open neighborhood $\mathcal{W}_0$ of $w_{\zeta_{p^m}}\widetilde X_m(0)$ in $\mathcal{W}_0(p^m)$.\end{enumerate}
\end{proposition} 

\begin{proof} (1) In  
${\mathcal{W}}_\infty=\phi_{\mathrm{DT}}^{-1}(\mathcal{W}_\infty(p^m)\cap\mathcal{W}_1(p^m))$ we have $\Pi(\phi_{\mathrm{DT}}^*)\omega_f=0$ by Lemma \ref{lemma9.3}; moreover, $\Pi(\phi_{\mathrm{DT}}^*)$ induces an isomorphism of the sheaf of locally analytic functions on $\mathcal{W}_\infty(p^m)$ because  the (complex) absolute value of $a_p$ is $p^{1/2}$. Then (1) follows from \cite[Theorem 8.1]{Coleman-Shimura}, using \cite[Proposition 3.1.2]{Katz-Dwork} (see also \cite[Lemma 5.1]{CI}) to check the condition on regular singular annuli. For (2), apply (1) to $w_{\zeta_{p^m}}\omega_\mathcal{F}$. \end{proof}

\begin{definition}\label{defCol}
The functions $F_{\infty}$ and $F_{0}$ in Proposition \ref{Col1} are the \emph{Coleman primitives} of $\mathcal{F}$ on $\mathcal{W}_\infty$ and $\mathcal{W}_0$, respectively.  
\end{definition} 

Note that (1) of Proposition \ref{Col1} says that 
$\Pi(\phi_{\mathrm{DT}}^*)F_{\infty}$ is overconvergent. More precisely, 
for any integer $m\geq1$ and any real number $0\leq \varepsilon<\varepsilon_m$, let $\mathcal{X}_m(\varepsilon)$ denote the affinoid subdomain of $\mathcal{W}_1(p^m)$ consisting of those points $x$ such that $|\widetilde{\mathbf{Ha}}(\digamma_m(x))|\geq |p|^{\varepsilon}$; to complete the notation, when 
$m=0$ and $0\leq \varepsilon<1$, we also denote $\mathcal{X}_0(\varepsilon)$ the affinoid subdomain of ${X}_0^\rig$ defined by the condition $|\widetilde{\mathbf{Ha}}|\geq |p|^{\varepsilon}$, so that $\mathcal{X}_0^-(\varepsilon)\subseteq \mathcal{X}_0(\varepsilon)$. 
For any integer $k$ and any integer $m\geq 0$,  
define the $\C_p$-vector space of \emph{overconvergent modular forms} of weight $k$ on ${X}_m$ to be 
\[S_k^\mathrm{oc}({X}_m)=\invlim_{\varepsilon} H^0(\mathcal{X}_m(\varepsilon),\omega_{m,\C_p}^{\otimes k})\]
where $0\leq\varepsilon<\varepsilon_m$ with $\varepsilon$ approaching $\varepsilon_m$.  
Then we have $\Pi(\phi_{\mathrm{DT}}^*)F_{\infty}\in  S_k^\mathrm{oc}({X}_m)$. 

The proof of \cite[Theorem 10.1]{Coleman-Shimura} shows that 
 $d(\Pi(\phi_{\mathrm{DT}}^*)(F_\infty))=\Pi(\phi_{\mathrm{DT}}^*)\omega_\mathcal{F}$; on the other hand,  
 $\Pi(\phi_{\mathrm{DT}}^*)\omega_\mathcal{F}=\omega_{\mathcal{F}^{[p]}}$, where recall that $\mathcal{F}^{[p]}=(1-U_pV)\mathcal{F}$ is the \emph{$p$-depletion} of $\mathcal{F}$.  
Define the overconvergent modular form 
 \[d^{-1}\omega_{\mathcal{F}^{[p]}}=\Pi(\phi_{\mathrm{DT}}^*)(F_\infty).\] Then 
 $\Pi(\phi_{\mathrm{DT}}^*)^{-1}d^{-1} \omega_{\mathcal{F}^{[p]}}=F_\infty$. Note that the definition of $d^{-1}\omega_{\mathcal{F}^{[p]}}$ depends on the choice of a constant defining $F_\infty$, 
 which we fix as follows by means of Serre--Tate expansions.  
Pick a point $x_\infty$ in the wide open neighborhood ${\mathcal{W}}_\infty$ of $\mathcal{A}_\infty(p^m)$ appearing Proposition \ref{Col1}; accordingly with our definitions, 
$\mathrm{red}(x_\infty)$ belongs to $\texttt{Ig}_\infty(\overline{\F}_p)$, so we may consider the $T_{x_\infty}$-expansion 
$\mathcal{F}(T_{x_\infty})$ of $\mathcal{F}$ at $x_\infty$
associated with the choice of a basis 
$\{x_A,x_A'\}$ of $\Ta_p(A)$ associated with the trivialization $\beta$ via Cartier duality (see \cite[\S 3.1]{Magrone}).
The $T_{x_\infty}$-expansion of $\mathcal{F}^{[p]}$ is then
 \[\mathcal{F}^{[p]}(T_{x_\infty})=\sum_{p\nmid n}a_nT_{x_\infty}^n\] 
 {see \cite[Lemma 5.2]{Burungale}).} Define 
 \begin{equation}\label{normcol}
 d^{-1}\mathcal{F}^{[p]}(T_{x_\infty})=\sum_{p\nmid n}\frac{a_n}{n+1}T_{x_\infty}^{n+1}.\end{equation}
We may then normalize the choice of $F_\infty$ by imposing that the $T_{x_\infty}$-expansion of  $d^{-1}\omega_{\mathcal{F}^{[p]}}$ is that in \eqref{normcol}; more precisely, we introduce the following: 
 \begin{definition}\label{def9.6} Let $d^{-1}\mathcal{F}^{[p]}_{x_\infty}$ denote the unique overconvergent modular form such that:
 \begin{itemize} 
 \item $d(d^{-1}\mathcal{F}^{[p]}_{x_\infty})=\mathcal{F}^{[p]}$; 
 \item The $T_{x_\infty}$-expansion of $d^{-1}\mathcal{F}^{[p]}_{\infty}$ is equal to $d^{-1}\mathcal{F}^{[p]}(T_{x_{\infty}})$.\end{itemize}
 \end{definition}
 
The previous definition fixes the choice of $d^{-1}\omega_{\mathcal{F}^{[p]}}$ and, consequently, of $F_\infty$, to be $d^{-1}\mathcal{F}^{[p]}_{x_\infty}$. 
Note that in the residue disk of $x_\infty$ we have 
 $d^{-1}\mathcal{F}^{[p]}_{x_\infty}=d^{-1}\mathcal{F}^{[p]}(T_{x_\infty})\omega_{x_\infty}$. 
 
 \begin{definition}\label{defnormCol}
 We say that the Coleman primitive $F_\infty$ in $\mathcal{W}_\infty$ appearing in Definition \ref{defCol} \emph{vanishes} at $x_\infty$ if 
 the choice of the constant is normalized as in \eqref{normcol}. 
 \end{definition}
 
With these definitions, if $F_\infty$ vanishes at $x_\infty$, we have 
\begin{equation}\label{Col-dminus1}
d^{-1}\mathcal{F}^{[p]}_{x_\infty}=\Pi(\phi_{\mathrm{DT}}^*)F_\infty.\end{equation}

 \subsection{Logarithmic de Rham cohomology}Let $L_0$ be the maximal unramified extension of $L$. The work of Hyodo--Kato \cite{HK}  equips the $L$-vector space $H^1_\dR({X}_m/L)$ with a canonical $L_0$-subvector space 
\[H^1_{\text{log-cris}}(\mathscr{X}_m)\longmono H^1_\dR({X}_m/L)\] equipped with a semi-linear Frobenius operator $\varphi$; by the results of Tsuji \cite{Tsiji}, 
there is a canonical comparison isomorphism $\mathbf{D}_{\dR,L}(V_m)\simeq H^1_\dR(X_m/L)$ of filtered $\varphi$-modules, where $V_m=H^1_{\text{ét}}({X}_m\otimes_\Q\overline\Q,\Q_p)$. For a Hecke module $M$, let us denote $M[\mathcal{F}]$ the eigencomponent corresponding to the eigenform $\mathcal{F}$; we also denote $F_\mathcal{F}\subseteq\overline{\Q}_p$ the Hecke field of $\mathcal{F}$ inside the algebraic closure of $\Q_p$. Set (generalizing previous definitions in the case of modular forms appearing as specializations of a Hida family) $V_\mathcal{F}^*=(V_m\otimes_{\Q_p}F_{\mathcal{F}})[\mathcal{F}]$. 
We then have a canonical isomorphism 
$\mathbf{D}_{\mathrm{cris},L_0}(V_\mathcal{F}^*)\simeq H^1_\text{log-cris}(\mathscr{X}_m)[\mathcal{F}]$ of $L_0\otimes_{\Q_p}F_\mathcal{F}$-modules 
compatible with the $\varphi$-action which induces after extending scalars an isomorphism $\mathbf{D}_{\dR,L}(V_\mathcal{F})\simeq H^1_\dR({X}_m/L)[\mathcal{F}]$  of $L\otimes_{\Q_p}F_\mathcal{F}$-modules. 

Let ${J}_m=\mathrm{Jac}({X}_m\otimes_{\Q_p}L)$ and consider the map 
{\small{\[\xymatrix{
\delta_m\colon {J}_m(L)\ar[r]^-{\mathrm{Kum}}& H^1_f(L,\Ta_p({J}_m))\ar[r]^-{\mathrm{proj}}& 
H^1_f(L,V_\mathcal{F}^*(1))\ar[r]^-{\log}& \frac{\mathbf{D}_\dR(V_\mathcal{F}^*(1))}{\Fil^0(\mathbf{D}_\dR(V_\mathcal{F}^*(1)))}\ar[r]^-\sim & (\Fil^0(\mathbf{D}_\dR(V_\mathcal{F})))^\vee}\] }}
where:
\begin{itemize}
\item $\mathrm{Kum}$ is the Kummer map and we write $\mathbf{D}_\mathrm{dR}=\mathbf{D}_{\mathrm{dR},L}$ to simplify the notation;
\item $\mathrm{proj}$ is induced by the projection map $\Ta_p({J}_m)\rightarrow V_\mathcal{F}$, and $V_\mathcal{F}= V_\mathcal{F}^*(1)$;
\item $\log$ is the inverse of the Bloch--Kato exponential map 
\[\exp\colon\frac{\mathbf{D}_\dR(V_\mathcal{F}^*(1))}{\Fil^0(\mathbf{D}_\dR(V_\mathcal{F}^*(1)))}\overset\sim\longrightarrow H^1_f(L,V_\mathcal{F}^*(1))\] 
which is an isomorphism in our setting; in fact, $\ker(\exp)=\mathbf{D}_\mathrm{cris}(V^*_\mathcal{F}(1))^{\varphi=1}$ (see the comment after Definition 3.10 in \cite{BK}), which is trivial since $\mathcal{F}$ is taken to be $p$-ordinary and therefore $p$ is not a root of the Hecke polynomial of $\mathcal{F}$.
\item The isomorphism 
\[\frac{\mathbf{D}_\dR(V_\mathcal{F}^*(1))}{\Fil^0(\mathbf{D}_\dR(V_\mathcal{F}^*(1)))}\simeq (\Fil^0(\mathbf{D}_\dR(V_\mathcal{F})))^\vee\] is induced by the de Rham pairing. 
\end{itemize} 

Following \cite[\S3.4]{BDP} and \cite[\S2.2]{Castella-MathAnn}, the map $\delta_m$ can be described as follows. First, recall that the Bloch--Kato Selmer group can be identified with the group of crystalline extensions  
\[0\longrightarrow V_\mathcal{F}^*(1)\longrightarrow W\overset{\rho}\longrightarrow \Q_p\longrightarrow 0\]
and since $\mathbf{D}_\mathrm{cris}(V_\mathcal{F}^*(1))^{\varphi=1}$ is trivial,
the resulting extension of $\varphi$-modules 
\begin{equation}\label{crist}
0\longrightarrow \mathbf{D}_\mathrm{cris}(V_\mathcal{F}^*(1))\longrightarrow \mathbf{D}_\mathrm{cris}(W)\longrightarrow L_0\longrightarrow 0\end{equation}
admits a unique section 
\[s_W^{\Frob}\colon L_0\longrightarrow \mathbf{D}_\mathrm{cris}(W)\] with $\eta_W^{\Frob}=s_W^{\Frob}(1)\in  
\mathbf{D}_\mathrm{cris}(W)^{\varphi=1}$. We also fix a section 
\[s_W^\mathrm{Fil}\colon L\longrightarrow \mathrm{Fil}^0(\mathbf{D}_\dR(W))\]
of the exact sequence of $L$-vector spaces 
\begin{equation}\label{dR-Fil}
0\longrightarrow \mathrm{Fil}^0(\mathbf{D}_\dR(V_\mathcal{F}^*(1)))\longrightarrow 
\mathrm{Fil}^0(\mathbf{D}_\dR(W))\longrightarrow L\longrightarrow 0\end{equation}
obtained by extending scalars from $L_0$ to $L$ in \eqref{crist}, using the canonical isomorphism with de Rham cohomology, and 
taking the $\mathrm{Fil}^0$-parts of the resulting sequence. Define $\eta_W^{\mathrm{Fil}}=s_W^\mathrm{Fil}(1)$ and consider the difference \[\eta_W=\eta_W^{\Frob}-\eta_W^{\mathrm{Fil}}\] 
viewed as an element in $\mathbf{D}_\dR(W)$; this difference comes from an element in $\mathbf{D}_\dR(V_\mathcal{F}^*(1))$, denoted with the same symbol $\eta_W$, and its image modulo $\mathrm{Fil}^0(\mathbf{D}_\dR(V_\mathcal{F}^*(1)))$ is well defined. Then we have (see \cite[Lemma 2.4]{Castella-MathAnn} and the references therein) 
\[\log(W)=\eta _W\mod \mathrm{Fil}^0(\mathbf{D}_\dR(V_\mathcal{F}^*(1))).\]

Let $\Delta\in J_s(L)$ be the class of a degree zero divisor in ${X}_m$, with support contained in the finite set of points $S\subseteq {X}_m(L)$. Define the map 
\begin{equation}\label{Kummer}
\xymatrix{
\kappa_m\colon J_m(L)\ar[r]^-{\mathrm{Kum}}& H^1_f(L,\Ta_p( J_m))\ar[r]^-{\mathrm{proj}}& 
H^1_f(L,V_\mathcal{F}^*(1))}\end{equation} 
and consider the class $\kappa_m(\Delta)\in H^1_f(L,V_\mathcal{F}^*(1))$. 
Denote $W_\Delta$ the extension class associated with $\kappa_m(\Delta)$. Attached to $W_\Delta$ we then have the class $\eta_{W_\Delta}$ in $\mathbf{D}_\dR(V_\mathcal{F}^*(1))$ constructed before, and we may consider the (weight $2$) newform $\mathcal{F}^*$ associated with the twisted form  
$\mathcal{F}\otimes\psi_\mathcal{F}^{-1}$, where $\psi_\mathcal{F}$ denotes the character of $\mathcal{F}$. Let as before $\omega_{\mathcal{F}^*}$ denote the differential form attached to $\mathcal{F}^*$; denote with the same symbol $\omega_{\mathcal{F}^*}$ the corresponding element in $\mathbf{D}_\dR(V_{\mathcal{F}^*}^*)$ via the isomorphism 
$\mathbf{D}_\dR(V_{\mathcal{F}^*}^*)\simeq H^1_\dR(\widetilde{X}_m/L)[\mathcal{F}^*]$. Note that $\omega_{\mathcal{F}^*}$ belongs to $\mathrm{Fil}^1(\mathbf{D}_\dR(V_{\mathcal{F}^*}^*))$, which is equal to $\mathrm{Fil}^0(\mathbf{D}_\dR(V_{\mathcal{F}}))$; we therefore obtain a class 
$\omega_{\mathcal{F}^*}\in \mathrm{Fil}^0(\mathbf{D}_\dR(V_{\mathcal{F}}))$.

\begin{lemma}\label{Col5}
$\delta_m(\Delta)(\omega_{\mathcal{F}^*})=\langle \eta_{W_\Delta},\omega_{\mathcal{F}^*}\rangle_\dR$.  
\end{lemma}

\begin{proof}
Follow the argument in the case of modular curves in 
\cite[\S4.1]{BDP} (the good reduction case) and \cite[\S2.2]{Castella-MathAnn} (the bad reduction case).  
\end{proof}

Pick as before a point $x_\infty$ in the wide open space ${\mathcal{W}}_\infty$. Let $F_\infty^*$ be the Coleman primitive of $\omega_{\mathcal{F}^*}$ on $\mathcal{W}_\infty$ which vanishes at $x_\infty$ (\emph{cf.} Definition \ref{defnormCol}). 
We may then consider the map $j_m^{(x_\infty)}\colon {X}_m(\C_p)\rightarrow  J_m(\C_p)$ which associates to $P$ the divisor $(P)-(x_\infty)$. We simply write $j_m$ for this map when $x_\infty$ 
is understood. 
 
\begin{lemma}\label{Col6} Let $\Delta=j_m(P)$ and $F_\infty^*$ the Coleman primitive of $\omega_{\mathcal{F}^*}$ on $\mathcal{W}_\infty$ which vanishes at $\infty$. Assume that $m>1$. Then 
$\langle \eta_{W_\Delta},\omega_{\mathcal{F}^*}\rangle_\dR=F_{\omega_{\mathcal{F}^*}}(P)$. 
\end{lemma}

\begin{proof} The proof follows \cite[\S4.2]{DR} and \cite[Proposition 2.9]{Castella-MathAnn}, which adapts 
the proof of \cite[Proposition 3.21]{BDP} to the semistable setting. 

\emph{Step 1.} We first describe the classes $\eta_{W_\Delta}^{\mathrm{Fil}}$ and $\eta_{W_\Delta}^{\Frob}$. 
Let $S=\{P,x_\infty\}$ and $Y_S=\mathscr{X}_m(\C_p)-S$ as before. The class $\eta^{\mathrm{Fil}}_{W_\Delta}$ is an element in $\mathrm{Fil}^0(\mathbf{D}_\dR(W_\Delta))$ with $\rho_\dR(\eta^{\mathrm{Fil}}_{W_\Delta})=1$, where $\rho_\dR$ is the top right arrow map in the following commutative diagram 
\[\xymatrix{
0\ar[r] & 
\mathrm{Fil}^0(\mathbf{D}_\dR(V_{\mathcal{F}}^*(1))\ar[r]\ar[d]^-{\simeq}& 
\mathrm{Fil}^0(\mathbf{D}_\dR(W_\Delta))\ar[r]^-{\rho_\dR}\ar[d] &
L\otimes_{\Q_p}F_{\mathcal{F}}\ar[r]\ar[d]^-{\Delta}&0\\
0\ar[r]&
\mathrm{Fil}^1\left(H^1_\dR({X}_m/L)\right)[\mathcal{F}]\ar[r]&
\mathrm{Fil}^1\left(H^1_\dR(Y_S/L)\right)[\mathcal{F}]\ar[r]^-{\oplus\res_Q} &
(L\otimes_{\Q_p}F_\mathcal{F})_0^{S}
\ar[r]
&0
}\]which realizes the exact sequence in the top horizontal line (which is \eqref{dR-Fil}) as the pull-back of the bottom horizontal line with respect to the rightmost $L\otimes_{\Q_p}F_\mathcal{F}$-linear vertical map $\Delta$ taking $1$ to $(P,-x_\infty)$; in the bottom horizontal arrow, 
$\res_Q(\omega)$ is the residue at $Q\in S$ of the differential form $\omega$, 
and $(L\otimes_{\Q_p}F_\mathcal{F})^S_0$ denotes the degree zero divisors over $S$ with coefficients in $L\otimes_{\Q_p}F_\mathcal{F}$, \emph{i.e.} those $(x_Q)_{Q\in S}$ in 
$L\otimes_{\Q_p}F_\mathcal{F}$ with $\sum_{Q\in S}n_Q=0$. Therefore, we have 
$\res_P(\eta^\mathrm{Fil}_{W_\Delta})=1$ and $\res_{x_\infty}(\eta^\mathrm{Fil}_{W_\Delta})=-1$. 
 
Similarly, the class $\eta^{\Frob}_{W_\Delta}$ is an element in $\mathbf{D}_\mathrm{cris}(W_\Delta)^{\varphi=1}$ with $\rho_\mathrm{cris}(\eta^{\Frob}_{W_\Delta})=1$,
where $\rho_\mathrm{cris}$ is the top right arrow map in the following commutative diagram 
\[\xymatrix{
0\ar[r] & 
\mathbf{D}_\mathrm{cris}(V_{\mathcal{F}}^*(1))\ar[r]\ar[d]^-{\simeq}& 
\mathbf{D}_\mathrm{cris}(W_\Delta)\ar[r]^-{\rho_\mathrm{cris}}\ar[d] &
L_0\otimes_{\Q_p}F_{\mathcal{F}}\ar[r]\ar[d]^-{\Delta}&0\\
0\ar[r]&
H^1_\text{log-cris}({X}_m/L_0)[\mathcal{F}](1)\ar[r]&
H^1_\text{log-cris}(Y_S/L_0)[\mathcal{F}](1)\ar[r]^-{\oplus\res_Q} &
(L_0\otimes_{\Q_p}F_{\mathcal{F}})_0^{S}
\ar[r]
&0
}\]  which realizes the exact sequence in the top horizontal line (which is \eqref{crist}) as the pull-back of the bottom horizontal line with respect to the rightmost $L_0\otimes_{\Q_p}F_\mathcal{F}$-linear vertical map $\Delta$ taking $1$ to $(P,-x_\infty)$; as before in the bottom horizontal arrow, 
$\res_Q(\omega)$ is the residue at $Q\in S$ of the differential form $\omega$, 
and $(L\otimes_{\Q_p}F_\mathcal{F})^S_0$ denotes the degree zero divisors over $S$ with coefficients in $L\otimes_{\Q_p}F_\mathcal{F}$. By the discussion closing \S\ref{sec:9.1.2} (see especially \eqref{eq:9.3}), $\eta^{\Frob}_{W_\Delta}$ is represented by a pair of sections 
$(\eta_\infty^{\Frob},\eta_0^{\Frob})$ of 
$\Omega^1_{\rig}(\widetilde{\mathcal{W}}_\infty)\times \Omega^1_\rig(\widetilde{\mathcal{W}}_0)$. 
Since $\eta^{\Frob}_{W_\Delta}$ is fixed by $\varphi$, we have  
$\eta_\infty^{\Frob}=\phi_\mathrm{DT}\eta_\infty^{\Frob}+dG_\infty$ for a rigid analytic function $G_\infty$ on $\widetilde{\mathcal{W}}_\infty$, 
and $\eta_0^{\Frob}=\tilde{\phi}_\mathrm{DT}\eta_0^{\Frob}+dG_0$ for a rigid analytic function $G_0$ on $\widetilde{\mathcal{W}}_0$. Moreover, we also have $\res_{Q}(\eta_{W_\Delta}^{\Frob})=\mathrm{res}_Q(\eta_{W_\Delta}^{\mathrm{Fil}})$ for all $Q\in S$, and since 
$\res_{Q}(\eta_{W_\Delta}^{\Frob})=\res_{\mathcal{V}_Q}(\eta_{W_\Delta}^{\Frob})$ for all $Q\in S$, we may rewrite the last condition in the form  
$\res_{\mathcal{V}_Q}(\eta_{W_\Delta}^{\Frob})=\mathrm{res}_Q(\eta_{W_\Delta}^{\mathrm{Fil}})$ for all $Q\in S$. 

\emph{Step 2.} (\emph{Cf.} \cite[Lemma 3.20]{BDP}.) We now show that 
\begin{equation} \label{9.9step1} 
\sum_{\mathcal{V}\subseteq\widetilde{\mathcal{W}}_\infty}\res_{\mathcal{V}}(\langle F_\infty^*,\eta_{\infty}^{\mathrm{Frob}}\rangle_\dR) +\sum_{\mathcal{V}\subseteq\widetilde{\mathcal{W}}_0}\res_{\mathcal{V}}(\langle F_0^*,\eta_{0}^{\mathrm{Frob}}\rangle_\dR)=0.\end{equation}
We begin by showing that the first summand in \eqref{9.9step1} is zero. 
Recall $\eta_\infty^{\Frob}=\phi_\mathrm{DT}\eta_\infty^{\Frob}+dG_\infty$. By the Leibeniz rule we then have 
\[d(\langle \phi_\mathrm{DT} F_\infty^*,G_\infty\rangle_\dR)=\langle\phi_\mathrm{DT} F_\infty^*,dG_\infty\rangle_\dR+\langle \phi_\mathrm{DT}\omega_{\mathcal{F}^*},G_\infty\rangle_\dR\] where we use that $d(\phi_\mathrm{DT} F^*_\infty)=\phi_\mathrm{DT} d F_\infty^*$ because $\phi_\mathrm{DT}$ is horizontal for $d$. Therefore, the RHS is exact on each $\mathcal{V}$, so we have 
$\res_\mathcal{V}(\langle\phi_\mathrm{DT} F_\infty^*,dG_\infty\rangle_\dR)=-\res_\mathcal{V}(\langle \phi_\mathrm{DT}\omega_{\mathcal{F}^*},G_\infty\rangle_\dR)$; on the other hand, $\langle \phi_\mathrm{DT}\omega_{\mathcal{F}^*},G_\infty\rangle_\dR$ is a rigid analytic differential form on $\widetilde{\mathcal{W}}_\infty$, so the sum of its residues is zero for all $\mathcal{V}$. We conclude that 
\begin{equation}\label{eq:9.9} 
\sum_{\mathcal{V}\subseteq\widetilde{\mathcal{W}}_\infty}\res_\mathcal{V}(\langle\phi_\mathrm{DT} F_\infty^*,dG_\infty\rangle_\dR)=0.\end{equation} We then observe that 
$\res_{\mathcal{V}}(\langle F_\infty^*,\eta_{\infty}^{\mathrm{Frob}}\rangle_\dR)=
\res_{\mathcal{V}}(\langle \phi_\mathrm{DT} F_\infty^*,\phi_\mathrm{DT}\eta_{\infty}^{\mathrm{Frob}}\rangle_\dR)$; combing this with the equation $\eta_\infty^{\Frob}=\phi_\mathrm{DT}\eta_\infty^{\Frob}+dG_\infty$ and the equation \eqref{eq:9.9} we conclude that 
\[\sum_{\mathcal{V}\subseteq\widetilde{\mathcal{W}}_\infty}\res_{\mathcal{V}}(\langle F_\infty^*,\eta_{\infty}^{\mathrm{Frob}}\rangle_\dR)=\sum_{\mathcal{V}\subseteq\widetilde{\mathcal{W}}_\infty}\res_{\mathcal{V}}(\langle \phi_\mathrm{DT} F_\infty^*,\eta_{\infty}^{\mathrm{Frob}}\rangle_\dR).\] It follows that 
\[\Pi(1)\sum_{\mathcal{V}\subseteq\widetilde{\mathcal{W}}_\infty}\res_{\mathcal{V}}(\langle F_\infty^*,\eta_{W_\Delta}^{\mathrm{Frob}}\rangle_\dR)=\sum_{\mathcal{V}\subseteq\widetilde{\mathcal{W}}_\infty}\res_{\mathcal{V}}(\langle \Pi(\phi_\mathrm{DT} )F_\infty^*,\eta_{\infty}^{\mathrm{Frob}}\rangle_\dR).\]
Now $\Pi(\phi_\mathrm{DT})F_\infty^*$ is rigid analytic, and therefore the RHS is zero; since $\Pi(1)\neq0$, we conclude that 
\[\sum_{\mathcal{V}\subseteq\widetilde{\mathcal{W}}_\infty}\res_{\mathcal{V}}(\langle F_\infty^*,\eta_{W_\Delta}^{\mathrm{Frob}}\rangle_\dR)=0.\] A similar argument, replacing $\widetilde{\mathcal{W}}_\infty$ with $\widetilde{\mathcal{W}}_0$, $\eta_\infty$ with $\eta_0$, $G_\infty$ with $G_0$, $F_\infty^*$ with $F_0^*$ and $\phi_\mathrm{DT}$ by $\tilde\phi_\mathrm{DT}$ shows that 
\[\sum_{\mathcal{V}\subseteq\widetilde{\mathcal{W}}_0}\res_{\mathcal{V}}(\langle F_0^*,\eta_{0}^{\mathrm{Frob}}\rangle_\dR)=0\] and \eqref{9.9step1} follows.

\emph{Step 3.} (\emph{Cf.} \cite[Lemma 3.19]{BDP}.) We now show that 
\begin{equation}\label{9.9step2}
\sum_{\mathcal{V}\subseteq\widetilde{\mathcal{W}}_\infty}\res_\mathcal{V}(F^*_\infty\eta_{\infty}^{\mathrm{Fil}})+\sum_{\mathcal{V}\subseteq\widetilde{\mathcal{W}}_0}\res_\mathcal{V}(F^*_0\eta_{0}^{\mathrm{Fil}})=F^*_\infty(P).\end{equation}
Since $F_\infty^*$ vanishes at $x_\infty$, $F^*_\infty\eta^\mathrm{Fil}_{\infty}$ is locally analytic in a neighborhood of $x_\infty$, and it follows that $\res_{x_\infty}(F^*_\infty\eta^\mathrm{Fil}_{\infty})=0$. 
On the other hand, since $\res_P(\eta^\mathrm{Fil}_{W_\Delta})=1$, we have $\res_P(F_\infty^*\eta^{\mathrm{Fil}}_{W_\Delta})=F_\infty^*(P)$, so we conclude that 
\[\sum_{\mathcal{V}\subseteq\widetilde{\mathcal{W}}_\infty}\res_\mathcal{V}(F^*_\infty\eta_{\infty}^{\mathrm{Fil}})=F^*_\infty(P).\]
On the other hand, $F_0^*\eta_{W_\Delta}^\mathrm{Fil}$ is analytic on $\mathcal{W}_0$, so the second summand in the LHS of \eqref{9.9step2} is zero, and \eqref{9.9step2} follows. 

\emph{Step 4.} The result now follows combining \eqref{9.9step1} and \eqref{9.9step2} with \eqref{pairing-residues} and using that, since $m>1$, the wide opens $\widetilde{\mathcal{W}}_\infty$ 
and $\widetilde{\mathcal{W}}_0$ are disjoint.  
\end{proof}

\begin{corollary}\label{CoroCole} Let $\Delta=(P)-(x_\infty)$ and $F_\infty^*$ the Coleman primitive of $\omega_{\mathcal{F}^*}$ on $\mathcal{W}_\infty$ which vanishes at $\infty$. Assume that $m>1$. Then 
$\delta_m(\Delta)(\omega_{\mathcal{F}^*})=F_{\omega_{\mathcal{F}^*}}(P)$. 
\end{corollary} 
\begin{proof}
This follows immediately from Lemma \ref{Col5} and Lemma \ref{Col6}. 
\end{proof}

 \section{Reciprocity laws} 

\subsection{Analytic $p$-adic $L$-function} Let $c\mathcal{O}_K$, with $c\geq 1$ and $p\nmid c$, be the conductor of $x \mapsto\lambda(x)\lambda^{-1}(\bar{x})$.
Consider the CM points $x(\mathfrak{a})$ with $\mathfrak{a} \in \Pic(\mathcal{O}_c)$, defined in \S \ref{sec.CMpoints}.
Recall that $x(\mathfrak{a})$ has a model defined over ${\Z_p^\unr}$, and define 
the fiber product $x(\mathfrak{a})_{{\I}} \defeq x(\mathfrak{a}) \otimes_{\Z_p^\unr}\widetilde{\I}$, where $\widetilde{\I}\defeq \I\otimes_{\Z_p}\Z_p^{\unr}$.
Define now a $\widetilde{\I}$-valued measure $\mu_{\omega_\I,\mathfrak{a}}$ on $\Z_p$ by
\begin{equation*} \label{mu_a}
\int_{\Z_p} (T_{x(\mathfrak{a})}+1)^x d\mu_{\omega_\I,\mathfrak{a}}(x) = 
\boldsymbol{\mathcal{F}}^{[p]} \left( (T_{x(\mathfrak{a})}+1)^{\mathrm{N}(\mathfrak{a}^{-1})\sqrt{(-D_K)}^{-1}}\right) \in \widetilde{\I}[[T_{x(\mathfrak{a})}]],
\end{equation*}
where 
$\boldsymbol{\mathcal{F}}^{[p]}(T_x)=\sum_{p\nmid n}a_nT_{x(\mathfrak{a})}^n$ is the $p$-depletion of $\boldsymbol{\mathcal{F}}(T_x)=\sum_{n\geq 1}a_nT_{x(\mathfrak{a})}^n$, and for an ideal $\mathfrak{a}\subseteq\mathcal{O}_c$, we define $\mathrm{N}(\mathfrak{a})=c^{-1}\cdot\sharp(\mathcal{O}_c/\mathfrak{a})$.
The $p$-adic $L$-function associated with $\omega_\I$ and $\boldsymbol{\xi}$ is the $\widetilde{\I}$-valued measure on  
$\Gamma_\infty=\Gal(H_{cp^{\infty}}/K)$ given for any continuous function $\varphi\colon\Gamma_\infty\rightarrow \widetilde{\I}$ by 
\[
\mathscr{L}_{\I,\boldsymbol{\xi}}^\mathrm{an} (\varphi) = \sum_{\mathfrak{a} \in \Pic(\mathcal O_{c})}
{\boldsymbol{\chi}}^{-1}
{\boldsymbol{\xi}}(\mathfrak{a})
\mathrm{N}(\mathfrak{a})^{-1}
	\int_{\Z_p^{\times}} 
	(\varphi \big|[\mathfrak{a}])(u)d\mu_{\omega_\I,\mathfrak{a}}(u).
	\]
For each arithmetic morphism $\nu\colon\I\rightarrow\mathcal{O}_\nu$, the choice of $\omega_\I$ determines 
a modular form $\mathcal{F}_\nu$. If $\mathscr{L}_{\mathcal{F}_\nu,\xi_\nu}$ is the $p$-adic $L$-function attached to $\mathcal{F}_\nu$ and $\xi_\nu$ and constructed in \cite{Magrone}, then the main result of \cite{LMW} gives 
\begin{equation}\label{int-1}
\mathscr{L}_{\I,\boldsymbol{\xi}}^\mathrm{an} (\nu)=\vartheta_\nu^{-1}(c)\mathscr{L}_{\mathcal{F}_\nu,\xi_\nu}.\end{equation}

We recall some results from \cite{Magrone} and \cite{BCK}.  
For any ideal $\mathfrak{a}\subseteq\mathcal{O}_c$,
 any continuous function  $\phi\colon\Z_p^\times\rightarrow\mathcal{O}_{\C_p}$ and any power series $G(T_{x(\mathfrak{a})}) \in W[[T_{x(\mathfrak{a})}]]$, define the formal power series 
$([\phi]G)(T_{x(\mathfrak{a})})\in \Z_p^\unr(\phi)[[T_{x(\mathfrak{a})}]]$, where $\Z_p^\unr(\phi)$ is the extension of $\Z_p^\unr$ generated by the values of $\phi$, 
by the formula 
\begin{equation}
\label{faphi}([\phi]G)(T_{x(\mathfrak{a})})=
\int _{\Z_p^\times} \phi(x)(T_{x(\mathfrak{a})}+1)^x d\mu_{\omega_\I,\mathfrak{a}}(x).
\end{equation}
If $\mathcal{F}$ is a quaternionic newform for $X_1$ or the $p$-stabilization of a newform for $X_0$ (meaning it is the Jacquet--Langlands lift of a level $\Gamma_0(Np)$ elliptic newform or the $p$-stabilization of a level $\Gamma_0(N)$ elliptic newform), define \[\mathcal{F}^{[p]}_\mathfrak{a}(T_{x(\mathfrak{a})}) \defeq \mathcal{F}^{[p]}\left((T_{x(\mathfrak{a})} + 1)^{\mathrm{N}(\mathfrak{a})^{-1}\sqrt{-D_K}^{-1}}\right).\]
By \cite[Proposition 4.5]{Magrone} (see also \cite[Proposition 4.1]{BCK}), if $\phi\colon(\Z/p^n\Z)^\times\rightarrow\overline\Q_p^\times$ is a primitive Dirichlet character, and $[\mathfrak{a}]$ is an ideal class in $\Pic(\mathcal{O}_c)$ with $p\nmid c$ as before, we have 
\begin{equation}\label{mag4.5}
([\phi]\mathcal{F}^{[p]}_\mathfrak{a})(0) 
=p^{-n}\mathfrak{g}(\phi)\sum_{u\in (\Z/p^n\Z)^\times}\phi^{-1}(u)\mathcal{F}(x(\mathfrak{a})\star\mathbf{n}(u/p^n)),
\end{equation} where $\mathfrak{g}(\phi)$ is the Gauss sum of $\phi$.

Recall the point $x(\mathfrak{a})=[(\iota_K,a^{-1}\xi)]$ defined in \S\ref{sec.CMpoints}, which corresponds to the sequence 
$(x_m(\mathfrak{a}))_{m\geq 0}$ of Heegner points, each one in 
${X}_m(H_{cp^\infty})$. 
Fix an integer $n\geq 1$. For any $x$ in $\Q_p$, define the \emph{$\star$-action} of $\mathbf{n}(x)$ on the point 
$x(\mathfrak{a})$ by the formula
\[x(\mathfrak{a})\star\mathbf{n}(x)=[(\iota_K,a^{-1}\xi\mathbf{n}(x)]\]
where $\mathbf{n}(x)$
denotes the element in $\widehat{B}^\times$ whose $p$-component has image equal to $\smallmat{1}{x}{0}{1}$ in $\GL_2(\Q_p)$ via the isomorphism $i_p$
and whose components at other primes are trivial. 
A simple computation (see also \cite[page 587]{CH}) shows that for any $u\in\Z_p^\times$ we have 
\[\xi\cdot\mathbf{n}(u/p^n)=
i_\mathfrak{p}(u/p^n)\xi^{(n)}\cdot\smallmat{u^{-1}}{u^{-1}}01,
\] where $i_\mathfrak{p}(u/p^n)$ is the element of $\widehat{K}^\times$ having all components equal to $1$ 
except the $\p$-component, equal to $u/p^n$.
We thus obtain 
\begin{equation}\label{CHeq}
x(\mathfrak{a})\star \mathbf{n}(u/p^n)=[(\iota_K,a^{-1}\xi\mathbf{n}(u/p^n))]=
\left[\left(\iota_K,a^{-1}i_\mathfrak{p}(u/p^n)\xi^{(n)}\cdot\smallmat{u^{-1}}{u^{-1}}01
\right)\right].\end{equation}
By \cite[Proposition 4.1]{BCK}, for any $u\in\Z_p^\times$, $(x(\mathfrak{a})\star\mathbf{n}(u/p^n))$ is still a CM point defined over $\mathcal{Z}=\Z^{\unr}\cap K^{\mathrm{ab}}$, where $K^\mathrm{ab}$ is the maximal abelian extension of $K$. Moreover, we have 
$(x(\mathfrak{a})\star\mathbf{n}(u/p^n))\otimes_\mathcal{Z}\overline{\F}_p=\bar{x}(\mathfrak{a})$ and 
\[t_{x(\mathfrak{a})}(x(\mathfrak{a})\star\mathbf{n}(u/p^n))=\zeta_{p^n}^{-u\mathrm{N}(\mathfrak{a}^{-1})\sqrt{-D_K}^{-1}},\] where, in the notation of \cite{BCK}, $t_{x(\mathfrak{a})}=T_{x(\mathfrak{a})}+1$.
 
 \subsection{Weight $2$ specializations} Let $\nu$ be an arithmetic morphisms of signature $(2,\psi)$ and let the conductor of $\psi$ be $p^m$ for some integer $m\geq1$. Let $\hat{\phi}\colon K^\times\backslash\widehat{K}^\times\rightarrow F^\times$
be the $p$-adic avatar of a Hecke character $\phi\colon K^\times\backslash\A_K^\times\rightarrow\overline\Q^\times$ of infinity type $(1,-1)$ and conductor $p^n$ for some integer $n\geq m$ such that the Galois character $\widetilde{\phi}\colon \Gal(K^\mathrm{ab}/K)\rightarrow F^\times$ factors through $\widetilde{\Gamma}_\infty$.
The next task consists in computing the $(\nu,\hat\phi^{-1})$-specialization of $\mathscr{L}_{\I,\boldsymbol{\xi}}^{\mathrm{alg}}$. We put 
\[\mathscr{L}_{\I,\boldsymbol{\xi}}^{\mathrm{alg}}(\nu,\hat\phi^{-1})=\mathrm{sp}_{\nu,\phi}\left(\mathscr{L}_{\I,\boldsymbol{\xi}}^{\mathrm{alg}}\right).\]
For a number field $L$ and the ring of algebraic integers $\mathcal{O}$ of a finite extension of $\Q$ there is a canonical exact sequence 
\[0\longrightarrow{J}_m(L)\otimes_\Z\mathcal{O} \longrightarrow\Pic({X}_m/L )\otimes_\Z\mathcal{O} \overset{\deg}\longrightarrow \mathcal{O} \longrightarrow 0\]
and taking ordinary parts, since the degree of $U_p$ is $p$, we obtain a canonical isomorphism
\begin{equation}\label{eq9.14}{J}_m(L)^\ord\otimes_\Z\mathcal{O}\stackrel{\sim}{\longrightarrow} 
\Pic({X}_m/L)^\ord\otimes_\Z\mathcal{O}.\end{equation} 
We denote $\varrho_m$ the inverse of this canonical isomorphism. 
Consider the divisor 
\[Q_{cp^n,m}=
\sum_{\sigma\in \Gal(H_{cp^{n+m}}/H_{cp^n})}{P}_{cp^{n+m},m}^{\tilde\sigma}\otimes \chi_\nu(\tilde\sigma) 
\]where $\tilde\sigma\in \Gal(L_{cp^{n+m},m}/H_{cp^n})$ is any lift of $\sigma$ (the independence of the lift follows from the results recalled in \ref{CM section}). 
This defines a canonical class $\varrho_m(Q_{cp^n,m})$ in
${J}_m(\overline\Q)\otimes_{\Z}\mathcal{O}_\nu(\chi_\nu)$, which is fixed by the action of $\Gal(\overline\Q/L_{cp^{n+m},m})$. Tracing through the definition of big Heegner points, we see (\emph{cf.}  \cite[\S3.4]{LV-Pisa}, see especially \cite[(3.6)]{LV-Pisa}) that when $n\geq m\geq 2$, 
\begin{equation}\label{eq13.2}
\varrho_m(Q_{cp^n,m})=\left(\frac{\nu(\mathbf{a}_p)}{p}\right)^m\cdot \mathrm{sp}_\nu(\mathfrak{X}_{cp^n}). 
\end{equation} 

Recall the overconvergent modular form  $d^{-1}\mathcal{F}_{\nu, x_\infty}^{[p]}$ in Definition \ref{def9.6} for $\mathcal{F}=\mathcal{F}_\nu$. 

\begin{theorem}\label{thm13.5} Let $\nu$ be an arithmetic morphism with signature $(2,\psi)$, where $\mathrm{cond}(\psi)=p^m$ and $m\geq 2$, and $\phi\colon K^\times\backslash\A_K^\times\rightarrow\overline\Q^\times$ be of infinity type $(1,-1)$ and $\mathrm{cond}(\phi)=p^n$ with $n\geq m$. 
Then 
\[\mathscr{L}_{\I,\boldsymbol{\xi}}^{\mathrm{alg}}(\nu,\hat\phi^{-1})=\frac{\epsilon(\phi)}
{{\xi}_{\nu,\mathfrak{p}}(p^n)\cdot p^n}\cdot\sum_{\mathfrak{a}\in\Pic(\mathcal{O}_{cp^n})}
(\hat\xi_\nu^{-1}\hat\chi_\nu\hat\phi)(a)d^{-1}\mathcal{F}_{\nu , x_\infty}^{[p]}\left(x_{cp^n,m}(\mathfrak{a}^{-1})\right)
.\] 
\end{theorem}

\begin{proof} We first relate $\mathscr{L}_{\I,\boldsymbol{\xi}}^{\mathrm{alg}}(\nu,\hat\phi^{-1})$ to the Coleman primitive. 
Since $\hat\phi\colon \Gamma_\infty\rightarrow\overline\Q_p$ 
has Hodge--Tate weight $w=1$ (so $\hat{\phi}^{-1}$ has Hodge--Tate weight $w=-1$) and conductor $n>1$, 
from  \eqref{prop12.2} we have 
\begin{equation}\label{13.5eq1}
\begin{split}
\mathscr{L}_{\I,\boldsymbol{\xi}}^{\mathrm{alg}}(\nu,\hat\phi^{-1})
&=
\mathrm{sp}_{\nu,\hat\phi}\left(\mathcal{L}_{\omega_\I}^{\widetilde{\Gamma}_\infty}(\res_\mathfrak{P}(\mathfrak{Z}_{\boldsymbol{\xi}}))\right)\\
&=
\mathcal{E}(\hat\phi^{-1},\nu)
\cdot
(\omega_{\mathcal{F}_\nu}\otimes\phi^{-1})\left(\log(\mathrm{sp}_{\nu,\hat\phi}(\res_{\mathfrak{P}}(\mathfrak{Z}_{\boldsymbol{\xi}}))\right).\end{split}
\end{equation} 
Using that the characters $\xi_\nu$ and $\phi$ has conductors $p^m$ and $p^n$ respectively, and $n\geq m$, by  \eqref{eq13.2} we have
{\footnotesize{\begin{equation}\label{13.5eq2}
\begin{split}
\mathscr{L}_{\I,\boldsymbol{\xi}}^{\mathrm{alg}}(\nu,\hat\phi^{-1})
&=
\mathcal{E}(\hat\phi^{-1},\nu)\sum_{\sigma\in\Gal(H_{cp^n}/H_c)}(\hat\xi_\nu^{-1}\hat\phi^{-1})(\sigma)
\log(\mathrm{sp}_{\nu}(\res_{\mathfrak{p}}(\cor_{H_c/K}(\mathfrak{X}_{cp^\infty}^\sigma))))(\omega_{\mathcal{F}_\nu^*})
\\
&=
\mathcal{E}(\hat\phi^{-1},\nu)\sum_{\sigma\in\Gal(H_{cp^n}/K)}(\hat\xi_\nu^{-1}\hat\phi^{-1})(\sigma)
\log(\mathrm{sp}_{\nu}(\res_{\mathfrak{p}}(\mathfrak{X}_{cp^\infty}^\sigma)))(\omega_{\mathcal{F}_\nu^*})
\\
&=
\mathcal{E}(\hat\phi^{-1},\nu)\sum_{\sigma\in\Gal(H_{cp^n}/K)}\nu(\mathbf{a}_p)^{-n}(\hat\xi_\nu^{-1}\hat\phi^{-1})(\sigma)
\log(\mathrm{sp}_{\nu}(\res_{\mathfrak{p}}(\mathfrak{X}_{cp^n}^\sigma)))(\omega_{\mathcal{F}_\nu^*})
\\
&=
\mathcal{E}(\hat\phi^{-1},\nu)\left(\frac{p}{\nu(\mathbf{a}_p)}\right)^m\sum_{\sigma\in\Gal(H_{cp^n}/K)}
\nu(\mathbf{a}_p)^{-n}(\hat\xi_\nu^{-1}\hat\phi^{-1})(\sigma)\log(\varrho_m(Q_{cp^{n},m}^\sigma))(\omega_{\mathcal{F}_\nu^*})
\\
&=
\mathcal{E}(\hat\phi^{-1},\nu)\left(\frac{p}{\nu(\mathbf{a}_p)}\right)^m\sum_{\sigma\in\Gal(H_{cp^n}/K)}\nu(\mathbf{a}_p)^{-n}
(\hat\xi_\nu^{-1}\hat\chi_\nu\hat\phi^{-1})(\sigma)\log(\varrho_m( P_{cp^{n+m},m}^\sigma))(\omega_{\mathcal{F}_\nu^*}).
\end{split}\end{equation}}}
Let $F_\infty^*$ be the Coleman primitive of $\omega_{\mathcal{F}_\nu^*}$ on $\mathcal{W}_\infty(p^m)$ which vanishes at $x_\infty$. 
It follows from \eqref{eq9.14} that 
\[\log(\varrho_m( P_{cp^{n+m},m}^\sigma))=\log(j_m( P_{cp^{n+m},m}^\sigma)).\]
Applying Corollary \ref{CoroCole} (and using linearity) we thus obtain 
\[\mathscr{L}_{\I,\boldsymbol{\xi}}^{\mathrm{alg}}(\nu,\hat\phi^{-1})=\mathcal{E}(\hat\phi^{-1},\nu)\left(\frac{p}{\nu(\mathbf{a}_p)}\right)^m\sum_{\sigma\in\Gal(H_{cp^n}/K)}\nu(\mathbf{a}_p)^{-n}
(\hat\xi_\nu^{-1}\hat\chi_\nu\hat\phi^{-1})(\sigma)F_\infty^*({P}_{cp^{n+m},m}^\sigma).\]
On the other hand, since ${P}_{cp^{n+m},m}$ is defined over the subfield $H_{cp^{n-1}}(\zeta_{p^n})$ of $L_{cp^n}$, and $\chi_\nu$ is a primitive character modulo $p^n$,
we see that, after setting $\varphi=\nu(\mathbf{a}_p)^{-n}\hat\xi_\nu^{-1}\hat\chi_\nu\hat\phi^{-1}$ to simplify the notation,  
\begin{equation} \begin{split}
\sum_{\sigma}
\varphi(\sigma)F_\infty^*({P}_{cp^{n+m},m}^\sigma)&= 
\sum_{\sigma}
\varphi(\sigma)F_\infty^*({P}_{cp^{n+m},m}^\sigma)-\frac{\nu(\mathbf{a}_p)}{p}\sum_\sigma\varphi(\sigma)F^*_\infty(\phi({P}_{cp^{n+m},m}^\sigma))\\
&=\sum_{\sigma}
\varphi(\sigma)\Pi(\phi^*)F_\infty^*({P}_{cp^{n+m},m}^\sigma)\\
&=\sum_{\sigma}
\varphi (\sigma) d^{-1}{\mathcal{F}_{\nu , x_\infty}^{[p]}}({P}_{cp^{n+m},m}^\sigma)\\
\end{split}\end{equation}
where the sum is over all $\sigma\in\Gal(H_{cp^n}/K)$, and the last equation follows from \eqref{Col-dminus1} and the fact that $d^{-1}\omega_{\mathcal{F}^{[p]}}=d^{-1}\omega_{{\mathcal{F}^*}^{[p]}}$. Therefore, 
{\footnotesize{
\begin{equation}\label{13.5eq4}\mathscr{L}_{\I,\boldsymbol{\xi}}^{\mathrm{alg}}(\nu,\hat\phi^{-1})=\mathcal{E}(\hat\phi^{-1},\nu)\left(\frac{p}{\nu(\mathbf{a}_p)}\right)^m\cdot\sum_{\sigma\in\Gal(H_{cp^n}/K)}\nu(\mathbf{a}_p)^{-n}(\hat\xi_\nu^{-1}\hat\chi_\nu\hat\phi^{-1})(\sigma)d^{-1}{\mathcal{F}_{\nu, x_\infty}^{[p]}}({P}_{cp^{n+m},m}^\sigma).\end{equation} }}
We now observe that 
\begin{equation}\label{13.5eq3} U_pF_\infty^*=\left(\frac{\nu(\mathbf{a}_p)}{p}\right)F_\infty^*.\end{equation}
Since ${P}_{cp^{n+m},m}=U_p^m{P}_{cp^n,m}=U_p^m x_{cp^n,m}(1)$, it follows from \eqref{13.5eq3} and \eqref{13.5eq4} that 
(use Shimura's reciprocity law to keep trace of the Galois action) 
\begin{equation}\label{13.5eq5}\mathscr{L}_{\I,\boldsymbol{\xi}}^{\mathrm{alg}}(\nu,\hat\phi^{-1})=\mathcal{E}(\hat\phi^{-1},\nu)\sum_{\sigma\in\Gal(H_{cp^n}/K)}\nu(\mathbf{a}_p)^{-n}(\hat\xi_\nu^{-1}\hat\chi_\nu\hat\phi^{-1})(\sigma)d^{-1}{\mathcal{F}_{\nu , x_\infty}^{[p]}}(x_{cp^{n},m}(1)^\sigma).\end{equation}
Since $\Phi_\nu=\nu(\mathbf{a}_p)({\xi}_{\nu, \mathfrak{p}}(p)p)^{-1}$, we have 
\[
\mathcal{E}(\hat\phi^{-1},\nu)
=\frac{\epsilon(\phi)\nu(\mathbf{a}_p)^n}
{{\xi}_{\nu,\mathfrak{p}}(p^n)\cdot p^n}
\] and the result follows. 
\end{proof}

 \subsection{Recipocity law} \label{RL}
Fix an algebraic Hecke character $\lambda\colon K^\times\backslash\A^\times_K\rightarrow\overline\Q^\times$ as in \S\ref{sec.BigCharacters} and consider the family of Hecke characters $\boldsymbol{\xi}$ obtained from $\lambda$. 
We fix $\nu$ and $\phi$ as in the proof of Theorem \ref{thm13.5}; so $\nu$ is an arithmetic morphism of signature $(2,\psi)$ with $\mathrm{cond}(\psi)=p^m$ for some integer $m\geq2$, and $\hat{\phi}$
is the $p$-adic avatar of a Hecke character $\phi\colon K^\times\backslash\A_K^\times\rightarrow\overline\Q^\times$ of infinity type $(1,-1)$ and conductor $p^n$ for some integer $n\geq m$, so the associated Galois character $\tilde{\phi}$ factors through $\widetilde{\Gamma}_\infty$.

\begin{proposition}\label{prop14.1}
Let $\nu$ and $\phi$ be as before. Then
\[\mathscr{L}_{\I,\boldsymbol{\xi}}^{\mathrm{alg}}(\nu,\hat{\phi}^{-1})=\left(\frac{\phi_\p(-1)}
{\sqrt{-D_K}}\right) \mathscr{L}^\mathrm{an}_{\I,\boldsymbol{\xi}}(\nu,\hat{\phi}^{-1}).\] 
\end{proposition}
\begin{proof}
The character $\hat{\xi}_\nu$ has infinity type $(1,-1)$, so the character $\varphi=\hat\xi_\nu\hat\phi^{-1}$ has infinity type $(0,0)$, thus finite order.
Recall that, by definition, 
\[
\nu(\mathscr{L}^\mathrm{alg}_{\I,\boldsymbol{\xi}}({\tilde\phi}^{-1}))= \sum_{\mathfrak{a}\in \Pic{\mathcal{O}_c}} \hat{\xi}_\nu\hat{\chi}^{-1}_\nu(\mathfrak{a})\mathrm{N}(\mathfrak{a})^{-1} \int_{\Z_p^\times} 
\tilde{\phi}^{-1} | [\mathfrak{a}] (z) d\mu_{\omega_\I,{\mathfrak{a}}}(z).
\]
Since $\phi^{-1}$ has infinity type $(-1,1)$ and we chose the representatives $\mathfrak{a}$ such that $((p),\mathfrak{a})=1$,
then 
\[\tilde{\phi}^{-1}|[\mathfrak{a}](z)=\tilde{\phi}^{-1}(\mathrm{rec}_{K}(a)\mathrm{rec}_{K,\mathfrak{p}}(z))=
\hat{\phi}^{-1}(ai_\mathfrak{p}(z))=
\phi^{-1}(\mathfrak{a})\phi_{\mathfrak{p}}^{-1}(z)z^{-1},\] 
where recall that $\mathfrak{a}=a\widehat{\mathcal{O}}_c\cap K$ and $i_\mathfrak{p}\colon \Z_p^\times\rightarrow\widehat{K}^\times$ denotes the map which takes $z\in \Z_p^\times \cong \mathcal{O}_{K,\mathfrak{p}}^\times$ to the element $i_\mathfrak{p}(z)$ with  $\mathfrak{p}$-component equal to $z$ and trivial components at all the other places. 
Hence,
\[
\nu(\mathscr{L}^\mathrm{alg}_{\I,\boldsymbol{\xi}}(\hat{\phi}^{-1}))
= \sum_{\mathfrak{a}\in \Pic({\mathcal{O}_c})} \hat{\xi}_\nu\hat{\chi}^{-1}_\nu(\mathfrak{a}){\mathrm{N}}(\mathfrak{a})^{-1} \phi^{-1}(\mathfrak{a}) \int_{\Z_p^\times}
\phi_{\mathfrak{p}}^{-1}(z)z^{-1} d\mu_{\omega_\I,\mathfrak{a}}(z).
\]
By \cite[\S 3.5, (5)]{hida-elementary} (and \cite[(6.7)]{Magrone} for negative exponents), we have
\[
\nu(\mathscr{L}^\mathrm{alg}_{\I,\boldsymbol{\xi}}(\hat{\phi}^{-1})) = \sum_{\mathfrak{a}\in \Pic({\mathcal{O}_c})} \hat{\xi}_\nu\hat{\chi}^{-1}_\nu(\mathfrak{a}){\mathrm{N}}(\mathfrak{a})^{-1} \phi^{-1}(\mathfrak{a})\cdot 
([
\phi_{\mathfrak{p}}^{-1}] d^{-1} {\mathcal{F}}_{\nu,\mathfrak{a}}^{{[p]}}(T_{x(\mathfrak{a})}) )|_{T_{x(\mathfrak{a})}=0},
\]
Set $C_0(\xi_\nu,\chi_\nu,\phi)=\sqrt{-D_K}p^{-n}\mathfrak{g}(
\phi^{-1}_\mathfrak{p})$. Applying \eqref{mag4.5}, and using the equality of $T_{x(\mathfrak{a})}$-expansions \[\mathrm{N}(\mathfrak{a})\sqrt{-D_K}(d^{-1}{\mathcal{F}}_{\nu}^{[p]})_{\mathfrak{a}}=d^{-1}{\mathcal{F}}_{\nu,\mathfrak{a}}^{[p]},\] we see that 
\begin{equation*}
\footnotesize
{\nu(\mathscr{L}^\mathrm{alg}_{\I,\boldsymbol{\xi}}(\hat{\phi}^{-1})) =
C_0(\xi_\nu,\chi_\nu,\phi)\sum_{\mathfrak{a}\in \Pic({\mathcal{O}_c})} \sum_{u \in (\Z/p^n\Z)^\times}
(\hat{\xi}_\nu\hat{\chi}^{-1}_\nu\phi^{-1})(\mathfrak{a})
\phi_\mathfrak{p}(u)d^{-1}\mathcal{F}_{\nu ,x(\mathfrak{a})}^{[p]}(x(\mathfrak{a})\star \mathbf{n}(u/p^n)).
}
\end{equation*}
Here $d^{-1}\mathcal{F}_{\nu ,x(\mathfrak{a})}^{[p]}$ denotes the overconvergent modular form in Definition \ref{def9.6} for $\mathcal{F}=\mathcal{F}_\nu$ where the basis point is taken to be $x(\mathfrak{a})$ instead of the point $x_\infty$ fixed before. 
Since $d^{-1}\mathcal{F}_{\nu ,x(\mathfrak{a})}^{[p]}$ has weight $0$ and character $\psi$, using \eqref{CHeq} we obtain 
(recall that $\mathfrak{a}=a\mathcal{O}_c^\times\cap K$) 
\[d^{-1}\mathcal{F}_{\nu , x(\mathfrak{a})}^{[p]}(x(\mathfrak{a})\star \mathbf{n}(u/p^n)) = \psi^{-1}(\langle u\rangle)d^{-1}\mathcal{F}_{\nu ,x(\mathfrak{a})}^{[p]}
\left(\left[\left(\iota_K,a^{-1}i_\mathfrak{p}(u/p^n)\xi^{(n)}\right)\right]\right).\]
To simplify the notation, we temporarily write 
$z_n(a)=\left(\left[\left(\iota_K,a^{-1}i_\mathfrak{p}(u/p^n)\xi^{(n)}\right)\right]\right).$
We have  
\[
\begin{split}
\hat{\xi}_\nu^{-1}\hat{\chi}_\nu(a^{-1}i_{\mathfrak{p}}(u/p^n)) &= 
\hat{\xi}_\nu\hat{\chi}_\nu^{-1}(\mathfrak{a})\hat{\xi}_\nu^{-1} \hat{\chi}_\nu(i_{\mathfrak{p}}(u/p^{n})),\\
\hat{\phi}(a^{-1}i_{\mathfrak{p}}(u/p^{n})) &= \phi^{-1}(\mathfrak{a})\phi_{\mathfrak{p}}(u)\phi_{\mathfrak{p}}(p^{-n})up^{-n}.
\end{split}
\]
By \eqref{chi_k}, 
$\chi_{\nu, \mathfrak{p}}^{-1}(z)=\psi^{1/2}(\langle z \rangle) $
for $z \in \Z_p^\times \cong \mathcal{O}_{K,\mathfrak{p}}^\times$. Also, $
{\chi}_{\nu,\mathfrak{p}}^{-1}(p^{-n})= \psi^{1/2}(\langle p^np^{-n}\rangle)=1$ 
and by \eqref{hatxi}, $\xi_{\nu,\mathfrak{p}}(z)=\psi^{1/2}(\langle z \rangle)$ for $z \in \Z_p^\times \cong \mathcal{O}_{K,\mathfrak{p}}^\times$.
Therefore, after setting 
\[C(\xi_\nu,\chi_\nu,\phi)=C_0(\xi_\nu,\chi_\nu,\phi) \xi_{\nu,\mathfrak{p}}(p^{-n})\phi_\mathfrak{p}(p^n)=
\frac{\sqrt{-D_K}\cdot\mathfrak{g}(\phi^{-1}_\mathfrak{p})p^{-n}\phi_{\mathfrak{p}}(p^n)}
{\xi_{\nu,\mathfrak{p}}(p^n)}
\]
we have 
\[\begin{split}
\nu(\mathscr{L}^\mathrm{alg}_{\I,\boldsymbol{\xi}}(\hat{\phi}^{-1})) &=
C(\xi_\nu,\chi_\nu,\phi)
\sum_{\mathfrak{a}\in \Pic({\mathcal{O}_c})} \sum_{u \in (\Z/p^n\Z)^\times}
(\hat{\xi}_\nu^{-1}\hat{\chi}_\nu\hat{\phi})(a^{-1}i_{\mathfrak{p}}(u/p^{n}))d^{-1}\mathcal{F}_{\nu ,x(\mathfrak{a})}^{[p]}(z_n(a))\\
&=C(\xi_\nu,\chi_\nu,\phi)\sum_{\mathfrak{a}\in \Pic({\mathcal{O}_{cp^n}})}
(\hat{\xi}_\nu^{-1}\hat{\chi}_\nu\hat{\phi})(a)d^{-1}\mathcal{F}_{\nu ,x(\mathfrak{a})}^{[p]}
\left([(\iota_K,a^{-1}\xi^{(n)})]\right)\\
&=C(\xi_\nu,\chi_\nu,\phi)
\sum_{\mathfrak{a}\in \Pic({\mathcal{O}_{cp^n}})}
(\hat{\xi}_\nu^{-1}\hat{\chi}_\nu\hat{\phi})(a)\cdot
d^{-1}\mathcal{F}_{\nu, x(\mathfrak{a})}^{[p]}\left(x_{cp^n,m}(\mathfrak{a}^{-1})\right)\end{split}
\]
where for each $\mathfrak{a}\in\Pic(\mathcal{O}_{cp^n})$ we let 
$\mathfrak{a}=a\widehat{\mathcal{O}}_{cp^n}\cap K$. 
We now observe that 
$d^{-1}{\mathcal{F}}_{\nu ,x_\infty}^{{[p]}} $ and 
$d^{-1}{\mathcal{F}}_{\nu,x(\mathfrak{a})}^{{[p]}} $ 
differ by a constant; however, since the character $\hat\chi_\nu$ is primitive,
we can replace the first with the second in the previous formula. 
Comparing with Theorem \ref{thm13.5}, the result follows from the equality 
$\epsilon(\phi_\p)=\mathfrak{g}(\phi_\p^{-1})\phi_\p(-p^{n})$. 
\end{proof}

\begin{theorem}\label{thmL=L}
Let $\sigma_{-1,\p}\defeq \mathrm{rec}_\p(-1)$. Then in $\widetilde{\I}[[\widetilde{\Gamma}_\infty]]$ we have: 
\[\mathscr{L}_{\I,\boldsymbol{\xi}}^{\mathrm{alg}}=\left(
\frac{\sigma_{-1,\p}}{\sqrt{-D_K}}\right)\cdot\mathscr{L}_{\I,\boldsymbol{\xi}}^{\mathrm{an}}.\] \end{theorem}

\begin{proof}
The equality holds when specialized at arithmetic primes of weight $2$ by Proposition \ref{prop14.1}, and since these arithmetic primes are dense, the claimed equality holds in  
$\widetilde{\I}[\mathfrak{j}^{-1}][[\widetilde{\Gamma}_\infty]$.  
Since the right hand side belongs to $\widetilde{\I}[[\widetilde{\Gamma}_\infty]]$, the equality takes place in this ring, concluding the proof. 
\end{proof}

\subsection{Aritmetic applications}\label{secArithApp}
The root number of the functional equations of the $L$-functions of $\mathcal{F}_\nu$ is constant save for a finite number of \emph{exceptional} specializations; we call this common value the \emph{generic root number} of the Hida family $\I$ (see \cite[\S9.2]{LV-MM} for more details, references and the connection of this root number with Greenberg's conjecture).  

\begin{corollary}\label{coro} Assume that the generic root number of $\I$ is $+1$. Then 
$\mathfrak{Z}_{c}$ is not $\I$-torsion. 
\end{corollary} 

\begin{proof} Since $w=+1$, it is known that the complex $L$-function of $\mathcal{F}_\nu$ twisted by $\xi_\nu$ does not vanish for infinitely many such choices; hence, $\mathscr{L}_{\I,\boldsymbol{\xi}}^{\mathrm{alg}}$ is not zero. 
Therefore, the same is true for 
$\mathscr{L}_{\I,\boldsymbol{\xi}}^{\mathrm{alg}}$; specializations at any $\nu\colon \I\rightarrow\overline\Q_p$  therefore have only finitely many zeroes. If $\mathfrak{Z}_{c}$ is torsion, then there would be specializations having infinitely many zeroes, which is a contradiction.   
\end{proof}

Let $\widetilde{H}^1_f(K,\mathbf{T}^\dagger)$ denote Nekov\'{a}\v{r}'s extended Selmer group of the representation $\mathbf{T}^\dagger$ (\cite[\S2.4]{Howard-Inv}, \cite[\S5.6]{LV-MM}). 
For each arithmetic morphism $\nu$, let $H^1_f(K,\mathbf{T}_\nu^\dagger)$ denote 
the Bloch-Kato Selmer group of the self-dual representation $\mathbf{T}_\nu^\dagger$ 
obtained by specializing $\mathbf{T}^\dagger$ at $\nu$ (thus, if $F_\nu$ denotes as before the residue field of $\nu$, then $\mathbf{T}_\nu^\dagger$
becomes a $2$-dimensional $F_\nu$-vector space). 

\begin{corollary}
    \label{coro1} 
 Assume that the generic root number of $\I$ is $+1$. Then  $\widetilde{H}^1_f(K,\mathbf{T}^\dagger)$ is a $\I$-module of 
 rank $1$ and $\dim_{F_\nu}(H^1_f(K,T_\nu^\dagger))=1$ for all except possibly finitely many
 arithmetic primes $\nu$. 
\end{corollary}

\begin{proof}
   This follows from Corollary \ref{coro} using \cite[Theorem 10.4]{LV-MM} and  
   \cite[Theorem 10.6]{LV-MM}. 
\end{proof}

 \subsection{Specializations} \label{secspecialization}
 Let $\mathcal{F}_k^\sharp$ be a $p$-ordinary newform on ${X}_0$ of weight 
$k\equiv 2\mod 2(p-1)$ and trivial character, and consider the self-dual twist 
$V_{\mathcal{F}_k^\sharp}^\dagger=V_{\mathcal{F}_k^\sharp}^*(k/2)$ of the Deligne Galois representation $V_{\mathcal{F}_k^\sharp}^*$ associated with $\mathcal{F}_k^\sharp$.
Let $W_k$ denote the generalized quaternionic Kuga--Sato variety constructed in \cite[\S5.1-\S5.4]{Magrone} and let  
\[\Phi_{\mathcal{F}_k^\sharp}^{\et}\colon \epsilon_{W}\operatorname{CH}^{k-1}(W_{k}/L)_\Q\longrightarrow 
H^1\Big(L,V_{\mathcal{F}_k^\sharp}^\dagger\Big)\] be the $p$-adic Abel--Jacobi map, where $L$ is a sufficiently 
big number field (\cite[\S5.5]{Magrone}).  
Let $\Delta_{cp^n}$ be the generalized Heegner cycle in \cite[\S5.6, \S5.7]{Magrone}; then we have generalized Heegner classes  
$\Phi^\text{\'et}_{\mathcal{F}_k^\sharp}(\Delta_{cp^n})$ for $n\geq 0$  
(\emph{cf.} \cite[\S5.8]{Magrone}). 
Let
$u_c=\sharp(\mathcal{O}_c^\times)/2$ and $\alpha$ the unit root of the Hecke polynomial at $p$ acting on $\mathcal{F}_k^\sharp$. We normalize these points to obtain a non-compatible family of quaternionic generalized Heegner classes by setting (here recall that $\Frob_\p$ is the Frobenius element at $\p$ and similarly denote $\Frob_{\bar\p}$ the Frobenius element at $\bar\p$) 
\begin{itemize} 
\item ${z}_{\mathcal{F}_k^\sharp,c}=\tfrac{1}{u_c}\left( 1- \tfrac{p^{k/2-1}}{\alpha}\Frob_\p \right) \left( 1- \tfrac{p^{k/2-1}}{\alpha}\Frob_{\bar{\p}} \right)\cdot \Phi^\text{\'et}_{\mathcal{F}_k^\sharp}(\Delta_{c})$;
\item ${z}_{\mathcal{F}_k^\sharp,cp^{n}}=\left( 1- \tfrac{p^{k-2}}{\alpha}\right)\cdot\Phi^\text{\'et}_{\mathcal{F}_k^\sharp}(\Delta_{cp^n})$ for $n\geq 1$. \end{itemize}
Then 
$\cores_{H_{cp^{n}}/H_{cp^{n-1}}}(z_{\mathcal{F}_k^\sharp,cp^n})=
\alpha\cdot z_{\mathcal{F}_k^\sharp,cp^{n-1}}$ for all $n\geq 1$ (\cite[\S7.1.2]{Magrone}) 
and we can define (using Shapiro's lemma for the isomorphisms) 
\begin{itemize}
\item $\boldsymbol{x}_{c}^\sharp=\invlim_n\alpha^{-n}z_{\mathcal{F}_k^\sharp,cp^n}\in 
H^1_\mathrm{Iw}\Big(\Gamma_\infty,V_{\mathcal{F}_k^\sharp}^\dagger\Big)\cong H^1\Big(H_c,V_{\mathcal{F}_k^\sharp}^\dagger\otimes \mathcal{O}[[\Gamma_\infty]]\Big)$; 
\item $\boldsymbol{z}_{c}^\sharp=\mathrm{cores}_{H_c/K}(\boldsymbol{x}^\sharp_c)\in H^1_\mathrm{Iw}\Big(\widetilde{\Gamma}_\infty,{V}^\dagger_{\mathcal{F}_k^\sharp}\Big)\cong H^1\Big(K,V_{\mathcal{F}_k^\sharp}^\dagger\otimes \mathcal{O}[[\widetilde\Gamma_\infty]]\Big)$.  
\end{itemize}
For any character $\xi\colon \widetilde\Gamma_\infty\rightarrow\overline\Q_p^\times$, we can then consider the specialization map, and obtain an element $\boldsymbol{z}_\xi^\sharp\in H^1\Big(H_c,V_{\mathcal{F}_k^\sharp,\xi}^\dagger\Big)$; here $V_{\mathcal{F}_k^\sharp,\xi}^\dagger=V_{\mathcal{F}_k^\sharp}^\dagger\otimes\xi$ (\emph{cf.} \cite[\S5.9]{Magrone}).  

Let $\boldsymbol{\mathcal{F}}$ be the quaternionic Hida family passing through the modular form $\mathcal{F}_k$. We also assume that the residual $p$-adic representation $\bar\rho_{\mathcal{F}_k}$ is irreducible, $p$-ordinary and $p$-distinguished. Define 
$\mathfrak{z}_c=\boldsymbol{\vartheta}^{-1}\left(\frac{-\sqrt{-D_K}}{c^{2}}\right)\mathfrak{Z}_c$ and write as before $\mathfrak{z}_c(\nu)$ for $\nu(\mathfrak{z}_c)$. 

\begin{theorem}\label{teospec} For all $\nu$ of weight $k\equiv 2\mod{2(p-1)}$, we have
\[\nu(\mathscr{L}_{\I,\boldsymbol{\xi}}^{\mathrm{alg}})=
\frac{ \sqrt{-D_K}^{k_\nu/2-1}}{c^{k_\nu-2}}
\mathcal{L}^{\widetilde\Gamma_\infty}_{\mathcal{F}_k^\sharp,\xi_k}(\mathrm{res}_\p(\boldsymbol{z}_{\xi_k}^\sharp)).\]
\end{theorem}

\begin{proof} 
Let $\nu$ be as in the statement, let $\mathcal{F}_k=\mathcal{F}_{\nu}$ and $\mathcal{F}_k^\sharp$ the form whose ordinary $p$-stabilization is $\mathcal{F}_k$. 
Let $\hat\phi\colon \widetilde\Gamma_\infty\rightarrow\overline{\Q}_p^\times$ be the $p$-adic avatar of a Hecke character $\phi$ of infinity type $(k/2,-k/2)$; then $\xi=\hat\phi\hat\xi_k^{-1}$ is a finite order character. Consider the map $\mathcal{L}^{\widetilde\Gamma_\infty}_{\mathcal{F}_k^\sharp,\xi_k}$ obtained by composing the map
$\mathcal{L}^{\Gamma_\infty}_{\mathcal{F}_k^\sharp,\xi_k}$ in \eqref{PR} with the canonical map arising from the inclusion $\Gamma_\infty\hookrightarrow\widetilde{\Gamma}_\infty$.  
Combining Theorem \ref{thmL=L}, \eqref{int-1} and \cite[Theorem 7.2]{Magrone}, we have:
\begin{align*}
    \nu(\mathscr{L}_{\I,\boldsymbol{\xi}}^{\mathrm{alg}})(\hat\phi^{-1}) &= \nu \left(
\frac{\sigma_{-1,\p}}{\sqrt{-D_K}}\right)\cdot \nu(\mathscr{L}_{\I,\boldsymbol{\xi}}^{\mathrm{alg}})(\hat\phi^{-1}) &\text{ (Theorem \ref{thmL=L})}\\
&= \left(\frac{\sigma_{-1,\p}}{\sqrt{-D_K}}\right)
\left(c^{-k/2 +1} \mathscr{L}_{\mathcal{F}_\nu, \xi_\nu}(\hat\phi^{-1})\right) &\text{ (Equation \eqref{int-1})}\\
&= \left(\frac{\sigma_{-1,\p}}{\sqrt{-D_K}}\right)
\left(c^{-k/2 +1} \mathscr{L}_{\mathcal{F}_\nu^\sharp, \xi_\nu}(\hat\phi^{-1})\right) &\text{ (\cite[Lemma 6.1]{LMW})}\\
&= (-1)^{k/2-1}\cdot
\frac{ \sqrt{-D_K}^{k/2-1}}{c^{k-2}}
\mathcal{L}^{\widetilde\Gamma_\infty}_{\mathcal{F}_k^\sharp,\xi_k}(\mathrm{res}_\p(\boldsymbol{z}_{\xi_k}^\sharp))(\hat\phi^{-1})&\text{ (\cite[Theorem 7.2]{Magrone}).}
\end{align*}
 By assumption, $(k-2)/2$ is divisible by the even number $p-1$, so the first factor on the RHS disappears. 
Since this equation holds for infinitely many $\phi$, the result follows.
\end{proof}

\bibliographystyle{amsalpha}
\bibliography{references}

@article{LMW,
	author = {Longo, Matteo and Magrone, Paola and Walchek, Eris R.},
	journal = {preprint},
	title = {On quaternionic ordinary families of modular forms and $p$-adic ${L}$-functions},
	url = { },
	volume = {},
	year = {2026},
	}

@article{OhtaMA,
	author = {Ohta, Masami},
	doi = {10.1007/s002080000119},
	fjournal = {Mathematische Annalen},
	issn = {0025-5831},
	journal = {Math. Ann.},
	mrclass = {11F33 (11F67 11G18 11R23)},
	mrnumber = {1800769},
	mrreviewer = {Andrea Mori},
	number = {3},
	pages = {557--583},
	title = {Ordinary {$p$}-adic \'{e}tale cohomology groups attached to towers of elliptic modular curves. {II}},
	url = {https://doi.org/10.1007/s002080000119},
	volume = {318},
	year = {2000}}

@article{OhtaC,
	author = {Ohta, Masami},
	doi = {10.1023/A:1000556212097},
	fjournal = {Compositio Mathematica},
	issn = {0010-437X},
	journal = {Compositio Math.},
	mrclass = {11F33 (11F67 11G18)},
	mrnumber = {1674001},
	mrreviewer = {Andrea Mori},
	number = {3},
	pages = {241--301},
	title = {Ordinary {$p$}-adic \'{e}tale cohomology groups attached to towers of elliptic modular curves},
	url = {https://doi.org/10.1023/A:1000556212097},
	volume = {115},
	year = {1999}}

@article{Ohta-ES,
	author = {Ohta, Masami},
	doi = {10.1515/crll.1995.463.49},
	fjournal = {Journal f\"{u}r die Reine und Angewandte Mathematik. [Crelle's Journal]},
	issn = {0075-4102},
	journal = {J. Reine Angew. Math.},
	mrclass = {11F33 (11F67)},
	mrnumber = {1332907},
	mrreviewer = {Andrea Mori},
	pages = {49--98},
	title = {On the {$p$}-adic {E}ichler-{S}himura isomorphism for {$\Lambda$}-adic cusp forms},
	url = {https://doi.org/10.1515/crll.1995.463.49},
	volume = {463},
	year = {1995}}

@incollection{Katz-Dwork,
	author = {Katz, Nicholas},
	booktitle = {S\'{e}minaire {B}ourbaki, 24{\`e}me ann\'{e}e (1971/1972)},
	mrclass = {14G13 (14G20)},
	mrnumber = {498577},
	pages = {Exp. No. 409, pp. 167--200},
	publisher = {Springer, Berlin-New York},
	series = {Lecture Notes in Math., Vol. 317},
	title = {Travaux de {D}work},
	url = {https://mathscinet.ams.org/mathscinet-getitem?mr=498577},
	year = {1973}}

@article{DR,
	author = {Darmon, Henri and Rotger, Victor},
	doi = {10.1090/jams/861},
	fjournal = {Journal of the American Mathematical Society},
	issn = {0894-0347},
	journal = {J. Amer. Math. Soc.},
	mrclass = {11G05 (11G40)},
	mrnumber = {3630084},
	mrreviewer = {Rolf Berndt},
	number = {3},
	pages = {601--672},
	title = {Diagonal cycles and {E}uler systems {II}: {T}he {B}irch and {S}winnerton-{D}yer conjecture for {H}asse-{W}eil-{A}rtin {$L$}-functions},
	url = {https://mathscinet.ams.org/mathscinet-getitem?mr=3630084},
	volume = {30},
	year = {2017}}

@article{EmBr,
	author = {Breuil, Christophe and Emerton, Matthew},
	fjournal = {Ast\'{e}risque},
	isbn = {978-2-85629-282-2},
	issn = {0303-1179},
	journal = {Ast\'{e}risque},
	mrclass = {22E50 (11F70 11F80)},
	mrnumber = {2667890},
	mrreviewer = {Anne-Marie H. Aubert},
	number = {331},
	pages = {255--315},
	title = {Repr\'{e}sentations {$p$}-adiques ordinaires de {${\rm GL}_2(\bold Q_p)$} et compatibilit\'{e} local-global},
	url = {https://mathscinet.ams.org/mathscinet-getitem?mr=2667890},
	year = {2010}}

@article{Tsiji,
	author = {Tsuji, Takeshi},
	doi = {10.1007/s002220050330},
	fjournal = {Inventiones Mathematicae},
	issn = {0020-9910},
	journal = {Invent. Math.},
	mrclass = {14F30 (14F20)},
	mrnumber = {1705837},
	mrreviewer = {Abdellah Mokrane},
	number = {2},
	pages = {233--411},
	title = {{$p$}-adic \'{e}tale cohomology and crystalline cohomology in the semi-stable reduction case},
	url = {https://mathscinet.ams.org/mathscinet-getitem?mr=1705837},
	volume = {137},
	year = {1999}}

@incollection{HK,
	author = {Hyodo, Osamu and Kato, Kazuya},
	fjournal = {Ast\'{e}risque},
	issn = {0303-1179},
	journal = {Ast\'{e}risque},
	mrclass = {14F30 (14F20 14F40 14G20)},
	mrnumber = {1293974},
	mrreviewer = {Adolfo Quir\'{o}s},
	note = {P\'{e}riodes $p$-adiques (Bures-sur-Yvette, 1988)},
	number = {223},
	pages = {221--268},
	title = {Semi-stable reduction and crystalline cohomology with logarithmic poles},
	url = {https://mathscinet.ams.org/mathscinet-getitem?mr=1293974},
	year = {1994}}

@article{CI,
	author = {Coleman, Robert and Iovita, Adrian},
	fjournal = {Ast\'{e}risque},
	isbn = {978-2-85629-282-2},
	issn = {0303-1179},
	journal = {Ast\'{e}risque},
	mrclass = {11G30 (11F80 14F30)},
	mrnumber = {2667889},
	mrreviewer = {Hui June Zhu},
	number = {331},
	pages = {179--254},
	title = {Hidden structures on semistable curves},
	url = {https://mathscinet.ams.org/mathscinet-getitem?mr=2667889},
	year = {2010}}

@article{Carayol-Shimura,
	author = {Carayol, Henri},
	fjournal = {Compositio Mathematica},
	issn = {0010-437X},
	journal = {Compositio Math.},
	mrclass = {11G18 (11G15 14H25 14K22)},
	mrnumber = {860139},
	mrreviewer = {Ernst-Ulrich Gekeler},
	number = {2},
	pages = {151--230},
	title = {Sur la mauvaise r\'{e}duction des courbes de {S}himura},
	url = {https://mathscinet.ams.org/mathscinet-getitem?mr=860139},
	volume = {59},
	year = {1986}}

@incollection{Coleman-Shimura,
	author = {Coleman, Robert F.},
	booktitle = {{$p$}-adic monodromy and the {B}irch and {S}winnerton-{D}yer conjecture ({B}oston, {MA}, 1991)},
	doi = {10.1090/conm/165/01602},
	mrclass = {11F85 (11F67 11G18)},
	mrnumber = {1279600},
	mrreviewer = {Andrea Mori},
	pages = {21--51},
	publisher = {Amer. Math. Soc., Providence, RI},
	series = {Contemp. Math.},
	title = {A {$p$}-adic {S}himura isomorphism and {$p$}-adic periods of modular forms},
	url = {https://mathscinet.ams.org/mathscinet-getitem?mr=1279600},
	volume = {165},
	year = {1994}}

@article{Colemanhigher,
	author = {Coleman, Robert F.},
	fjournal = {Journal de Th\'{e}orie des Nombres de Bordeaux},
	issn = {1246-7405},
	journal = {J. Th\'{e}or. Nombres Bordeaux},
	mrclass = {11F85 (11F11)},
	mrnumber = {1617406},
	mrreviewer = {Conjeeveram S. Rajan},
	number = {2},
	pages = {395--403},
	title = {Classical and overconvergent modular forms of higher level},
	url = {http://jtnb.cedram.org/item?id=JTNB_1997__9_2_395_0},
	volume = {9},
	year = {1997}}

@article{W,
	author = {Wan, Xin},
	journal = {preprint https://arxiv.org/pdf/2304.09806.pdf},
	title = {A {N}ew $\pm$ {I}wasawa {T}heory and {C}onverse of {G}ross-{Z}agier and {K}olyvagin {T}heorem (with an {A}ppendix by {Y}angyu {F}an)},
	year = {2023}}

@article{CW,
	author = {Castella, Francesc and Wan, Xin},
	doi = {10.1016/j.aim.2022.108266},
	fjournal = {Advances in Mathematics},
	issn = {0001-8708},
	journal = {Adv. Math.},
	mrclass = {11R23 (11G05 11G40)},
	mrnumber = {4387238},
	mrreviewer = {K\^{a}z\i m B\"{u}y\"{u}kboduk},
	pages = {Paper No. 108266, 45},
	title = {The {I}wasawa main conjectures for {$\rm GL_2$} and derivatives of {$p$}-adic {$L$}-functions},
	url = {https://mathscinet.ams.org/mathscinet-getitem?mr=4387238},
	volume = {400},
	year = {2022}}

@book{Kato1,
	author = {Kato, Kazuya},
	journal = {LNM 1553},
	publisher = {Springer-Verlag},
	title = {Lectures on the approach to Iwasawa Theory for the Hasse-Weil $L$-function via $\mathbf{B}_{\mathrm{dR}}$.},
	year = {1991}}

@book{BrCo,
	author = {Brinon, Olivier and Conrad, Brian},
	publisher = {preprint},
	title = {{CMI} summer school notes on $p$-adic {H}odge theory},
	url = {https://math.stanford.edu/~conrad/papers/notes.pdf},
	year = {2009}}

@article{Scholze3,
	author = {Scholze, Peter},
	doi = {10.4007/annals.2015.182.3.3},
	fjournal = {Annals of Mathematics. Second Series},
	issn = {0003-486X},
	journal = {Ann. of Math. (2)},
	mrclass = {11S37},
	mrnumber = {3418533},
	mrreviewer = {Kimball L. Martin},
	number = {3},
	pages = {945--1066},
	title = {On torsion in the cohomology of locally symmetric varieties},
	url = {https://doi.org/10.4007/annals.2015.182.3.3},
	volume = {182},
	year = {2015}}

@article{BL-Coleman,
	author = {B\"{u}y\"{u}kboduk, K\^{a}zim and Lei, Antonio},
	doi = {10.1112/jlms.12471},
	fjournal = {Journal of the London Mathematical Society. Second Series},
	issn = {0024-6107},
	journal = {J. Lond. Math. Soc. (2)},
	mrclass = {11R23 (11F11)},
	mrnumber = {4339947},
	number = {4},
	pages = {1682--1716},
	title = {Interpolation of generalized {H}eegner cycles in {C}oleman families},
	url = {https://doi.org/10.1112/jlms.12471},
	volume = {104},
	year = {2021}}

@article{JLZ,
	author = {Jetchev, Dimitar and Loeffler, David and Zerbes, Sarah Livia},
	doi = {10.1112/plms.12363},
	fjournal = {Proceedings of the London Mathematical Society. Third Series},
	issn = {0024-6115},
	journal = {Proc. Lond. Math. Soc. (3)},
	mrclass = {11F67 (11F80 11F85)},
	mrnumber = {4210260},
	mrreviewer = {Nils Matthes},
	number = {1},
	pages = {124--152},
	title = {Heegner points in {C}oleman families},
	url = {https://doi.org/10.1112/plms.12363},
	volume = {122},
	year = {2021}}

@article{LZ-IwasawaZ2,
	author = {Loeffler, David and Zerbes, Sarah Livia},
	doi = {10.1142/S1793042114500699},
	fjournal = {International Journal of Number Theory},
	issn = {1793-0421},
	journal = {Int. J. Number Theory},
	mrclass = {11R23 (11F80 11G40 11S40)},
	mrnumber = {3273476},
	mrreviewer = {Andreas Nickel},
	number = {8},
	pages = {2045--2095},
	title = {Iwasawa theory and {$p$}-adic {$L$}-functions over {$\Bbb{Z}_p^2$}-extensions},
	url = {https://doi.org/10.1142/S1793042114500699},
	volume = {10},
	year = {2014}}

@article{Ochiai-coleman,
	author = {Ochiai, Tadashi},
	fjournal = {American Journal of Mathematics},
	issn = {0002-9327},
	journal = {Amer. J. Math.},
	mrclass = {11F80 (11R23)},
	mrnumber = {1993743},
	mrreviewer = {Jacques Tilouine},
	number = {4},
	pages = {849--892},
	title = {A generalization of the {C}oleman map for {H}ida deformations},
	url = {http://muse.jhu.edu/journals/american_journal_of_mathematics/v125/125.4ochiai.pdf},
	volume = {125},
	year = {2003}}

@article{BCK,
	author = {Burungale, Ashay and Castella, Francesc and Kim, Chan-Ho},
	doi = {10.2140/ant.2021.15.1627},
	fjournal = {Algebra \& Number Theory},
	issn = {1937-0652},
	journal = {Algebra Number Theory},
	mrclass = {11R23 (11F33)},
	mrnumber = {4333660},
	number = {7},
	pages = {1627--1653},
	title = {A proof of {P}errin-{R}iou's {H}eegner point main conjecture},
	url = {https://doi.org/10.2140/ant.2021.15.1627},
	volume = {15},
	year = {2021}}

@article{Brasca-Eigen2,
	author = {Brasca, Riccardo},
	journal = {Compos. Math.},
	number = {1},
	pages = {32--62},
	title = {{$p$}-adic modular forms of non-integral weight over {S}himura curves},
	volume = {149},
	year = {2013}}

@article{Morita,
	author = {Morita, Yasuo},
	journal = {Hokkaido Math. J.},
	number = {2},
	pages = {209--238},
	title = {Reduction modulo {${\mathfrak P}$} of {S}himura curves},
	volume = {10},
	year = {1981}}

@article{Shnidman,
	author = {Shnidman, Ariel},
	journal = {Ann. Inst. Fourier (Grenoble)},
	number = {3},
	pages = {1117--1174},
	title = {{$p$}-adic heights of generalized {H}eegner cycles},
	volume = {66},
	year = {2016}}

@article{Chenevier,
	author = {Chenevier, Ga\"{e}tan},
	journal = {Duke Math. J.},
	number = {1},
	pages = {161--194},
	title = {Une correspondance de {J}acquet-{L}anglands {$p$}-adique},
	volume = {126},
	year = {2005}}

@article{Burungale,
	author = {Burungale, Ashay A.},
	journal = {J. Inst. Math. Jussieu},
	number = {1},
	pages = {189--222},
	title = {On the non-triviality of the {$p$}-adic {A}bel-{J}acobi image of generalised {H}eegner cycles modulo {$p$}, {II}: {S}himura curves},
	volume = {16},
	year = {2017}}

@article{Brako,
	author = {Brako\v{c}evi\'{c}, Miljan},
	journal = {Int. Math. Res. Not. IMRN},
	number = {21},
	pages = {4967--5018},
	title = {Anticyclotomic {$p$}-adic {$L$}-function of central critical {R}ankin-{S}elberg {$L$}-value{R}ankin-{S}elberg {$L$}-value},
	year = {2011}}

@article{Kassaei,
	author = {Kassaei, Payman L.},
	journal = {Compos. Math.},
	number = {2},
	pages = {359--395},
	title = {{$\mathscr P$}-adic modular forms over {S}himura curves over totally real fields},
	volume = {140},
	year = {2004}}

@article{Brooks,
	author = {Hunter Brooks, Ernest},
	journal = {Int. Math. Res. Not. IMRN},
	number = {12},
	pages = {4177--4241},
	title = {Shimura curves and special values of {$p$}-adic {$L$}-functions},
	year = {2015}}

@article{Magrone,
	author = {Magrone, Paola},
	fjournal = {Annali della Scuola Normale Superiore di Pisa. Classe di Scienze. Serie V},
	issn = {0391-173X,2036-2145},
	journal = {Ann. Sc. Norm. Super. Pisa Cl. Sci. (5)},
	mrclass = {11F11 (11F80 11G40 14G35)},
	mrnumber = {4553539},
	mrreviewer = {Matteo\ Longo},
	number = {4},
	pages = {1807--1870},
	title = {Generalized {H}eegner cycles and {$p$}-adic {$L$}-functions in a quaternionic setting},
	volume = {23},
	year = {2022}}

@book{hida-elementary,
	author = {Hida, Haruzo},
	pages = {xii+386},
	publisher = {Cambridge University Press, Cambridge},
	series = {London Mathematical Society Student Texts},
	title = {Elementary theory of {$L$}-functions and {E}isenstein series},
	volume = {26},
	year = {1993}}

@article{Castella-MathAnn,
	author = {Castella, Francesc},
	doi = {10.1007/s00208-012-0871-4},
	fjournal = {Mathematische Annalen},
	issn = {0025-5831},
	journal = {Math. Ann.},
	number = {4},
	pages = {1247--1282},
	title = {Heegner cycles and higher weight specializations of big {H}eegner points},
	url = {https://doi.org/10.1007/s00208-012-0871-4},
	volume = {356},
	year = {2013}}

@article{Castella,
	author = {Castella, Francesc},
	doi = {10.1017/s1474748019000094},
	fjournal = {Journal of the Institute of Mathematics of Jussieu. JIMJ. Journal de l'Institut de Math\'{e}matiques de Jussieu},
	issn = {1474-7480},
	journal = {J. Inst. Math. Jussieu},
	mrclass = {11G40 (11G05 11R23)},
	mrnumber = {4167004},
	number = {6},
	pages = {2127--2164},
	title = {On the {$p$}-adic variation of {H}eegner points},
	url = {https://doi.org/10.1017/s1474748019000094},
	volume = {19},
	year = {2020}}

@article{DT,
	author = {Diamond, Fred and Taylor, Richard},
	journal = {Invent. Math.},
	number = {3},
	pages = {435-462},
	title = {Non-optimal levels of mod l modular representations.},
	url = {http://eudml.org/doc/144177},
	volume = {115},
	year = {1994}}

@article{CL,
	author = {Castella, Francesc and Longo, Matteo},
	fjournal = {Annales Math\'ematiques du Qu\'ebec},
	journal = {Ann. Math. Qu\'e.},
	number = {2},
	pages = {303--324},
	title = {Big {H}eegner points and special values of {$L$}-series},
	volume = {40},
	year = {2016}}

@article{LV-Pisa,
	author = {Longo, Matteo and Vigni, Stefano},
	fjournal = {Annali della Scuola Normale Superiore di Pisa. Classe di Scienze. Serie V},
	journal = {Ann. Sc. Norm. Super. Pisa Cl. Sci. (5)},
	number = {3},
	pages = {859--888},
	title = {Vanishing of special values and central derivatives in {H}ida families},
	volume = {13},
	year = {2014}}

@article{LV-MM,
	author = {Longo, Matteo and Vigni, Stefano},
	fjournal = {Manuscripta Mathematica},
	journal = {Manuscripta Math.},
	number = {3-4},
	pages = {273--328},
	title = {Quaternion algebras, {H}eegner points and the arithmetic of {H}ida families},
	volume = {135},
	year = {2011}}

@incollection{BK,
	address = {Boston, MA},
	author = {Bloch, S. and Kato, Kazuya},
	booktitle = {The {G}rothendieck {F}estschrift, {V}ol.\ {I}},
	pages = {333--400},
	publisher = {Birkh\"auser Boston},
	series = {Progr. Math.},
	title = {{$L$}-functions and {T}amagawa numbers of motives},
	volume = {86},
	year = {1990}}

@article{PR,
	author = {Perrin-Riou, B.},
	fjournal = {Bulletin de la Soci\'et\'e Math\'ematique de France},
	journal = {Bull. Soc. Math. France},
	number = {4},
	pages = {399--456},
	title = {Fonctions {$L$} {$p$}-adiques, th\'eorie d'{I}wasawa et points de {H}eegner},
	volume = {115},
	year = {1987}}

@article{Fouquet,
	author = {Fouquet, Olivier},
	fjournal = {Compositio Mathematica},
	journal = {Compos. Math.},
	number = {3},
	pages = {356--416},
	title = {Dihedral {I}wasawa theory of nearly ordinary quaternionic automorphic forms},
	volume = {149},
	year = {2013}}

@article{CH,
	author = {Castella, Francesc and Hsieh, Ming-Lun},
	doi = {10.1007/s00208-017-1517-3},
	fjournal = {Mathematische Annalen},
	issn = {0025-5831},
	journal = {Math. Ann.},
	mrclass = {11G40 (11F67 11F70 11F85 11S40 11S80)},
	mrnumber = {3747496},
	mrreviewer = {Daniel Barsky},
	number = {1-2},
	pages = {567--628},
	title = {Heegner cycles and {$p$}-adic {$L$}-functions},
	url = {https://doi.org/10.1007/s00208-017-1517-3},
	volume = {370},
	year = {2018}}

@article{BDP,
	author = {Bertolini, Massimo and Darmon, Henri and Prasanna, Kartik},
	fjournal = {Duke Mathematical Journal},
	journal = {Duke Math. J.},
	number = {6},
	pages = {1033--1148},
	title = {Generalized {H}eegner cycles and {$p$}-adic {R}ankin {$L$}-series},
	volume = {162},
	year = {2013}}

@article{Howard-Inv,
	author = {Howard, Benjamin},
	fjournal = {Inventiones Mathematicae},
	journal = {Invent. Math.},
	number = {1},
	pages = {91--128},
	title = {Variation of {H}eegner points in {H}ida families},
	volume = {167},
	year = {2007}}

@article{ChHs2,
	author = {Chida, Masataka and Hsieh, Ming-Lun},
	fjournal = {Compositio Mathematica},
	journal = {Compos. Math.},
	number = {5},
	pages = {863--897},
	title = {On the anticyclotomic {I}wasawa main conjecture for modular forms},
	volume = {151},
	year = {2015}}

@incollection{Kat,
	author = {Katz, Nicholas},
	booktitle = {Algebraic surfaces (Orsay, 1976-78)},
	pages = {138--202},
	publisher = {Springer, Berlin},
	series = {Lecture Notes in Math.},
	title = {{S}erre--{T}ate local moduli},
	volume = {868},
	year = {1981}}

@article{Buz,
	author = {Buzzard, Kevin},
	fjournal = {Duke Mathematical Journal},
	journal = {Duke Math. J.},
	number = {3},
	pages = {591--612},
	title = {Integral models of certain Shimura Curves},
	volume = {87},
	year = {1997}}

@article{KLZ.ERL,
	author = {Kings, Guido and Loeffler, David and Zerbes, Sarah Livia},
	doi = {10.4310/CJM.2017.v5.n1.a1},
	fjournal = {Cambridge Journal of Mathematics},
	issn = {2168-0930},
	journal = {Camb. J. Math.},
	mrclass = {11F85 (11F67 11G40 14G35 14H52)},
	mrnumber = {3637653},
	mrreviewer = {Ivan Mati\'{c}},
	number = {1},
	pages = {1--122},
	title = {Rankin-{E}isenstein classes and explicit reciprocity laws},
	url = {https://doi.org/10.4310/CJM.2017.v5.n1.a1},
	volume = {5},
	year = {2017}}
\end{document}